\documentclass[12pt]{amsart}
\usepackage{amsmath,amssymb}
\usepackage[new]{old-arrows}
\usepackage{syntonly}
\usepackage{eufrak}
\usepackage{amsthm}
\usepackage{enumitem}
\usepackage{setspace}
\usepackage{graphicx}
\usepackage{graphics}
\usepackage{color}
\usepackage{xcolor}
\usepackage{hyperref}
\usepackage{youngtab}
\usepackage{marginnote}
\usepackage{lscape}
\usepackage{faktor}
\usepackage{hyperref}
\usepackage[normalem]{ulem}
\usepackage{bm}
\usepackage{afterpage}
\usepackage{mathrsfs}
\usepackage[all]{xy}
\xyoption{matrix}
\xyoption{arrow}
\usepackage{tikz}
\usetikzlibrary{calc,shadings,patterns}
\usetikzlibrary{matrix}
\usetikzlibrary{arrows}
\usetikzlibrary{matrix}
\usepackage{verbatim}
\usepackage{multirow}
\usetikzlibrary{decorations.pathreplacing}
\usepackage{enumitem,pifont,xcolor}
\definecolor{myblue}{RGB}{0,29,119}

\newtheorem{theorem}{Theorem}[section]
\newtheorem{proposition}[theorem]{Proposition}
\newtheorem{corollary}[theorem]{Corollary}
\newtheorem{lemma}[theorem]{Lemma}
\theoremstyle{definition}
\newtheorem{definition}[theorem]{Definition}
\newtheorem{example}[theorem]{Example}
\newtheorem{conjecture}[theorem]{Conjecture}
\newtheorem{remark}[theorem]{Remark}

\newtheorem{notation}[theorem]{Notation}
\newtheorem{claim}[theorem]{Claim}
\newtheorem{problem}[theorem]{Problem}

\newtheorem*{theorem*}{Theorem}

\makeatletter
\@addtoreset{equation}{section}
\makeatother

\newcommand{\Zz}{{\mathbb Z}} 
\newcommand{\Aa}{{\mathbb A}}
\newcommand{\Kk}{{\mathbb K}}

\DeclareMathOperator{\End}{End}
\DeclareMathOperator{\coker}{coker}

\DeclareMathOperator{\findim}{fin.\! dim}

\DeclareMathOperator{\gldim}{gl\,\! dim}
\DeclareMathOperator{\domdim}{dom.\! dim}

\DeclareMathOperator{\Hom}{Hom}
\DeclareMathOperator{\rad}{rad}

\DeclareMathOperator{\pdim}{p\!\dim}
\DeclareMathOperator{\defect}{def}
\DeclareMathOperator{\rel}{rel}
\DeclareMathOperator{\soc}{soc}
\DeclareMathOperator{\topp}{top}
\DeclareMathOperator{\del}{del}
\DeclareMathOperator{\rank}{rank}

\newcommand{\cB}{{\mathcal B}}

\newcommand{\cL}{{\mathcal L}}

\newcommand{\cP}{{\mathcal P}}

\newcommand{\cS}{{\mathcal S}}

\providecommand{\AMS}{$\mathcal{A}$\kern-.1667em%
\lower.25em\hbox{$\mathcal{M}$}\kern-.125em$\mathcal{S}$}

\newtheorem{innercustomthm}{Theorem}
\newenvironment{customthm}[1]
  {\renewcommand\theinnercustomthm{#1}\innercustomthm}{\endinnercustomthm}

\begin{document}


\author{Emre SEN}

\email{emresen641@gmail.com}
\title{Nakayama Algebras which are Higher Auslander Algebras}
\begin{abstract} 
We prove that any cyclic Nakayama algebra which is a higher Auslander algebra can be uniquely constructed from Nakayama algebras of smaller ranks by reversing the syzygy filtration process. This creates chains of higher Auslander algebras upto ${\bm\varepsilon}$-equivalences. Therefore, the classification of all cyclic Nakayama algebras which are higher Auslander algebras reduces to the classification of linear ones.  We give two applications of this: for any integer $k$ where $2\leq k\leq 2n-2$, there is  a Nakayama algebra of rank $n$ which is a higher Auslander algebra of global dimension $k$ and the possible values of the global dimensions of cyclic Nakayama algebras which are higher Auslander algebras form the sets $\left\{2,\ldots,2n-2\right\}\setminus\left\{n-1\right\}$ if $n$ is even and $\left\{2,\ldots,2n-2\right\}\setminus\left\{ 2,n-1\right\}$ if $n$ is odd.
\end{abstract}

\maketitle

{\let\thefootnote\relax}

\section{Introduction}
O. Iyama intoduced higher Auslander algebras in \cite{iyama} which corresponds bijectively to endomorphism algebras of cluster-tilting modules. A finite dimensional algebra $A$ over a field $\mathbb{K}$ is called a higher Auslander algebra if the dominant dimension of $A$ is equal to the global dimension of $A$. This generalizes classical Auslander algebras where \begin{align*}
\gldim A\leq 2\leq\domdim A.
\end{align*}

In general, it is not easy to find algebras of given arbitrary global and dominant dimensions. This is true even for algebras with fairly well understood module categories. In the last section of the work \cite{rene} the authors introduce the spectrum of a fixed quiver $Q$ over a field $\Kk$ as the set of all possible values of global dimensions of $\Kk Q/I$ where $I$ is an admissible ideal provided that $\Kk Q/I$ is a higher Auslander algebra i.e
\begin{align}\label{defspec}
\zeta(Q)=\left\{\gldim(\Kk Q/I)\,\,\vert\,\,KQ/I\,\,\text{is a higher Auslander algebra for ideal }I \right\}.
\end{align}

Then they state the problem:
\begin{problem}\label{problem} Let $Q$ be oriented cyclic quiver of rank $n$. What is the spectrum $\zeta(Q)$?
\end{problem}

In their own words: "In general it seems very complicated to answer the previous question for a given n" [pg.803]\cite{rene}. Notice that the algebras which are the subject of the problem are cyclic Nakayama algebras i.e. indecomposable modules are uniserial. They posed the following conjecture, and verified it up to $n\leq 14$ by using computers.

\begin{conjecture}\label{conj} For every $n\geq 2$ and $k$ with $2\leq k\leq 2n-2$ there exists a connected, linear or cyclic Nakayama algebra with $n$ simple modules that is a higher Auslander algebra with global dimension $k$.
\end{conjecture}

Here we develop new tools based on the syzygy filtration method i.e. $\bm\varepsilon$-construction, and give complete solutions to the problem \ref{problem} and the conjecture. In particular this shows the efficiency of the ${\bm\varepsilon}$-construction also as pointed out by C.M. Ringel (appendices B and C \cite{rin2}). Briefly, the core idea of the syzygy filtration method is constructing the algebra $\bm\varepsilon(\Lambda)$ whose modules are equivalent to modules filtered by the \emph{second} syzygies of the original algebra $\Lambda$ and then repeating this process recursively. In \ref{sectionprelim}, definition and the properties of the method are stated.
 
  Results concerning the problem and the conjecture are:

\begin{theorem}\label{thmsolutiontoproblem}[Solution to the problem \ref{problem}] Let $\Lambda$ be a cyclic Nakayama algebra over quiver $Q$ which is higher Auslander algebra of rank $n$, then the spectrum of $Q$ is
\begin{enumerate}[label=\roman*)]
\item  $\zeta(Q)=\left\{2,3,\ldots,2n-2\right\}\setminus\left\{2,n-1\right\}$ if $n$ is odd,
\item $\zeta(Q)=\left\{2,3,\ldots,2n-2\right\}\setminus \left\{n-1\right\}$ if $n$ is even.
\end{enumerate}
\end{theorem}

\begin{theorem}\label{thmconj}
Conjecture \ref{conj} holds, for every pair of integers $n,k$ with $2\leq k\leq 2n-2$, there exists Nakayama algebra with $n$ simple modules that is a higher Auslander algebra with global dimension $k$.
\end{theorem}

To state our main tools, we describe the \emph{defect} of an algebra here which is the number of injective but non-projective modules \cite{rene}. Actually, it is the complementary notion to the number of relations defining a cyclic Nakayama algebra (see \cite{sen18}) because
\begin{align*}
\rank\Lambda=\#\rel\Lambda+\defect\Lambda
\end{align*}
where $\rank\Lambda$ is the number of nonisomorphic simple $\Lambda$-modules, $\#\rel\Lambda$ is the minimal number of relations defining the algebra $\Lambda$ and $\defect\Lambda$ denotes the defect of the $\Lambda$. For semisimple algebras $\Aa^n_1$ we set the defect to $n$. In particular the defect is additive, in other words if $\Lambda$ is a linear Nakayama algebra then the defect of $\Lambda$ is the sum of defects of the connected components  (see definition \ref{defect of linear}).

Under $\bm\varepsilon$-construction (Definition \ref{defep}), the defect of algebras might change. However, there is an exception for higher Auslander algebras.
\begin{customthm}{(A)} \label{thmLamdaauslanderimplies} If $\Lambda$ is higher Auslander algebra which is cyclic, then ${\bm\varepsilon}(\Lambda)$ is also a higher Auslander algebra which is either cyclic or linear (not necessarily connected) provided that global dimension of $\Lambda$ is greater or equal than three. Moreover their defects are equal, i.e. $\defect\Lambda=\defect\bm\varepsilon(\Lambda)$ .
\end{customthm}


Here is the first of our key tools, which can be viewed as reversing construction of syzygy filtered algebra (see def.\ref{defigherNakayama}) process.

\begin{customthm}{(B)}\label{thmreverseepsilon}
 If $\Lambda$ is a \textbf{cyclic} Nakayama algebra of rank $n$ which is a higher Auslander algebra of global dimension $k$, then there exists a \textbf{unique} cyclic Nakayama algebra $\Lambda'$ of rank $n+\defect\Lambda$ which is a higher Auslander algebra of global dimension $k+2$ and ${\bm\varepsilon}(\Lambda')\cong\Lambda$.
\end{customthm}

\begin{customthm}{(C)}\label{thmreverseepsilon2}
$(i)$ If $\Lambda$ is a \textbf{connected linear} Nakayama algebra of rank $n$ which is a higher Auslander algebra of global dimension $k$, then there exists a \textbf{unique} cyclic Nakayama algebra $\Lambda'$ of rank $n+\defect\Lambda$ which is a higher Auslander algebra of global dimension $k+2$ and ${\bm\varepsilon}(\Lambda')\cong\Lambda$.\\
\par $(ii)$ Let $(\Lambda,\tau)$ denote the algebra $\Lambda =\Lambda_1\times ... \times \Lambda_t$ with connected linear Nakayama algebras $\Lambda_1,\ldots,\Lambda_t$ and the cyclic permutation $\tau$ of the simple $\Lambda$-modules, such that the restriction to the simple $\Lambda_i$-modules is the Auslander-Reiten
translation for the simple $\Lambda_i$-modules, and $\tau$ maps the simple projective  $\Lambda_i$-module to the simple injective $\Lambda_{i-1}$-module (with $\Lambda_0 =\Lambda_t$). If $(\Lambda,\tau)$ is a higher Auslander algebra of global dimension $k$ and rank $n$, then there exists a \textbf{unique} \textbf{cyclic} connected Nakayama algebra $\Lambda'$ of rank $n+\defect\Lambda$ which is a higher Auslander algebra of global dimension $k+2$ and ${\bm\varepsilon}(\Lambda')\cong\Lambda$.
\end{customthm}
The uniqueness part of the above result  depends on the permutation $\tau$, since different cyclic permutations of linear Nakayama algebras can produce various nonisomorphic higher Auslander algebras sharing the same rank and global dimension. Indeed, we show that they are enumerated by \emph{necklace} numbers which we discuss in \ref{necklace}.\\

The strength of the syzygy filtration method together with Theorems \ref{thmLamdaauslanderimplies},\ref{thmreverseepsilon}, \ref{thmreverseepsilon2} is hidden in its recursive nature as illustrated below.

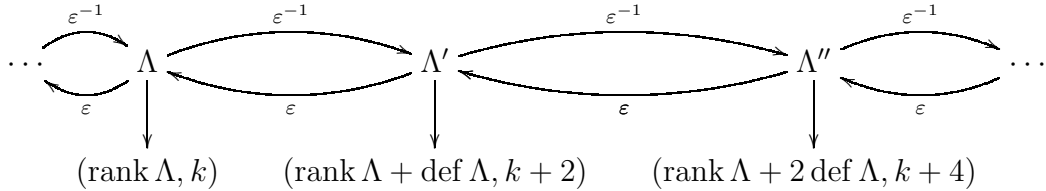
\begin{center}
$\xymatrixcolsep{5pt}\xymatrix{\cdots\ar@/^1pc/[r]^{\varepsilon^{-1}}&\Lambda \ar@/^1pc/[rr]^{\varepsilon^{-1}}\ar@/^1pc/[l]^{\varepsilon}\ar[d] && \Lambda'\ar@/^1pc/[rr]^{\varepsilon^{-1}}\ar@/^1pc/[ll]^{\varepsilon}\ar[d] && \Lambda''\ar@/^1pc/[r]^{\varepsilon^{-1}}\ar@/^1pc/[ll]^{\varepsilon}\ar@/^1pc/[ll]^{\varepsilon}\ar[d]&\cdots\ar@/^1pc/[l]^{\varepsilon} \\
&(\rank \Lambda,k)&&(\rank \Lambda+\defect\Lambda,k+2)&&(\rank\Lambda+2\defect\Lambda,k+4)}$
\end{center}

The strategy to solve the problem \ref{problem} is to detect small rank and small global dimensional Nakayama algebras which are higher Auslander algebras and then generate a chain of algebras up to ${\bm\varepsilon}$-equivalences by using the invariance of the defect. It is quite unexpected that classification of cyclic Nakayama algebras which are higher Auslander reduces to the classification of \emph{linear} Nakayama algebras which are higher Auslander.\\

We state another tool based on the syzygy filtration method.
\begin{customthm}{(D)}\label{thmdoubling}
Let $\Lambda$ be a cyclic Nakayama algebra of rank $n$ and $(c_1,\ldots,c_n)$ be its Kupisch series. If $\Lambda'$ is double covering\footnote{In the earlier version we used the term repetition. We thank C.M. Ringel for the correct terminology.} of $\Lambda$ i.e. $c_i=c_{i+n}$ for all $i$ $1\leq i\leq n$  where $c_i$ is the length of projective $\Lambda'$-module $P_i$, then $\bm\varepsilon(\Lambda')$ is double covering of $\bm\varepsilon(\Lambda)$. 
\end{customthm}

This is a very useful result. If we choose $\Lambda$ as a higher Auslander algebra, then coverings of $\Lambda$ are also higher Auslander algebra (see corollary \ref{corAus}). Therefore, we can create families of algebras sharing the same global dimension simultaneously. In particular, this result shows how ${\bm\varepsilon}$-construction preserves higher Auslander algebra structure in contrast to other reduction methods (see Example \ref{example method}).

We will give a constructive proof of Theorem \ref{thmsolutiontoproblem}. Organization of the paper is the following: in \ref{sectionprelim} we recall basics and some properties of the syzygy filtration method. In section \ref{sectionmaintools}, we present proofs and some outcomes of theorems \ref{thmLamdaauslanderimplies}, \ref{thmdoubling}. Proofs of Theorems \ref{thmreverseepsilon} and \ref{thmreverseepsilon2} are given in the sections \ref{section cyclic} and \ref{section linear}. After obtaining small rank and  global dimensional linear Nakayama algebras which are higher Auslander algebras in \ref{sectionCalculation}, we apply our main tools recursively to complete the proof of Theorem \ref{thmsolutiontoproblem}. The last part of \ref{sectionCalculation} contains some remarks and examples which are necessary to complete the solution to the conjecture \ref{thmconj}. There is also appendix \ref{sectionappendix} where we list different families of algebras with odd global dimension.\\

We are grateful to Gordana Todorov for all her help and support. We are thankful to Claus Michael Ringel for his interest in the syzygy filtration method, stating its difference from other methods in \cite{rin2}, informing us how it fits into simplification process nicely \cite{rin1}   together with suggestions to improve the text and fixing a bug in the earlier versions (see def. \ref{nakayamacycle} and thm \ref{thmreverseepsilon}), to Shijie Zhu for various discussions on the subject which leads to the joint work \cite{stz}, to Rene Marczinzik for suggesting their conjecture as a possible application of syzygy filtration method, as well as providing some GAP computations, to Laertis Vaso for notifying some mistakes in the earlier versions and to the referee for very helpful, valuable suggestions and comments to improve the paper.

\subsection{Preliminaries on Syzygy Filtration}\label{sectionprelim}
Here we briefly recall some basics about Nakayama algebras. For details we refer to \cite{sen18} and \cite{sen19}. A finite dimensional bound quiver algebra is called Nakayama algebra if all indecomposable modules are uniserial. There are two types of Nakayama algebras depending on the underlying quiver: linear or cyclic. 
A cyclic Nakayama algebra of rank $n$ can be described by the irredundant system of relations $\boldsymbol\alpha_{k_{2i}}\ldots\boldsymbol\alpha_{k_{2i-1}}=0$ where $1\leq i\leq  r$ and $k_{f}\in\left\{1,2,\ldots,n\right\}$ where each arrow $\boldsymbol\alpha_i$, $1\leq i\leq n-1$ starts at the vertex $i$ and ends at the vertex $i+1$ and $\boldsymbol\alpha_n$ starts at vertex $n$ and ends at vertex $1$. For instance, if there are $5$ vertices i.e. rank is $5$ and the relations are $\boldsymbol\alpha_3\boldsymbol\alpha_2\boldsymbol\alpha_1=0$, $\boldsymbol\alpha_5\boldsymbol\alpha_4\boldsymbol\alpha_3=0$ and $\boldsymbol\alpha_1\boldsymbol\alpha_5=0$ then the indecomposable projective modules are
\begin{align}
P_1=\begin{vmatrix}
S_1\\S_2\\S_3
\end{vmatrix}, P_2=\begin{vmatrix}
S_2\\S_3\\S_4\\S_5
\end{vmatrix}, P_3=\begin{vmatrix}
S_3\\S_4\\S_5
\end{vmatrix}, P_4=\begin{vmatrix}
S_4\\S_5\\S_1
\end{vmatrix}, P_5=\begin{vmatrix}
S_5\\S_1
\end{vmatrix}
\end{align}
where $S_i$, $1\leq i\leq 5$ stands for simple module at the vertex $i$.

We assume that the algebra $\Lambda$ is given as the path algebra of the cyclic oriented quiver of rank $n$ modulo the irredundant system of relations:
\begin{gather}
\label{relations}
\begin{gathered}
\boldsymbol\alpha_{k_2}\ldots\boldsymbol\alpha_{k_1+1}\boldsymbol\alpha_{k_1}\ \ =0 \\
\boldsymbol\alpha_{k_4}\ldots\boldsymbol\alpha_{k_3+1}\boldsymbol\alpha_{k_3}\ \ =0 \\
\dots \\
\boldsymbol\alpha_{k_{2r-2}}\ldots\boldsymbol\alpha_{k_{2r-3}+1}\boldsymbol\alpha_{k_{2r-3}}=0\\
\boldsymbol\alpha_{k_{2r}}\ldots\boldsymbol\alpha_{k_{2r-1}+1}\boldsymbol\alpha_{k_{2r-1}}=0
\end{gathered}
\end{gather}

Using the above system of relations, we describe indecomposable projective and injective modules. Since the system is irredundant, projective-injective modules will occur as projective covers of the simples labeled by the index next to the first index of each relation. Explicitly they are:
\begin{align}
P_{k_1+1}=I_{k_4},\quad P_{k_3+1}=I_{k_6},\quad \dots \quad P_{k_{2r-3}+1}=I_{k_{2r}},\qquad P_{k_{2r-1}+1}=I_{k_{2}}
\end{align}
Furthermore, we get \textbf{classes of projective modules} characterized by their socles:
\begin{align}
\label{projectiveclasses}
\begin{split}
P_{k_{1}}\hookrightarrow\ldots\hookrightarrow  P_{(k_{2r-1})+1}=I_{k_{2}} \quad\text{ have simple } S_{k_{2}} \text{  as their socle}\\
P_{k_3}\hookrightarrow\ldots\hookrightarrow P_{k_1+1}=I_{k_4}  \quad\text{ have simple } S_{k_4} \text{  as their socle}\\
\dots\qquad\qquad\qquad\qquad\qquad\qquad\qquad\\
P_{k_{2r-1}}\hookrightarrow\ldots\hookrightarrow P_{(k_{2r-3})+1}=I_{k_{2r}}\quad  \text{ have simple } S_{k_{2r}} \text{  as their socle}
\end{split}
\end{align}
Similarly, we get \textbf{classes of injective modules} characterized by their tops:
\begin{align}\label{injective class}
\begin{split}
 P_{(k_{2r-1})+1}=I_{k_{2}}\twoheadrightarrow\ldots\twoheadrightarrow I_{k_4+1}  \quad\text{ have simple } S_{k_{2r-1}+1}\text{  as their top}\\
P_{k_1+1}=I_{k_4}\twoheadrightarrow\ldots\twoheadrightarrow I_{k_2+1}   \quad\text{ have simple } S_{k_1+1} \text{  as their top}\\
\dots\qquad\qquad\qquad\qquad\quad\quad\quad\quad\quad\quad\\
P_{(k_{2r-3})+1}=I_{k_{2r}} \twoheadrightarrow\ldots\twoheadrightarrow I_{k_4+1}  \quad\quad\text{ have simple } S_{k_{2r-3}+1} \text{  as their top}
\end{split}
\end{align}

\begin{definition}\cite{sen18},\cite{sen19}\label{defbase}
$\cS(\Lambda)$ is the complete set of representatives of socles of indecomposable projective modules over $\Lambda$ i.e. $\cS(\Lambda)=\left\{S_{k_2}, S_{k_4},\ldots,S_{k_{2r}}\right\}$. $\cS'(\Lambda)$ is the complete set of representatives of simple modules such that they are indexed by one cyclically larger (see definition 2.3.2 \cite{sen18}) indices of $\cS(\Lambda)$ i.e. $\cS'(\Lambda)=\left\{S_{k_{2}+1}, S_{k_4+1},\ldots,S_{k_{2r}+1}\right\}$. In other words, those are Auslander-Reiten translates of elements of $\cS(\Lambda)$. The base set $\cB(\Lambda)$ of $\Lambda$ is given by:
\begin{align}\label{baseset}
\cB(\Lambda):=\left\{ \Delta_1\cong\begin{vmatrix}
    S_{k_{2r}+1} \\
    \vdots  \\
    S_{k_{2}}
\end{vmatrix}\!, \Delta_2\cong\begin{vmatrix}
    S_{k_{2}+1}  \\
    \vdots  \\
   S_{k_{4}}
\end{vmatrix}\!,\ldots,\Delta_r\cong \begin{vmatrix}
   S_{k_{2r-2}+1}  \\
    \vdots  \\
    S_{k_{2r}}
\end{vmatrix}\!\right\}
\end{align}
\end{definition}

\begin{definition}\label{defep} \cite{sen19}
For any cyclic Nakayama algebra $\Lambda$, syzygy filtered algebra $\bm{\varepsilon}(\Lambda)$ is the endomorphism algebra of direct sum of projective modules over $\Lambda$ indexed by $\cS'(\Lambda)$, i.e.
\begin{align*}
\bm{\varepsilon}(\Lambda):=\End_{\Lambda}\left(\bigoplus\limits_{S\in \cS'(\Lambda)}P(S)\right)
\end{align*}
\end{definition}

\begin{definition}\cite{sen19} \label{defigherNakayama}
For any cyclic Nakayama algebra $\Lambda$, $j$th syzygy filtered algebra $\bm{\varepsilon}^j(\Lambda)$ is the endomorphism algebra of direct sum of projective modules which are indexed depending on $(j-1)$th syzygy filtered algebra i.e.
\begin{align*}
\bm{\varepsilon}^j(\Lambda):=\End_{\bm{\varepsilon}^{j-1}(\Lambda)}\left(\bigoplus\limits_{S\in\cS'(\bm{\varepsilon}^{j-1}(\Lambda))}P(S)\right)
\end{align*}
provided that $\bm{\varepsilon}^{j-1}(\Lambda)$ is a cyclic non self-injective Nakayama algebra.
\end{definition}

\begin{definition}\cite{aus}\label{def domdim} Dominant dimension of algebra $A$ is:
\begin{align*}
\domdim A=\sup\left\{m\vert\,I_i\, \text{is projective for }i=0,1,\ldots m\right\}+1
\end{align*}
where $0\rightarrow\,_AA\rightarrow I_0\rightarrow I_1\ldots$ is the minimal injective resolution of $A$. 
Dominant dimension of semisimple algebra is set to be infinity. We can define dominant dimension of any module $M$ by
 \begin{align*}
\domdim M=\sup\left\{m\vert\,I_i\, \text{is projective for }i=0,1,\ldots m\right\}+1
\end{align*}
where $0\rightarrow _A M\rightarrow I_0\rightarrow I_1\ldots$ is the minimal injective resolution of $M$. If we restrict $M$ to projective modules we get the following characterization of dominant dimension 
\begin{align}\label{gercek tanim2}
\domdim A=\sup\left\{\domdim P\,\vert\, P \text{ is projective non-injective } A\, \text{module}\right\}.
\end{align}
It is convenient for us to study projective dimensions of injective modules, because for a Nakayama algebra $\Lambda$, $\domdim\Lambda=\domdim\Lambda^{op}$ holds. Therefore \ref{gercek tanim2}
 is equivalent to
 \begin{align}\label{gercek tanim}
\domdim A=\sup\left\{\domdim I\,\vert\, I \text{ is injective non-projective } A\, \text{module}\right\}.
\end{align}
 We always use the characterization of dominant dimension \ref{gercek tanim} in the proofs. 
Global dimension of algebra $A$ is the supremum of projective dimensions of all $A$-modules:
\begin{align*}
\gldim A=\sup\left\{\pdim M\vert\,M\in \text{mod-}A \right\}
\end{align*} where mod-$A$ denotes the category of finitely generated $A$-modules.
\end{definition}

Although we do not use directly in this paper, in order to give the complete picture, we also mention here the definition of $\varphi$-dimension. Because, the syzygy filtration method is originally introduced in  \cite{sen19} to calculate its possible values.
\begin{definition}\label{varphi}
For a given $A$-module $M$, $\varphi\left(M\right)$ is defined in \cite{it} as:
$$\varphi(M):=\min\{t\ |\ \rank\left(L^t\langle add M\rangle\right)=\rank\left(L^{t+j}\langle add M\rangle\right)\text{ for }\forall j\geq 1\}$$ where $L[M]:=[\Omega M]$ in the Grothendieck group $K_0$ of $A$-modules modulo projective summands and $\langle addM\rangle$ is the subgroup of $K_0$ generated by all the indecomposable summands of the module $M$.   
$\varphi$-dimension of algebra $A$ is:
$$\varphi\dim(A):=\sup\{\varphi(M)\ |\ M \in\text{ mod-}A\}$$ where mod-$A$ denotes the category of finitely generated $A$-modules.
\end{definition}

\begin{remark} Here we collect some useful remarks about the outcomes of the syzygy filtration method. For details we refer to papers \cite{sen18}, \cite{sen19}. We assume that $\Lambda$ is a connected cyclic Nakayama algebra.
\begin{enumerate}[label=\arabic*.]
\item\label{remarklis1} If the global dimension of $\Lambda$ is finite, then there exists number $m$ such that ${\bm\varepsilon}^m(\Lambda)$ is a cyclic Nakayama algebra and ${\bm\varepsilon}^{m+1}(\Lambda)$ is a linear Nakayama algebra (not necessarily connected) which might have semisimple components.
\item\label{remarklis2} If $\Lambda$ is not selfinjective, then we have the following reductions:
\begin{enumerate}
\item $\gldim\Lambda=\gldim{\bm\varepsilon}(\Lambda)+2$ if $2\leq \gldim\Lambda<\infty$.
\item $\domdim\Lambda=\domdim{\bm\varepsilon}(\Lambda)+2$ when $3\leq \domdim\Lambda$. We want to explain why  $3$ is the lower bound. Semisimple algebra has global dimension $0$ or $-\infty$ and dominant dimension $\infty$. When we apply $\bm\varepsilon$-construction, we might get linear and semisimple components. For example $(3,2,3,2,2)$ is syzygy equivalent to $\Aa_2\oplus\Aa_1$. The former has dominant dimension $2$, global dimension $3$. The latter has dominant dimension one and global dimension one if we use the convention that dominant dimension of $\Aa_1$ is infinity. To avoid this, we choose lower bound for reduction of dominant dimensions $3$ in general. We give further details in example \ref{example dominant dimension}.
\item $\findim\Lambda=\findim{\bm\varepsilon}(\Lambda)+2$ if $2\leq \findim\Lambda$.
\item $\del\Lambda=\del{\bm\varepsilon}(\Lambda)+2$ if $\del\Lambda\geq 2$, where $\del$ denotes the delooping level of $\Lambda$. It turns out that $\del\Lambda=\findim\Lambda$ if $\Lambda$ is Nakayama algebra (\cite{rin2},\cite{sen20}).
\item $\varphi\dim\Lambda=\varphi\dim{\bm\varepsilon}(\Lambda)+2$ if $\varphi\dim\Lambda\geq 2$.
\end{enumerate}
$\Lambda$ is selfinjective if and only of there is isomorphism $\Lambda\cong\bm\varepsilon(\Lambda)$. Therefore the statements above cannot be valid for a selfinjective Nakayama algebra. In this case we get $\gldim\Lambda=\domdim\Lambda=\infty$, $\varphi\dim\Lambda=\findim\Lambda=\del\Lambda=0$. When global dimension is infinite, $\varphi\dim\Lambda$ is always an even number (\cite{sen18} Theorem A) , and it satisfies $0\leq \varphi\dim\Lambda-\findim\Lambda\leq 1$ (\cite{sen19}). For all the reductions we mention above, the lower bound cannot be one. Because, a cyclic Nakayama algebra cannot have global dimension and $\varphi$-dimension one. If $\findim\Lambda=1$, then $\varphi\dim\Lambda=2$, and $\bm\varepsilon(\Lambda)$ becomes selfinjective, so $\findim\bm\varepsilon(\Lambda)=0$. 
\item\label{gorgor} If $\Lambda$ is a Gorenstein algebra, then ${\bm\varepsilon}(\Lambda)$ is also Gorenstein.
\item\label{secondsyzygy} For any $\Lambda$-module $M$, $\Omega^i(M)$ and their projective covers when $i\geq 2$ has $\cB(\Lambda)$-filtration. Elements of the base set are exactly the second syzygies of simple modules which are the tops of \emph{minimal} projective modules. We call a projective module \emph{minimal} if its radical is not projective. Explicitly, by using \ref{projectiveclasses}, the second syzygies of simple modules $S_{k_1},S_{k_3},\ldots,S_{k_{2r-1}}$ gives the elements of $\cB(\Lambda)$.
\item \label{submodule of projectives} Any indecomposable projective $\Lambda$-module $P$ has a submodule from the base set $\cB(\Lambda)$ provided that none of the elements of $\cB(\Lambda)$ are projective. This simply follows from $\soc\Delta_i$ and $\soc P$ are elements of $\cS(\Lambda)$, by uniseriality of modules we get $\Delta_i\subset P$. If $\Delta_i\in\cB(\Lambda)$ and projective, then there exists at least one projective module having $\Delta_i$ as a proper submodule. Because by the previous item \ref{secondsyzygy}, there is a simple $\Lambda$-module $S$ such that $\Omega^2(S)\cong \Delta_i$, therefore $\Delta_i$ is a proper submodule of $P(\Omega^1(S))$.
\item\label{remarklis3} The category of $\cB(\Lambda)$-filtered modules is equivalent to the category of ${\bm\varepsilon}(\Lambda)$-modules.
\end{enumerate}

\subsection{Categorical Equivalence}\label{categoricalequivalnce} Here we want to give further details for Remark \ref{remarklis3} the category of $\cB(\Lambda)$ filtered $\Lambda$-modules is equivalent to $\bm\varepsilon(\Lambda)$-modules. We apply the construction which can be found in \cite{aus} section II.2.5.
Let $Filt(\cB(\Lambda))$ be the category of $\cB(\Lambda)$-filtered $\Lambda$-modules. A module $U\in\cB(\Lambda)$ if and only if $U$ has minimal projective presentation: 
\begin{align}
\xymatrixcolsep{5pt} 
\xymatrix{ P_a\ar[rrr]^{f} &&& P_b\ar[rrr] &&& U
}
\end{align}
where $P_a,P_b\in Filt(\cB(\Lambda))$ and $f$ cannot be factored through a projective module in  $Filt(\cB(\Lambda))$. We can describe modules over syzygy filtered algebra in the following way. Let $\cP:=\bigoplus_{i\in \cS'(\Lambda)} P_i$. Then $\bm{\varepsilon}(\Lambda)=\End_{\Lambda}\cP$ is finite dimensional algebra and $\cP$ is left $\bm{\varepsilon}(\Lambda)$-module. If $X$ is right $\Lambda$-module, then $\Hom_{\Lambda}(\cP,X)$ is right $\bm{\varepsilon}(\Lambda)$-module.
Simple $\bm{\varepsilon}(\Lambda)$-modules are of the form 
$\Hom_{\Lambda}(\cP,\Delta_i)$ where $\Delta_i\in\cB(\Lambda)$ (see \ref{baseset}). We get the diagram:

\[\xymatrixcolsep{5pt}
\xymatrix{   \text{mod-}\Lambda\ar[rrrr]^{\Hom_{\Lambda}(\cP,-)} &&&&\text{mod-}\bm{\varepsilon}(\Lambda)\\
 Filt(\cB(\Lambda)) \ar[rrrru] \ar@{^{(}->}[u]
}\]

Restriction of $\Hom_{\Lambda}(\cP,-)$ onto $Filt(\cB(\Lambda))$ gives equivalence of categories mod-$\bm{\varepsilon}(\Lambda)$ and  $Filt(\cB(\Lambda))$. Specifically,
\begin{enumerate}[label=\roman*)]
\item Any $\bm{\varepsilon}(\Lambda)$-module is isomorphic to $\Hom_{\Lambda}(\cP,X)$ for some $X\in Filt(\cB(\Lambda))$. 
\item Simple $\bm\varepsilon(\Lambda)$-modules are of the form $\Hom_{\Lambda}(\cP,\Delta_i)$ where $\Delta_i\in\cB(\Lambda)$.
\item Projective $\bm\varepsilon(\Lambda)$-modules are of the form $\Hom_{\Lambda}(\cP,P)$ where $P$ is $\cB(\Lambda)$-filtered projective $\Lambda$-module.
\item  If $P$ is projective-injective $\Lambda$-module where $P\in Filt\cB(\Lambda)$, then $\Hom_{\Lambda}(\cP,P)$ is projective-injective $\bm\varepsilon(\Lambda)$-module. We give a brief explanation. By item \ref{submodule of projectives}, there is a submodule $\Delta_i$ of $P$, therefore $\Hom_{\Lambda}(\cP,P)$ is the longest module having the socle $\Hom_{\Lambda}(\cP,\Delta_i)$ in mod-$\bm\varepsilon(\Lambda)$ which makes it injective.
\end{enumerate}
 We point out that converse of the last statement is not true in general.
\end{remark}

\section{Proofs of Theorems \ref{thmLamdaauslanderimplies}, \ref{thmdoubling}}\label{sectionmaintools}
Here, we give proofs of theorems \ref{thmLamdaauslanderimplies} and \ref{thmdoubling} as well as their implications.
\subsection{Relationship between the defect and $\varepsilon$-construction}
The proof of Theorem \ref{thmLamdaauslanderimplies} relies on properties of injective modules. That's why we want to examine their structure in details. First we recall some statements from \cite{sen19} and their new proofs.

\begin{lemma}\label{injectivity} \cite{sen19}
If $X$ is a proper submodule of $\Delta_i$ where $\Delta_i\in\cB(\Lambda)$, then there exist a projective-injective module $P$ such that  $\Delta_i\subset P$ and the quotient $\faktor{P}{X}$ is an injective $\Lambda$-module.
\end{lemma}
\begin{proof}

Let $S$ be a submodule of projective-injective module $P$. We have the exact sequence
\begin{align*}
0\rightarrow S\rightarrow P\rightarrow \faktor{P}{S}\rightarrow 0.
\end{align*}
We want to find the injective envelope of the quotient $\faktor{P}{S}$. Consider the map $\pi: P'\rightarrow I(\faktor{P}{S})$ where $P'$ be the projective cover of the injective envelope $I(\faktor{P}{S})$. Because modules are uniserial and $I$ is injective there exists an embedding 

\begin{align}\label{sat exact}
c:\faktor{P}{S}\mapsto I(\faktor{P}{S}).
\end{align} We work on the diagram \ref{diagram sth2}.

\begin{figure}\caption{Diagram for the embedding \ref{sat exact}}\label{diagram sth2}
\begin{center}
 $
\xymatrix{&& 0\ar[d]   && 0\ar[d] && &&\\
&& \ker a\ar[rr]^{\approx}\ar[d]   && \ker b\ar[d] &&&&\\
0\ar[rr]&& S\ar[rr]\ar[d]^{\Large{a}}   && P\ar[rr]\ar[d]^b &&\faktor{P}{S}\ar@{.>}[dll]_{\beta}\ar[rr]\ar@{_{(}->}[d]^c&&0\\
0\ar[rr]&& \ker\pi\ar[rr]\ar[d]   && P'\ar[rr]^{\pi}\ar[d] &&I(\faktor{P}{S})\ar[rr]\ar[d]&&0\\
0\ar[rr]&& \coker a\ar[rr]\ar[d] && \coker b\ar[d]\ar[rr]&&\coker c\ar[d]\ar[rr] &&0\\
&& 0 && 0&&0 && }$
 \end{center}
 \end{figure}
 
\begin{enumerate}[label=\roman*)]
\begingroup
\allowdisplaybreaks
\item If $\ker a=S$ then $a$ is zero map and there exists $\beta$ from $\faktor{P}{S}$ to $P'$ such that $\pi\circ\beta=c$. Therefore $\beta$ is injective map and $\soc\faktor{P}{S}\cong \soc P' $ and in particular $\soc \faktor{P}{S}\in\cS(\Lambda)$. This makes $S$ an element of the base set $\cB(\Lambda)$ because $\soc\faktor{P}{S}\cong\tau^{-1}S$ and $S$ are consecutive modules. 
\item Let $S$ be a proper simple submodule of $\Delta_i$. Therefore $\ker a\ncong S$. This implies $a$ is monomorphism and then $b$ is monomorphism which makes $P\cong P'$. Since $P$ is injective then $b,c$ are injective. As a result $c$ is isomorphism i.e. $\faktor{P}{S}\cong I(\faktor{P}{S})$.
\endgroup
\end{enumerate} 
We use the diagram \ref{diagram sth2} as inductive step.
Let $0\subset X_1\subset X_2\subset \cdots\subset \Delta_i$ be the composition series of $\Delta_i$. Assume that $\faktor{P}{X_t}$ is injective for $1\leq t\leq k-1$. Consider the exact sequence 
\begin{align}
0\rightarrow X_k\rightarrow P\rightarrow \faktor{P}{X_k}\rightarrow 0 
\end{align}
and the map $\pi: P'\rightarrow I(\faktor{P}{X_k})$. Therefore there is embedding $c:\faktor{P}{X_k}\mapsto I(\faktor{P}{X_k})$.
We get the following diagram:
\begin{center}
 $
\xymatrix{
&& \ker a\ar[rr]^{\approx}\ar[d]   && \ker b\ar[d] &&&&\\
0\ar[rr]&& X_k\ar[rr]\ar[d]^{\Large{a}}   && P\ar[rr]\ar[d]^b &&\faktor{P}{X_k}\ar[rr]\ar@{_{(}->}[d]^c&&0\\
0\ar[rr]&& \ker\pi\ar[rr]   && P'\ar[rr]^{\pi} &&I(\faktor{P}{X_k})\ar[rr]&&0\\ }$
 \end{center}
 \begin{enumerate}[label=Case \roman*)]
 \item $\ker a=0$ implies $\ker b=0$ and $P$ is submodule of $P'$. Because $P$ is projective-injective, $P\cong P'$ and $c$ is isomorphism. Hence $\faktor{P}{X_k}$ is an injective module.
 \item $\ker a\cong X_k$ so $a$ is zero map and $\faktor{P}{X_k}$ is submodule of $P'$. So $\soc \faktor{P}{X_k}\in\cS(\Lambda)$ which makes $X_k$ isomorphic to $\Delta_i$.
 \item\label{case3}: By induction $\faktor{P}{X_t}$ is injective module, so $b'$ splits. this makes $\faktor{P}{X_t}\cong P'$ and $c'$ an isomorphism. So $\faktor{P}{X_k}$ is an injective module since $\faktor{P}{X_k}\cong I(\faktor{P}{X_k})$. See the diagram \ref{diagram sth}.
 \end{enumerate}
 
 \begin{figure}\caption{Diagram for \ref{case3}}\label{diagram sth}
  \begin{center}
 $\xymatrix{
&& X_t\ar[rr]^{\approx}\ar@{_{(}->}[d]   && X_t\ar@{_{(}->}[d] &&&&\\
0\ar[rr]&& X_k\ar[rr]\ar[dd]^a\ar@{^{(}->}[dr]   && P\ar[rr]\ar[dd]^<<<<<<<{b}\ar@{^{(}->}[dr] &&\faktor{P}{X_k}\ar[rr]\ar@{_{(}->}[dd]^<<<<<<<c\ar@{^{(}->}[dr]^{\approx}&&0\\
&&& \faktor{X_k}{X_t}\ar[rr]\ar@{_{(}->}[dl]   &&\faktor{P}{X_t}\ar@{_{(}->}[dl]^{b'}\ar[rr] && \faktor{P}{X_k}\ar@{_{(}->}[dl]^{c'}&\\
0\ar[rr]&& \ker\pi\ar[rr]   && P'\ar[rr]^{\pi} &&I(\faktor{P}{X_k})\ar[rr]&&0\\ }$
 \end{center}
  \end{figure}
\end{proof}

\begin{lemma}\label{every second syzy of injective} For any injective non-projective $\Lambda$-module $I$, $\Omega^1(I)$ is a proper submodule of $\Delta_i\in\cB(\Lambda)$ for some $i$.
\end{lemma}
\begin{proof} There are two possibilities that we need to eliminate separately, either there is $\cB(\Lambda)$-filtered proper submodule of $\Omega^1(I)$ or $\Omega^1(I)\in \cB(\Lambda)$. We start with the first case. Assume to the contrary that $\Delta_i$ is a proper submodule of $\Omega^1(I)$, so there exists an exact sequence
\begin{align}\label{sth sth 2}
0\rightarrow \Delta_i\rightarrow \Omega^1(I)\rightarrow X\rightarrow 0.
\end{align}
where $X$ is a subquotient of projective cover $P(I)$ of $I$.
Because $\soc X\in\cS(\Lambda)$, there exists $\Delta_{i-1}\in\cB(\Lambda)$ sharing the same socle. Moreover there exist an indecomposable projective module $P$ such that $\Delta_{i-1}\subset P$ by \ref{submodule of projectives}. Since $P$ is projective, $X$ and $\faktor{P(I)}{\Delta_i}$ have to be proper submodules of $P$. Otherwise $P$ would be quotient or subquotient of $P(I)$ which is not possible. We can compare the lengths:
\begin{gather}
\ell(P)>\ell(P(I))-\ell(\Delta_i)\implies\nonumber\\
\ell(P)-\ell(X)>\ell(P(I))-\ell(X)-\ell(\Delta_i)\implies\nonumber\\
\ell(P)-\ell(X)>\ell(I)\label{sth sth}
\end{gather}
The right hand side of the second inequality is the length of $I$, because $\ell(I)=\ell(P(I))-\ell(\Omega^1(I))$ and $\ell(\Omega^1(I))=\ell(X)+\ell(\Delta_i)$ by \ref{sth sth 2}. However \ref{sth sth} implies that $I$ is proper submodule of $\faktor{P}{X}$ which contradicts to the injectivity of $I$.
\par For the case $\Omega^1(I)\cong\Delta_i$ we use similar arguments. Socle of $I$ is $\tau^{-1}\topp\Delta_i\cong\soc\Delta_{i-1}$, so $\soc I\in\cS(\Lambda)$. By Remark \ref{submodule of projectives} there exists a projective module $P$ having the same socle. Since $P$ is projective, $\faktor{P(I)}{\Delta_i}$ have to be proper submodule of $P$. Otherwise $P$ would be quotient or subquotient of $P(I)$ which is not possible. However  $\faktor{P(I)}{\Delta_i}\cong \faktor{P(I)}{\Omega^1(I)}\cong I$ makes $I$ proper submodule of $P$ which violates the injectivity of $I$.
\par Because the socle of $\Omega^1(I)$ is an element of $\cS(\Lambda)$, there exists $\Delta_i\in\cB(\Lambda)$ having the same socle and satisfying either $\Delta_i\subseteq \Omega^1(I)$ or $\Omega^1(I)\subset \Delta_i$ by the uniseriality of modules. We showed that the former is not possible. Hence $\Omega^1(I)$ is submodule of $\Delta_i$ for some $i$.
\end{proof}
We can combine the results \ref{injectivity} and \ref{every second syzy of injective} to get:
\begin{corollary} If $\Delta_i$ is not a simple module, then for all proper submodules $X\subset\Delta_i$, there exist injective modules such that $\Omega^1(I)\cong X$.
\end{corollary}

\begin{proposition}\label{GorensteinReduction}
Let $I$ be an injective but non-projective ${\bm\varepsilon}(\Lambda)$-module. Then there exists an injective but non-projective $\Lambda$-module $I_{\Lambda}$ such that $\Hom_{\Lambda}\left(\cP,\Omega^2(I_{\Lambda})\right)\cong I$.
\end{proposition}
\begin{proof} By Remark \ref{secondsyzygy}, $\Omega^2(I_{\Lambda})$ has $\cB(\Lambda)$-filtration. For simplicity, we denote $\Omega^2(I_{\Lambda})$ by $I\Delta$.
Because $\soc I\Delta\in\cS(\Lambda)$, there exist indecomposable projective $\Lambda$-module $P$ such that $\soc P\cong\soc I\Delta$. Because modules are uniserial either $P\subseteq I\Delta$ or $I\Delta\subset P$. The former implies that $I\Delta$ is projective $\Lambda$-module (because it has projective submodule $P$), therefore by Remark \ref{categoricalequivalnce} $I$ is projective $\bm\varepsilon(\Lambda)$-module, which is impossible due to the assumption on non-projectivity of $I$.
\par Now we can assume that $I\Delta$ is submodule of projective $\Lambda$-module $P$, hence we get the exact sequence
\begin{align}\label{seq1}
0\rightarrow I\Delta\rightarrow P\rightarrow \faktor{P}{I\Delta}\rightarrow 0.
\end{align}

We conclude $P\notin Filt(\cB(\Lambda))$. Otherwise, $I$ would be a submodule of projective module $\Hom_{\Lambda}(\cP,P)$ which violates injectivity of $I$. We claim that the quotient $Q=\faktor{P}{I\Delta}$ is a submodule of an element of the base set $\cB(\Lambda)$. Notice that $\soc Q\in\cS(\Lambda)$, so there exists $\Delta_i\in\cB(\Lambda)$ such that $\soc Q\cong \soc\Delta_i$. Because modules are uniserial either $\Delta_i\subseteq Q$ or $Q\subset \Delta_i$ holds. The former  implies that the module $\begin{vmatrix}
\Delta_i\\I\Delta
\end{vmatrix}$ is indecomposable submodule of $P$. Therefore $I$ becomes submodule of the corresponding module $\begin{vmatrix}
S_i\\I
\end{vmatrix}$ in $\bm\varepsilon(\Lambda)$ which violates the injectivity of $I$ where $\Delta_i\in\cB(\Lambda)$ corresponds to simple $\bm\varepsilon(\Lambda)$-module $S_i$ for some $i$ i.e. $\Hom_{\Lambda}(\cP,\Delta_i)\cong S_i$.
\par Because $Q$ is a proper submodule of $\Delta_i$, by the lemma \ref{injectivity}, the quotient $\faktor{PI}{Q}$ is an injective $\Lambda$-module where $PI$ is the projective-injective module satisfying $Q\subset\Delta_i\subset PI$. Existence of such $PI$ follows from the lemma \ref{injectivity}. Therefore we get the sequence:
\begin{align}\label{seq2}
0\rightarrow Q\rightarrow PI\rightarrow \faktor{PI}{Q}\rightarrow 0.
\end{align}
By using the exact sequences \ref{seq1} and \ref{seq2}, we get $\Omega^2\left(\faktor{PI}{Q}\right)\cong I\Delta$. Therefore $I_{\Lambda}\cong \faktor{PI}{\left(\faktor{P}{I\Delta}\right)}$ is the injective $\Lambda$-module satisfying the hypothesis.
\end{proof}
As a corollary we can conclude that for every injective but non-projective ${\bm\varepsilon}(\Lambda)$-module $I$, there exists at least one injective but non-projective $\Lambda$-module $I_{\Lambda}$ such that $\Omega^2(I_{\Lambda})\cong I$. This can occur because  in the proof above, projective $P$ can be longer. In general there might be more than one injective module such that their second syzygies are isomorphic. Or the first or second syzygies of injective module can be projective. For example, the projective resolution of $I_2$ stops at $P_3$ if the algebra is given  by the Kupisch series  $(4,3,2,2)$.  We point out that this does not contradict to \ref{every second syzy of injective}, because $P_2$ is an element of the base set and $\Omega^1(I_2)\cong P_3\subset P_2$. Algebra with Kupisch series $(5,4,4,4,3)$ is $\bm\varepsilon$-equivalent to $(3,2,2)$. However $I_4$ and $I_3$ gives $I_2$ in $(3,2,2)$ via the functor $\Hom(\cP,-)$ (see \ref{categoricalequivalnce}). In particular dominant dimension is one, because $P_5$ is not projective-injective.

\par Another consequence of the Proposition \ref{GorensteinReduction} is the Remark \ref{gorgor}.
We call an algebra Gorenstein if the projective dimensions of injective modules are finite. If $\Lambda$ is a cyclic Nakayama algebra which is Gorenstein, then $\bm\varepsilon(\Lambda)$ is also Gorenstein because all injective non-projective $\bm\varepsilon(\Lambda)$-modules can be obtained as the second syzygies of injective $\Lambda$-modules. So we get
\begingroup
\allowdisplaybreaks
\begin{align*}
\sup\left\{\pdim I\vert I\in\text{mod-}\Lambda\text{ injective}\right\}<\infty \implies \sup\left\{\pdim I'\vert I'\in\text{mod-}\bm\varepsilon(\Lambda)\,\text{injective}\right\}<\infty.
\end{align*}
\endgroup

\begin{proposition}\label{different injectives} If $I_1,I_2,\ldots,I_m$ are indecomposable nonisomorphic injective $\Lambda$-modules satisfying $\Omega^2(I_i)\cong \Omega^2(I_{j})$ for all $1\leq i\neq j\leq m$, then projective covers $P(I_1),$ $P(I_2),\ldots,P(I_m)$ are isomorphic. In particular $\topp I_i\cong \topp I_{j}$ and $\topp P(\Omega^1(I_i))\ncong \topp P(\Omega^1(I_{j}))$  for all $1\leq i\neq j\leq m$.
\end{proposition}
\begin{proof}
Consider the projective resolution of $I_1,\ldots,I_m$ where $1\leq i\leq m$:
\begin{center}
 $\xymatrixcolsep{5pt}
\xymatrix{& \cdots\ar[rr]\ar[rd] && P(\Omega^1(I_i))\ar[rr]\ar[rd]   && P(I_i)\ar[rr] &&I_i\\
&&\Omega^2(I_i)\ar[ru] && \Omega^1(I_i)\ar[ru]}$
 \end{center}
Because $\Omega^2(I_i)$ is submodule of $P(\Omega^1(I_i))$, the socle of the quotient is consecutive to the top  of $\Omega^2(I_i)$ i.e. $\tau^{-1}\topp\Omega^2(I_i)\cong \soc\Omega^1(I_i)$ for all $i$. In particular it is isomorphic to socle of $P(I_i)$. Notice that $P(I_i)$'s are projective-injective modules since they are projective covers of injective modules. Indecomposable projective-injective modules with the same socle is unique upto isomorphism, the first result follows.
\par Let's denote the projective covers of $I_1,\ldots,I_m$ by $P$. It is clear that $\topp P\cong \topp I_i$ for all $i$. Because injective modules are not isomorphic, we get $\Omega^1(I_i)\ncong\Omega^1(I_j)$ and in particular  $P(\Omega^1(I_i))\ncong P(\Omega^1(I_j))$ for $i\neq j$.
\end{proof}
\begin{corollary}\label{cor of different injectives} If there exists at least two injective $\Lambda$-modules $I_1$, $I_2$ satisfying $\Omega^2(I_1)\cong\Omega^2(I_2)$, then the dominant dimension of $\Lambda$ is one.
\end{corollary}
\begin{proof}
By Proposition \ref{different injectives}, projective covers of $I_1$ and $I_2$ are isomorphic, however $P(\Omega^1(I_1))\ncong P(\Omega^1(I_2))$. Notice that they have isomorphic socles since $P(\Omega^1(I_1))\supset\Omega^2(I_1)\cong\Omega^2(I_2)\subset P(\Omega^1(I_2))$. They are uniserial modules, without loss of generality let $P(\Omega^1(I_1))$ be proper submodule of $P(\Omega^1(I_2))$. Therefore $P(\Omega^1(I_1))$ cannot be a projective-injective module, by definition \ref{gercek tanim}, the dominant dimension of $I_1$ is one which makes $\domdim\Lambda=1$.
\end{proof}

We are ready to state and prove the Theorem \ref{thmLamdaauslanderimplies}.
\begin{theorem}\label{thmdefect}[Theorem \ref{thmLamdaauslanderimplies}]
If $\Lambda$ is a higher Auslander algebra, then $\bm\varepsilon(\Lambda)$ is so, provided that $\gldim\Lambda\geq 3$. Moreover, their defects are same i.e. $\defect\Lambda=\defect\bm\varepsilon(\Lambda)$.
\end{theorem}

\begin{proof} [Proof of Theorem \ref{thmLamdaauslanderimplies}] By the remark \ref{remarklis2}, we get the reductions
\begin{gather}
\gldim\Lambda=\gldim{\bm\varepsilon}(\Lambda)+2\\
\domdim\Lambda=\domdim{\bm\varepsilon}(\Lambda)+2
\end{gather}
when $\gldim\Lambda\geq 3$ and $\domdim\Lambda\geq 3$.  If $\Lambda$ is a higher Auslander algebra i.e. $\gldim\Lambda=\domdim\Lambda$, it immediately follows that ${\bm\varepsilon}(\Lambda)$ is higher Auslander algebra because $\gldim\bm\varepsilon(\Lambda)=\domdim\bm\varepsilon(\Lambda)$. 
\par Now we want to show the equality of the defects. By the Proposition \ref{GorensteinReduction}, every injective non-projective $\bm\varepsilon(\Lambda)$-module can be obtained as a second syzygy of injective $\Lambda$-modules. Therefore $\defect\bm\varepsilon(\Lambda)\leq\defect\Lambda$ for all cyclic $\Lambda$. Moreover by the Proposition \ref{different injectives} and its corollary \ref{cor of different injectives}, if there are more than two injective $\Lambda$-modules which induces the same injective $\bm\varepsilon(\Lambda)$-module via the functor $\Hom_{\Lambda}(\cP,\Omega^2(-))$, then the dominant dimension cannot be reduced by two. In other words, $\Lambda$ cannot be a higher Auslander algebra. This forces that for each injective non-projective $\bm\varepsilon(\Lambda)$-module $I$, there exist exactly one injective $\Lambda$-module $I'$ such that $\Hom_{\Lambda}(\cP,\Omega^2(I'))\cong I$ which is unique. Therefore defects are equal.
\end{proof}

The Theorem \ref{thmLamdaauslanderimplies} does not cover global dimensions one and two. In the case of $\gldim\Lambda=2$, ${\bm\varepsilon}(\Lambda)$ has to be semisimple i.e. ${\bm\varepsilon}(\Lambda)\cong \Aa^{\times n}_1$ for some $n$. By the definition \ref{def domdim} its dominant dimension is infinity.  We give detailed exposition of global dimension $2$ case in Section \ref{sectionkis2}. However, since we set the defect of $\Aa^{\times n}_1$ to $n$, we can still conclude that the defects are same (see \ref{defect of linear}). A cyclic Nakayama algebra cannot have global dimension one, so we can skip it. An important consequence is:
\begin{corollary}\label{corAus}
If $\Lambda$ is a higher Auslander algebra, then  higher syzygy filtered algebras ${\bm\varepsilon}^j(\Lambda)$, provided that $1\leq j\leq m+1$ where $\bm\varepsilon^{m}(\Lambda)$ is cyclic, $\bm\varepsilon^{m+1}(\Lambda)$ is linear (not necessarily connected) have the same defect, which is the defect of $\Lambda$. In other words, the defect of higher Auslander algebra is an invariant under ${\bm\varepsilon}$-construction.
\end{corollary}

\subsection{Covering of Algebras}

Let $\Lambda$ be a cyclic Nakayama algebra of rank $n$. Let $(c_1,\ldots,c_n)$ be its Kupish series. Cyclic Nakayama algebra $\Lambda'$ is called double covering of $\Lambda$, if $c_i=c_{i+n}$ for all $1\leq i\leq n$ where $c_i$ is the length of projective $\Lambda'$-module $P_i$. In other words, Kupisch series of $\Lambda'$ is $(c_1,\ldots,c_n,c_1,\ldots,c_n)$. 

Because of the remark \ref{categoricalequivalnce}, to understand $\bm\varepsilon(\Lambda)$-modules we use $\cB(\Lambda)$-filtered $\Lambda$-modules. It is necessary for us to introduce the following.
\begin{definition}\label{b length} Let $M$ be a $\cB(\Lambda)$-filtered $\Lambda$-module. $\cB(\Lambda)$-length of $M$ is defined as the number $m$ such that
\begin{align}\label{composition series}
0\subset M_1\subset M_2\subset\cdots\subset M_m\cong M
\end{align}
where each $\faktor{M_{i+1}}{M_{i}}$ is an element of $\cB(\Lambda)$. The modules $\faktor{M_{i+1}}{M_{i}}$ are called $\cB(\Lambda)$-composition factors or series of module $M$.  $\cB(\Lambda)$-radical of $M$ is defined as $M_{m-1}$. $\cB(\Lambda)$-socle of $M$ is defined as $M_1$. $\cB(\Lambda)$-top of $M$ is $\faktor{M}{M_{m-1}}$. 
\end{definition}
By the previous definition, the $\cB(\Lambda)$-length of $\cB(\Lambda)$-filtered $\Lambda$-module $M$ is equal to the length of $\bm\varepsilon(\Lambda)$-module $\Hom_{\Lambda}(\cP,M)$.

\begin{theorem} \label{thmdoublininside}[Theorem \ref{thmdoubling}]  If $\Lambda'$ is a double covering of $\Lambda$, then $\bm\varepsilon(\Lambda')$ is double covering of $\bm\varepsilon(\Lambda)$. 
\end{theorem}

\begin{proof} For a given $N$, let $[j]_N$ denote the least positive residue of $j$ modulo $N$ when $j>0$ and $[N]_N=N$. We set $\rank\Lambda=n$, hence $\rank\Lambda'=2n$.  If $P_i$ is projective $\Lambda'$-module for $1\leq i\leq n$, the index of its socle is $[i+\ell(P_i)-1]_{2n}$ where $\ell(P_i)$ denotes the length of $P_i$. Because $\Lambda'$ is double covering of $\Lambda$, the index of the socle of $P_{i+n}$ is
\begin{align}\label{function gus}
[i+n+\ell(P_{i+n})-1]_{2n}=[i+n+\ell(P_i)-1]_{2n}
\end{align}

If $j$ is an index of $\soc P_i$ satisfying  $1\leq j\leq n$, then $j+n$ is the index of $\soc P_{i+n}$ satisfying $n+1\leq j+n\leq 2n$ by \ref{function gus}.
Therefore, if $S_{k_2},\ldots,S_{k_{2r}}$ are the socles of projective $\Lambda'$-modules satisfying $1\leq k_{2j}\leq n$, then $S_{[k_2+n]_{2n}},\ldots,S_{[k_{2r}+n]_{2n}}$ are in $\cS(\Lambda)$. Similarly, If $j$ is an index of $\soc P_i$ satisfying  $n+1\leq j\leq 2n$, then $j-n=[j+n]_{2n}$ is the index of $\soc P_{i+n}$ satisfying $1\leq j-n\leq n$ by \ref{function gus}. Therefore, if $S_{k_2},\ldots,S_{k_{2r}}$ are the socles of projective $\Lambda'$-modules satisfying $n+1\leq k_{2j}\leq 2n$, then $S_{[k_2+n]_{2n}},\ldots,S_{[k_{2r}+n]_{2n}}$ are in $\cS(\Lambda')$. After combining these two cases and using the definition of $\cS'(\Lambda')$-set, we have
 \begin{align}
\begin{gathered}\label{base set of doubling}
\cS(\Lambda')=\left\{S_{k_2},S_{k_4},\ldots,S_{k_{2r}},S_{k_{2+n}},\ldots,S_{k_{2r+n}}\right\}\\
\cS'(\Lambda')=\left\{\tau S_{k_2},\tau S_{k_4},\ldots,\tau S_{k_2r},\tau S_{k_{2}+n},\ldots,\tau S_{k_{2r}+n}\right\}.
\end{gathered}
\end{align}
where $1\leq k_{2j}\leq n$ for all $1\leq j\leq r$.

A module $\Delta$ is an element of the base set $\cB(\Lambda')$ if $\soc\Delta\in\cS(\Lambda')$, $\topp\Delta\in\cS'(\Lambda')$ and the composition factors of the largest subquotient of $\Delta$ are not from the set $\cS(\Lambda')\cup\cS'(\Lambda')$. We want to show that $P_i,P_{i+n}$ $1\leq i\leq n$ have the same $\cB(\Lambda')$-length if $P_i, P_{i+n}\in Filt\cB(\Lambda')$. We give a proof by induction.

\par Let $\topp\Delta_i\cong S_i$ and $\soc\Delta_i\cong T_i$. Therefore $\topp\Delta_{[i+n]_{2n}}\cong S_{[i+n]_{2n}}$ and $\soc\Delta_{[i+n]_{2n}}\cong T_{[i+n]_{2n}}$ for all $i$. $\cB(\Lambda')$-length of $\Delta_{[i+n]_{2n}}$ is one because for any subquotient $S_j$ of $\Delta_i$, we get $S_{[j+n]_{2n}}\notin \cS(\Lambda')\cup\cS'(\Lambda')$ by using \ref{base set of doubling} which makes $\Delta_{[i+n]_{2n}}\in\cB(\Lambda')$. Therefore, if module $M_i$ has $\cB(\Lambda')$-filtration, then $M_{[i+n]_{2n}}$ has also $\cB(\Lambda')$-filtration, because $\topp M_{[i+n]_{2n}}\in\cS'(\Lambda')$ and $\soc M_{[i+n]_{2n}}\in\cS(\Lambda')$ by similar arguments. This is the inductive step. Now, we assume that for all $\cB(\Lambda)$-filtered indecomposable modules $M_i$ of $\cB(\Lambda')$-length less than $t$, $M_{[i+n]_{2n}}$ has $\cB(\Lambda')$-filtration and the same length less than $t$. Let $\cB(\Lambda')$-length of module $M_i$ be $t$. The exact sequence
\begin{align}
0 \rightarrow \rad_{\cB(\Lambda')} M_i\rightarrow M_i\rightarrow \Delta_i\rightarrow 0
\end{align}
shows that $\rad_{\cB(\Lambda')} M_i$ is of $\cB(\Lambda')$-length $t-1$. So the module $\rad_{\cB(\Lambda')}M_{[i+n]_{2n}}$ is of $\cB(\Lambda')$-length $t-1$. Similarly $\cB(\Lambda')$-lengths of $\Delta_i$ and $\Delta_{[i+n]_{2n}}$ are one. Therefore
\begin{align}
0 \rightarrow \rad_{\cB(\Lambda')} M_{[i+n]_{2n}}\rightarrow M_{[i+n]_{2n}}\rightarrow \Delta_{[i+n]_{2n}}\rightarrow 0
\end{align}
is an exact sequence where the $\cB(\Lambda')$-length of $M_{[i+n]_{2n}}$ is $t$.
\par If $P_i$ has $\cB(\Lambda')$-filtration $1\leq i\leq n$ then $P_{i+n}$ has also, because $\topp P_i\in\cS'(\Lambda')$, $\soc P_i\in\cS(\Lambda')$ implies that top and socle of $P_{i+n}$ is in $\cS'(\Lambda')$ and $\cS(\Lambda')$ respectively which follows from  \ref{base set of doubling}. Moreover their $\cB(\Lambda')$-lengths are equal by the induction we used above. Hence we get repetition in the Kupisch series of $\Lambda'$. Consider the map $F:\mod\Lambda'\rightarrow \mod\Lambda$ such that $F(S_i)\cong F(S_{[i+n]_{2n}})\cong\tilde{S_i}$.  Via the map $F$, both $P_i,P_{[i+n]_{2n}}$ maps to $\tilde{P_i}=P(\tilde{S_i})$ of $\Lambda$ and $\cB(\Lambda')$ to $\cB(\Lambda)$. Therefore Kupisch series of $\bm\varepsilon(\Lambda)$ appears twice in $\bm\varepsilon(\Lambda')$.
\end{proof}

Actually, this result can be generalized directly by using an analogous idea of the proof of the Theorem \ref{thmdoublininside}:
\begin{corollary} If $\Lambda'$ is a $m$-fold covering of $\Lambda$, then $\bm\varepsilon(\Lambda')$ is $m$-fold covering of $\bm\varepsilon(\Lambda)$.
\end{corollary}

\begin{example}\label{example method} Here we give two examples:
 \begin{enumerate}[label=\roman*)]
\item Let $\Lambda=(5,4,4,3)$. Then $\bm\varepsilon(\Lambda)=(2,2)$. One can verify that $(5,4,4,3,5,4,4,3)$ gives $(2,2,2,2)$.
\item Let $\Lambda=(4,3,3,3)$. Then ${\bm\varepsilon}(\Lambda)=(3,2,2)$ and ${\bm\varepsilon}^2(\Lambda)=(2,1)$. One can verify that $(4,3,3,3,4,3,3,3)$ is $\bm\varepsilon$-equivalent to $(3,2,2,3,2,2)$. If we apply $\bm\varepsilon$ once more, we get $(2,1)\oplus (2,1)$. In particular this shows that our construction is different than other retraction methods, as opposed to what they claim in p.798 \cite{rene}, because they reduce $(2,2,3,2,2,3)$ to $(2, 2, 3, 2, 1)$ at p.792 and they would reduce $(3,3,3,4,3,3,3,4)$ to $(3, 3, 3, 4, 3, 2, 2)$.
\end{enumerate}
\end{example}

\begin{corollary} If $\Lambda$ is a cyclic Nakayama algebra of rank $n$ which is a higher Auslander algebra of global dimension $k$, then for each integer multiple of $n$, there exists a higher Auslander algebra of global dimension $k$.
\end{corollary}
\begin{proof}
Since $\Lambda$ is of finite global dimension, by remark \ref{remarklis1}, there is $m$ such that $\bm\varepsilon^{m+1}(\Lambda)$ is a linear Nakayama algebra which is a higher Auslander algebra, we denote it by $L$.
Let $\Lambda'$ be covering algebra of $\Lambda$ for some integer $t$. Then by the Theorem \ref{thmdoublininside}, ${\bm\varepsilon}^{m+1}(\Lambda')\cong \oplus ^tL$, where $\oplus^t L$ is higher Auslander linear Nakayama algebra. We exclude the semisimple case which is examined in \ref{friday example} (see also \ref{sectionkis2}). Because $\bm\varepsilon$-process reduces both of global dimension and dominant dimension by two, we get:
$\gldim \oplus^t L=\domdim \oplus^t L\iff \gldim\Lambda'=\domdim\Lambda'$. This finishes the proof. 
\end{proof}

\section{Proof of Theorem \ref{thmreverseepsilon}}\label{section cyclic}
First we assume existence of $\Lambda'$ such that $\bm\varepsilon(\Lambda')\cong\Lambda$. If we pose the condition that $\Lambda'$ is a higher Auslander algebra, this will lead uniqueness of such $\Lambda'$. And then we can describe Kupisch series of $\Lambda'$ which shows its existence and then finishes the proof.\\

 Let $\Lambda$ be a cyclic Nakayama algebra of rank $n$ and given by irredundant system of relations \ref{relations}. We have the equality
\begin{align}\label{longequality}
\#\rel\Lambda=\#\cB(\Lambda)=\#\cS(\Lambda)=\#\cS'(\Lambda)
\end{align}
(see definition \ref{defbase}) where $\#$ denotes the cardinality of given set. Moreover, we have
\begin{align}
\label{cardinalityofrelations}
\begin{split}
\#\rel\Lambda&=\text{number of projective classes}\,\,(\text{see}\,\ref{projectiveclasses})\\
&=\text{number of minimal projectives}\,\,( \text{see Remark}\,\ref{secondsyzygy})\\
&=\text{number of projective-injectives}\\
&=\text{number of injective classes}\,\,(\text{see}\,\ref{injective class})
\end{split}
\end{align}
An injective $\Lambda$-module $I$ is not projective if and only if socle of $I$ is not an element of $\cS(\Lambda)$. Therefore the defect of any cyclic Nakayama algebra $\Lambda$ is subject to equality
\begin{align}\label{rankequaation}
\rank\Lambda=\defect\Lambda+\#\rel\Lambda
\end{align}
since $\#\cS(\Lambda)=\#\rel\Lambda$.

Let $\Lambda$ and $\Lambda'$ be cyclic Nakayama algebras which are higher Auslander algebras satisfing $\bm\varepsilon(\Lambda')\cong \Lambda$.  Keeping this information in mind, we can deduce the following properties of $\Lambda$ and $\Lambda'$.
\begin{enumerate}[label=\arabic*)]
\item By \ref{rankequaation}, cyclic Nakayama algebras $\Lambda$ and $\Lambda'$ satisfy
\begin{align*}
\begin{gathered}
\rank\Lambda=\#\rel\Lambda+\defect\Lambda\\
\rank\Lambda'=\#\rel\Lambda'+\defect\Lambda'
\end{gathered}
\end{align*}
\item Because $\Lambda'$ is higher Auslander algebra, by Theorem \ref{thmdefect}:
\begin{align*}
\defect\Lambda'=\defect\Lambda
\end{align*}
\item Because of the $\bm\varepsilon$-construction i.e. $\bm\varepsilon(\Lambda')\cong\Lambda$:
\begin{align}\label{equalityof some numbers}
\#\rel\Lambda'=\#\cB(\Lambda')=\#\cS'(\Lambda')=\#\cS(\Lambda')=\rank\bm\varepsilon(\Lambda')=\rank\Lambda.
\end{align}
\item Hence we obtain the rank of $\Lambda'$:
\begin{align*}
\rank\Lambda' & =\#\rel\Lambda'+\defect\Lambda'\\
&=\rank\Lambda+\defect\Lambda.
\end{align*}
\end{enumerate}

\begin{definition} [Inheritance] \label{definheritance}
Let $S_1,S_2,\ldots,S_n$ be simple $\Lambda$-modules. If $\Lambda'$ is a cyclic Nakayama algebra satisfying $\bm\varepsilon(\Lambda')\cong\Lambda$, then there is a bijection between simple $\Lambda$-modules and elements of $\cB(\Lambda')$ by Remark \ref{remarklis3}. For each $S_i$, we denote the corresponding $\Lambda'$-module by $\Delta(S_i)$. In other words, $\Hom_{\Lambda'}(\cP,\Delta(S_i))\cong S_i$, see \ref{categoricalequivalnce}. The tops of elements of the base set $\cB(\Lambda')$ are called \emph{inherited simple modules}. Projective covers of inherited simple modules are called \emph{inherited projective modules}. The socles of inherited projective modules are called \emph{inherited socles}.
 \end{definition}

\begin{lemma} \label{lemma 3.3} The number of inherited simple and projective modules is the rank of $\Lambda$.
\end{lemma}
\begin{proof}
The number of inherited simple and projective modules is the number of elements of the base set $\cB(\Lambda')$. By \ref{equalityof some numbers}, $\#\cB(\Lambda')=\rank\Lambda$.
\end{proof}

\begin{proposition}\label{prop 3.3} A projective $\Lambda'$-module is inherited if it has $\cB(\Lambda')$-filtration.
\end{proposition}
\begin{proof}
By definition \ref{definheritance}, $P$ is inherited projective if $P=P(S)$ where $S$ is inherited simple module. Therefore $S\in\cS'(\Lambda')$. On the other hand, $\soc P\in\cS(\Lambda')$ since $P$ is projective module. Any module with top from $\cS'(\Lambda')$ and socle from $\cS(\Lambda')$ has $\cB(\Lambda')$-filtration.
\end{proof}

\begin{proposition}\label{prop 3.5}
The number of nonisomorphic socles of indecomposable $\Lambda'$-projective modules which are not socles of $\cB(\Lambda')$-filtered indecomposable $\Lambda'$-projective modules is the defect of $\Lambda$.
\end{proposition}
\begin{proof}
Notice that the number of socles of $\cB(\Lambda')$-filtered indecomposable $\Lambda'$-projective modules is the number of socles of $\Lambda$-projective modules via $\bm\varepsilon$-construction which is $\#\rel\Lambda$ by \ref{longequality}. The number of all socles of $\Lambda'$-projective modules is $\#\rel\Lambda'$. Therefore the number of new socles of $\Lambda'$ satisfies:
\begin{align*}
\#\text{new socles of}\, \Lambda'&=\#\text{all socles of}\, \Lambda'-\# \text{inherited socles of}\, \Lambda'\\
&= \#\rel\Lambda'-\#\rel\Lambda\\
&=\rank\Lambda-\#\rel\Lambda\\
&= \defect\Lambda
\end{align*}
\end{proof}
\begin{corollary}\label{cor 3.3} The number of projective classes of $\Lambda'$ which are not classes of $\cB(\Lambda')$-filtered $\Lambda'$-projective modules is the defect of $\Lambda$.
\end{corollary}
\begin{proof}
Because the number of socles is equal to number of classes of projective modules following \ref{projectiveclasses}, claim holds.
\end{proof}
\begin{proposition}\label{prop defect}
The number of nonisomorphic indecomposable $\Lambda'$-projective modules which are not $\cB(\Lambda')$-filtered indecomposable $\Lambda'$-projective modules is the defect of $\Lambda$.
\end{proposition}
\begin{proof}
Notice that the number of $\Lambda'$-projective modules which have $\cB(\Lambda')$-filtration is the rank of $\Lambda$ via $\bm\varepsilon$-construction. Therefore the number of new projective modules of $\Lambda'$ satisfies:
\begin{align*}
\#\text{new projectives of} \,\Lambda'&=\#\text{all projectives of}\, \Lambda'-\# \text{inherited projectives of}\, \Lambda'\\
&= \rank\Lambda'-\rank\Lambda\\
&= \defect\Lambda
\end{align*}
\end{proof}

\begin{proposition}\label{newmodules are injective}
Indecomposable projective $\Lambda'$-modules which are not filtered by $\cB(\Lambda')$ are  projective-injective $\Lambda'$-modules. Moreover, their socles are not inherited. They are unique in the sense that each of them are the only projective module in their projective class.
\end{proposition}
\begin{proof}
By the previous propositions, we have the following:
\begin{align*}
\#\text{new projectives of}\, \Lambda'=\#\text{new socles of }\Lambda'=\#\text{new classes of projectives of } \Lambda'.
\end{align*}
By the above equality, each new projective have nonisomorphic socles. Moreover each new projective module are the only projective module of their class. Therefore new projective modules are projective-injective.
\end{proof}
We give a summary of all the results \ref{lemma 3.3},  \ref{prop 3.3}, \ref{prop 3.5}, \ref{cor 3.3}, \ref{prop defect}, \ref{newmodules are injective} below. 
\begingroup
\allowdisplaybreaks
\begin{corollary}\label{misterious numbering1}
If $\Lambda'$ is a higher Auslander algebra satisfying $\bm\varepsilon(\Lambda')\cong \Lambda$ then
\begin{itemize}
\item the number of non-inherited simples of $\Lambda'$
\item the number of non-inherited projectives of $\Lambda'$
\item the number of classes of non-inherited projective $\Lambda'$-modules
\item the number of non-inherited socles of $\Lambda'$-projectives
\item the number of injective but non-projective $\Lambda'$-modules
\end{itemize}
are same. Their cardinalities are simply the defect of $\Lambda$. Each non-inherited projective module is projective-injective.
\end{corollary}
\endgroup

\begin{example} If we drop the condition on the equality of the defects, Proposition \ref{newmodules are injective} cannot be true anymore. Because the number of new projectives modules sharing the same socle can be arbitrarily large. For example, algebras having Kupisch series $(n,n-1,\ldots,3,2,2)$ produces the same syzygy filtered algebra $\Aa_2$. The number of inherited socles and the number of new socles is one, however the number of non-inherited projective modules is $n-2$.
\end{example}


\begin{remark}[Kupisch series]\label{kupisch} If $(c_1,\ldots,c_n)$ is Kupisch series of a connected Nakayama algebra, then we have
\begin{enumerate}[label=\roman*)]
\item $c_i>c_{i+1}$ $\iff$ $P(S_i)\supset P(S_{i+1})$.
\item $c_i=c_{i+1}$ $\iff$ $P(S_i)$ is minimal in its projective class \ref{projectiveclasses} i.e. $\rad P(S_i)$ is not projective and $P(S_{i+1})$ is projective-injective with $\defect P(S_{i+1})=0$.
\item $c_i<c_{i+1}$ $\iff$ $P(S_i)$ is minimal in its projective class \ref{projectiveclasses} and $P(S_{i+1})$ is projective-injective with $\defect P(S_{i+1})=c_{i+1}-c_i$.
\item All $c_{i}\geq 2$ for $1\leq i\leq n-1$. $c_n=1$ $\iff$ algebra is linear.
\end{enumerate}
\end{remark}
\begin{definition}\label{defconsecutive} Simple modules $S$,$S'$ are called consecutive if either $\tau S\cong S'$ or $\tau^{-1}S\cong S'$.
\end{definition}
\begin{notation}\label{defcyclic oredring} If $S_i$, $S_j$ are simple modules satisfying $\tau^{m}S_i\cong S_j$ for some $2\leq m\leq \rank\Lambda-1$, then we use the notation
\begin{align}
S_i> X >S_j
\end{align}
where $X$ is a simple module such that $X\cong \tau^{m_1} S_i\cong \tau^{-m_2}S_j$ for some nonzero $m_1,m_2$ satisfying $m_1+m_2=m$. Similarly
\begin{align}
X_1>X_2\cdots>X_j
\end{align}
means that each $X_i\cong \tau^{m_i} X_{i-1}\cong \tau^{-m_{i+1}} X_{i+1}$ for some nonzero $m_2,\ldots,m_{j-1}$.
\end{notation}

 Let $S_i$ and $S_{i+1}$ be two consecutive simple modules of $\Lambda$ i.e. $\tau S_i\cong S_{i+1}$ ($\tau S_n\cong S_1$) which correspond to modules $\Delta(S_i)$ and $\Delta(S_{i+1})$ in $\Lambda'$ such that
\begin{align}\label{ordering1}
\topp\Delta(S_i)>X>\topp\Delta(S_{i+1})
\end{align}
where $X$ is a simple $\Lambda'$-module which is not inherited from $\Lambda$ i.e. $X\notin\cS'(\Lambda')$. We analyze which modules can be socle of projective cover of $X$ in $\Lambda'$. Because of the Proposition \ref{newmodules are injective}, the socle of projective cover of $X$ cannot be isomorphic to socles of projective covers of $\Delta(S_i)$ and $\Delta( S_{i+1})$ in $\Lambda'$ because $\soc P(X)$ is \emph{not} an inherited socle. Therefore, the socles of the projective covers of those modules satisfy
\begin{align}\label{ordering2}
\soc(P(\topp\Delta(S_i)))>\soc P(X)>\soc(P(\topp\Delta(S_{i+1})))
\end{align}

\begin{claim}\label{claim 1} $P(\topp(\Delta(S_{i+1})))$ is projective-injective $\Lambda'$-module.
\end{claim}
\begin{proof}
\ref{ordering2} holds for any simple $\Lambda'$-module satisfying \ref{ordering1}. Assume to the contrary that $P(\topp(\Delta(S_{i+1})))$ is not projective-injective module, so there exists a projective-injective module $PI$ such that $P(\topp(\Delta(S_{i+1})))\subset PI$. We derive $\topp P(\Delta(S_i))>\topp PI>\topp P(\Delta(S_{i+1}))$ which implies the top of $PI$ is not inherited simple. This contradicts to \ref{ordering2}, because socles of $PI$ and $P(\Delta(S_{i+1}))$ are isomorphic.
\end{proof}
\begin{claim}\label{claim 2} 
$P(S_{i+1})$ has injective $\Lambda$-modules as quotients.
\end{claim}
\begin{proof}
By Remark \ref{submodule of projectives} every projective $\Lambda'$-module has a submodule from the base set $\cB(\Lambda')$. For a module $X$ satisfying \ref{ordering2}, the modules:
\begin{align}
\begin{vmatrix}
\tau\soc P(\topp\Delta(S_{i}))\\
\vdots\\
\soc P(X)
\end{vmatrix},\begin{vmatrix}
\tau\soc P(X)\\
\vdots\\
\soc P(\topp\Delta(S_{i+1}))\\
\end{vmatrix}
\end{align}
have $\cB(\Lambda')$-filtration. Let's denote corresponding $\Lambda$-modules via $\Hom_{\Lambda'}(\cP,-)$ by $M$ and $N$ respectively. Notice that
\begin{gather*}
\soc P(S_i)\ncong \soc M\\
\soc(P(S_{i+1}))\cong \soc N\ncong \soc M
\end{gather*}
in mod-$\Lambda$. Because $S_{i},S_{i+1}$ are consecutive simple $\Lambda$-modules, we can conclude that $\soc M\notin\cS(\Lambda)$ hence the injective envelope $I(M)$ is the quotient of $P(S_{i+1})$. We get the exact sequence
 \begin{align}
0\rightarrow N\hookrightarrow P(S_{i+1})\twoheadrightarrow \faktor{P(S_{i+1})}{N}\cong I(M)\rightarrow 0.
 \end{align}
\end{proof}

In other words socles of $P(S_i)$ and $P(S_{i+1})$ cannot be consecutive and the injective envelope of $\tau$-translate of socle of $P(S_i)$ is an injective $\Lambda$-module which is not projective and it is the quotient of $P(S_{i+1})$. The crucial conclusion is that the modules which we can extend has to be inverse translate of top of injective non-projective modules. Equivalently, $\Delta(S_i)$ is not a simple $\Lambda'$-module if and only if $P(\tau S_i)$ has injective quotients.

\begin{definition}\label{def defect of projective modules} [Defect of projective-injective modules] We define the defect of indecomposable projective $\Lambda$-module $P$ as the number of distinct nonisomorphic injective quotients of $P$ i.e.
\begin{align}\label{defect of projective}
\defect(P)=\#\text{injective quotients of } P.
\end{align}
\end{definition}
It is clear that defect is zero for projective modules which are not injective. It satisfies
\begin{align}
\defect\Lambda=\sum\limits_{P\text{ is projective}}\defect(P)=\sum\limits_{\substack{P\, \text{is projective}\\\text{-injective}}}\defect(P).
\end{align}

\begin{lemma}\label{delta is simple iff} Let $S_i,S_{i+1}$ be two consecutive simple $\Lambda$-modules. The inherited simple modules $\topp\Delta(S_i)$ and $\topp\Delta(S_{i+1})$ of $\Lambda'$ are consecutive i.e. $\tau\topp\Delta(S_i)\cong\topp\Delta(S_{i+1})$ if and only if the defect of $P(S_{i+1})$ is zero. Equivalently, $\Delta(S_i)$ is a simple $\Lambda'$-module if and only if $\defect(P(S_{i+1}))=0$.
\end{lemma}
\begin{proof}
We start with only if statement, zero defect of $P(S_{i+1})$ implies $\Delta(S_i)$ is simple $\Lambda'$-module. Assume to the contrary that there exists at least one simple $\Lambda'$-module $X$ such that 
\begin{align}
\topp\Delta(S_i)>X>\topp\Delta( S_{i+1})
\end{align}
where $X\notin\cS'(\Lambda')$. Because $\defect(P(S_{i+1}))=0$, if we denote the socles of $P(S_i)$ and $P(S_{i+1})$ by $T_i$ and $T_{i+1}$ respectively, either $T_i\cong T_{i+1}$ i.e. $P(S_{i+1})\subset P(S_i)$ or they are consecutive. The former is impossible because it forces that $P(S_{i+1})\subset P(X)\subset P(S_i)$ which violates the projective-injectivity of $P(X)$ obtained in Proposition \ref{newmodules are injective}.
So we focus on the latter. By Remark \ref{submodule of projectives} every projective $\Lambda'$-module has a submodule which is an element of the base set $\cB(\Lambda')$. On the other hand $P(X)$ is projective-injective $\Lambda'$-module by Proposition \ref{newmodules are injective} and its socle is not inherited from $\Lambda$ i.e.
\begin{align*}
\begin{gathered}
\soc P(X)\ncong \soc P(\Delta(S_{i}))=\soc\Delta(T_{i})\\
\soc P(X)\ncong \soc P(\Delta(S_{i+1}))=\soc\Delta(T_{i+1}).
\end{gathered}
\end{align*}
Because $T_i,T_{i+1}$ are simple modules, by \ref{categoricalequivalnce}, there exists $\Delta(T_i),\Delta(T_{i+1})\in\cB(\Lambda')$.
We get
\begin{align}
\soc(\Delta(T_i))>\soc P(X)>\soc(\Delta(T_{i+1}))
\end{align}
This forces that $\Delta(T_{i+1})$ is not a simple module, otherwise $\tau\soc(\Delta(T_i))\cong \Delta(T_{i+1}) $ and then $\soc P(X)\cong \Delta(T_{i+1})$, the socle of $P(X)$ is not inherited.
Now $\Delta(T_{i+1})$ is not simple, so there exists an exact sequence
\begin{align}
0\rightarrow N\rightarrow \Delta(T_{i+1})\rightarrow M\rightarrow 0.
\end{align}
where $M$ is submodule of $P(X)$.
Notice that $\topp M\cong\topp\Delta(T_{i+1})\in \cS'(\Lambda')$ and $\soc N=\soc\Delta(T_{i+1})\in\cS(\Lambda')$. However $\soc M\cong \soc P(X)$ implies that $\soc M\in\cS(\Lambda')$ and $\topp N\cong \tau\soc M\in\cS'(\Lambda')$, which makes $\Delta(T_{i+1})\notin\cB(\Lambda')$ and therefore creates contradiction. This means $\tau\topp\Delta(S_i)\cong \topp\Delta(S_{i+1})$, and in particular $\Delta(S_i)$ is simple module.
\par Now we prove that if $\Delta(S_i)$ is a simple $\Lambda'$-module, then $\defect(P(S_{i+1}))=0$. Consider the projective resolution of $\Delta(S_i)$
\begin{align}
\cdots\rightarrow P(\Delta(S_{i+1}))\rightarrow P(\Delta(S_i))\rightarrow \Delta(S_i)\rightarrow 0.
\end{align}
Because each of the modules are $\cB(\Lambda')$-filtered, we can apply the functor $\Hom_{\Lambda'}(\cP,-)$ to get the projective resolution
\begin{align}
\cdot\rightarrow P(S_{i+1})\rightarrow P(S_i)\rightarrow S_i\rightarrow 0
\end{align}
in mod-$\Lambda$.
$\Omega^2(\Delta(S_i))$ is an element of $\cB(\Lambda')$ by Remark \ref{secondsyzygy}, therefore $\Hom(\cP,\Omega^2(\Delta(S_i)))$ is simple $\Lambda$-module. The length of $P(S_{i+1})$ is $1+\ell(\rad P(S_{i}))=\ell P(S_i)$. So the defect of $P(S_{i+1})$ is zero by the properties of Kupisch series \ref{kupisch}.
\end{proof}

\begin{proposition}\label{Structure of cyclic ordering of simple modules}[Structure of cyclic ordering of simple modules] 
Let $S_i$ and $S_{i+1}$ be two consecutive simple modules of $\Lambda$ which are inherited to $\Lambda'$ as $\topp\Delta(S_i)$ and $\topp\Delta(S_{i+1})$. The number of simple $\Lambda'$-modules which can be inserted between $\topp\Delta(S_i)$ and $\topp\Delta(S_{i+1})$ without leaving any gap is the defect of $P(S_{i+1})$. Equivalently, the length of $\Delta(S_i)$ is $\defect(P(S_{i+1}))+1$.
\end{proposition}

\begin{proof}
The case $\defect(P(S_{i+1}))=0$ follows from the lemma \ref{delta is simple iff}. Nevertheless we recall its proof. $\defect(P(S_{i+1}))=0$ means the modules $\soc P(S_i)$ and $\soc P(S_{i+1})$ are consecutive in $\Lambda$. Therefore, if there is a simple module $X$ satisfying $\topp\Delta(S_{i})>X>\topp\Delta(S_{i+1})$, then the socle of $P(X)$ is either isomorphic to $\soc P(\Delta(S_i))$ or $\soc P(\Delta(S_{i+1}))$. Both cases violates the condition about socles we derived in \ref{newmodules are injective}, hence $\topp\Delta(S_{i})\cong\Delta(S_i),\topp\Delta(S_{i+1})$ are two consecutive simple $\Lambda'$-modules.
\par Now we study the case  $\defect(P(S_{i+1}))=d(i)\neq 0$. Assume that there are $t(i)$-many simple $\Lambda'$-modules satisfying:
\begin{gather}\label{ordering3}
\topp P(\Delta(S_i))>X_1>X_2>\cdots>X_{t(i)}>\topp P(\Delta(S_{i+1}))
\end{gather}
By proposition \ref{newmodules are injective}, each  $P(X_j)$, $1\leq j\leq t(i)$ should have nonisomorphic socles. Furthermore, by the observation \ref{ordering2}, the number of all candidates for socles is $d(i)=\defect(P(S_{i+1}))$. The case $t(i)>d(i)$ is impossible because of the pigeonhole principle. In details, if $c_i,c_{i+1}$ are the lengths of $P(S_i),P(S_{i+1})$ in $\Lambda$ respectively, then $c_{i+1}=c_i+d(i)$. Moreover, $c_i,c_{i+1}$ are the $\cB(\Lambda')$-lengths of $P(\Delta(S_i))$ and $P(\Delta(S_{i+1}))$ by \ref{b length}. Observe that the $\cB(\Lambda')$-length of the indecomposable module with top $\topp\Delta(S_{i+1})$ and socle $\soc P(\Delta(S_{i}))$ is $c_i-1$ because it is the $\cB(\Lambda')$-radical of $P(\Delta(S_i))$. $\cB(\Lambda')$-length of the module $M$ with top $\tau\soc P(\Delta(S_i))$ and socle $\soc P(\Delta(S_{i+1}))$ is $t(i)+1$. This crucial observation follows from the fact that each $\cB(\Lambda')$-socle of $P(X_i)$ appears in the $\cB(\Lambda')$-composition series of $M$.  Therefore the $\cB(\Lambda')$-length of $P(\Delta(S_{i+1}))$ and in particular $\Lambda$-length of $P(S_{i+1})$ has to be $c_i+t(i)$. This is impossible because $c_{i+1}=c_i+d(i)<c_i+t(i)=c_{i+1}$ is not true. \\

\par Therefore we conclude that $t(i)\leq d(i)$ for all $i$. Without loss of generality, assume to the contrary that $t(1)<d(1)$. 
If we add up all terms we get
\begin{gather} \label{some eq in }
\sum_i d(i)>\sum_i t(i).
\end{gather}
The Theorem \ref{thmdefect} implies that $\defect\Lambda'=\defect\Lambda$, therefore:
\begin{gather} 
\defect\Lambda'=\defect\Lambda=\sum_{P}\defect(P)=\sum_i\defect(P(S_{i+1}))=\sum_i d(i)
\end{gather}
On the other hand the number of simple $\Lambda'$-modules and in particular the number of projective modules which are non-inherited from $\Lambda$ is 
\begin{gather}
\#\text{new simples}=\#\text{new projectives}=\sum_i t(i)
\end{gather}
Proposition \ref{prop defect} states that the defect of $\Lambda'$ is the number of simple $\Lambda'$-modules which are not inherited from $\Lambda$, we get equality
\begin{gather} 
\defect\Lambda'=\sum_i d(i)=\sum _i t(i)=\#\text{new simples}
\end{gather}
which contradicts to \ref{some eq in }.
As a result $t(i)=d(i)$ for all $i$. 
\end{proof}

\begin{corollary}\label{length of deltas} $\defect(P(S_{i+1}))=m\iff \ell\Delta(S_i)=m+1$.
\end{corollary}
\begin{proof}
The case where $m=1$ follows from the lemma \ref{delta is simple iff}. The remaining cases follows from the proof of Proposition \ref{Structure of cyclic ordering of simple modules}, because $t(i)=d(i)$ for all $i$.
\end{proof}

The above result \ref{length of deltas} determines the structure of the base set uniquely for a given algebra $\Lambda$. 

\begin{corollary}\label{lengths of delta proj} Length of inherited projective $\Lambda'$-module $P(\topp\Delta(S_i))$ is 
\begin{align}
\ell P\topp(\Delta(S_i))=\sum\limits^{\ell P(S_i)+i}_{j=i} \ell\Delta(S_j)=\sum\limits^{\ell P(S_i)+i}_{j=i} \defect(P(S_{j+1}))+1
\end{align}
\end{corollary}

\begin{proof}
 The length of $P(\topp\Delta(S_i))$ is the sum of the lengths of $\cB(\Lambda)$-filtered composition factors. By corollary \ref{length of deltas} we get lengths of elements of the base set $\cB(\Lambda')$ in terms of the defects of projective $\Lambda$-modules. Summation of all these terms provides the length of $P(\Delta(S_i))$ .
\end{proof}

\begin{proposition}\label{def prop4} $\Lambda'$ is well-defined, its syzygy filtered algebra is $\Lambda$ and it is a higher Auslander algebra.
\end{proposition}
\begin{proof}
It is enough to find the lengths of indecomposable projective $\Lambda'$-modules because $\Lambda'$ is Nakayama algebra. By corollary \ref{lengths of delta proj} we get the lengths of $\cB(\Lambda')$-filtered projective modules which are inherited projectives. We need to detect the lengths of  projective modules which are not inherited.
Consider simple $\Lambda'$-modules
\begin{align}\label{exc seqq}
\topp\Delta(S_i)>X_1>\cdots>X_t>\topp\Delta(S_{i+1})
\end{align}
where $X_j$'s are non-inherited simple $\Lambda'$-modules i.e. $X_j\notin\cS'(\Lambda')$. If we denote the simple $\Lambda'$-module $\topp\Delta(S_i)$ by $\tilde{S_{i}}$, then by the Proposition \ref{Structure of cyclic ordering of simple modules}, we obtain the structure of $\Delta(S_i)$
\begin{align}\label{exc seqq sec}
 \Delta(S_i)\cong\begin{vmatrix}
\tilde{S_i}\\X_1\\\vdots\\X_t
\end{vmatrix}
\end{align}

 and $\defect(P(S_{i+1}))=t$.
All $P(X_j)$ contains the module $\begin{vmatrix}
\Delta(S_{i+1})\\\vdots\\\Delta(T_i)
\end{vmatrix}$ as a subquotient where $\Delta(T_i)$ is the submodule of $P(\Delta(S_i))$. Moreover each $P(X_j)$ has a submodule $\Delta(T_{i+j})$ which corresponds to socle of injective module $I(T_{i+j})$ in mod-$\Lambda$. Therefore the lengths of not inherited projective modules $P(X_j)$ $1\leq j\leq t$  are given by
\begin{align}\label{sth ccc}
t-j+1+\ell\left(\begin{vmatrix}
\Delta(S_{i+1})\\\vdots\\\Delta(T_i)
\end{vmatrix}\right)+\ell\left(\begin{vmatrix}
\Delta(T_{i+1})\\\vdots\\\Delta(T_{i+j})
\end{vmatrix}\right).
\end{align}
By corollary \ref{length of deltas}, it can be expresses in terms of the defects of $\Lambda$-modules. This describes the algebra $\Lambda'$, because we know the length of each projective $\Lambda'$-module. However we  do not want to explicitly write its Kupisch series to avoid messy notation. \\

 The second part follows from the definition \ref{defep}. By the Proposition \ref{Structure of cyclic ordering of simple modules} and corollary \ref{length of deltas} we know the structure of the elements of the base set $\cB(\Lambda')$
\begin{gather*}
\Delta(S_i)\, \text{is simple} \iff \defect(P(S_{i+1}))=0\\ \ell(\Delta(S_i))=t+1 \iff \defect(P(S_{i+1}))=t.
\end{gather*}
If $\Delta(S_j)$ is simple $\Lambda'$-module, let's denote it by $\tilde{S_j}=\Delta(S_j)$. If $\Delta(S_i)$ is not simple, it has the structure given in \ref{exc seqq sec}, so we use the notation $\tilde{S_i}=\topp\Delta(S_i)$.
Therefore the sets given in definition \ref{defbase} are of the form
\begin{align*}
\cB(\Lambda')=&\left\{\Delta(S_1),\Delta(S_2),\ldots,\Delta(S_n)\right\}\\
\cS'(\Lambda')=&\left\{ \topp\Delta(S_1),\topp\Delta(S_2),\ldots,\topp\Delta(S_n)\right\}=\left\{\tilde{S_1},\tilde{S_2}\ldots,\tilde{S_n}\right\}\\
\cS(\Lambda')=&\left\{\soc\Delta(S_1),\soc\Delta(S_2),\ldots,\soc\Delta(S_n)\right\}=\left\{\tau^{-1}\tilde{S_1},\tau^{-1}\tilde{S_2}\ldots,\tau^{-1}\tilde{S_n}\right\}.
\end{align*}

$\cB(\Lambda')$-filtered projective modules are exactly the inherited projective modules, therefore by the definition of syzygy filtered algebra \ref{defep} we get
\begin{align*}
\bm{\varepsilon}(\Lambda'):=\End_{\Lambda'}\left(\bigoplus\limits_{S\in \cS'(\Lambda')}P(S)\right)=\End_{\Lambda'}\left(P(\tilde{S_1})\oplus\ldots\oplus P(\tilde{S_n})\right)
\cong \Lambda.
\end{align*}

Simple $\Lambda$-modules are $S_1,S_2,\ldots,S_n$ which follows from 
\begin{align*}
\Hom_{\Lambda'}\left(\bigoplus\limits^n_{i=1} P(\tilde{S_i}),\tilde{S_i}\right)=\Hom_{\Lambda'}\left(\bigoplus\limits^n_{i=1} P(\Delta(S_i)),\Delta(S_i)\right)=S_i.
\end{align*} Similarly, projective $\Lambda$-modules are 
\begin{align*}
\Hom_{\Lambda'}\left(\bigoplus\limits^n_{i=1} P(\tilde{S_i}),P(\tilde{S_i})\right)=\Hom_{\Lambda'}\left(\bigoplus\limits^n_{i=1} P(\Delta(S_i)),P(\Delta(S_i))\right)=P(S_i).
\end{align*}\\

\par By lemma \ref{every second syzy of injective}, the first syzygy of any injective $\Lambda'$-module is a proper submodule of an element of the base set. By the Proposition \ref{Structure of cyclic ordering of simple modules} and corollary \ref{length of deltas} $\Delta(S_i)$ is not simple if and only if the defect of $P(S_{i+1})$ is nonzero. Moreover, the projective covers of proper submodules of $\Delta(S_i)$ are projective-injective modules which are unique in their class by Proposition \ref{newmodules are injective}.  On the other hand, there exists a projective $\Lambda'$-module having $\Delta(S_i)$ as a submodule by Remark \ref{submodule of projectives}. Now we can construct the projective resolution of any injective $\Lambda'$-module. By using the exact sequence
\begin{align}
0\rightarrow \Omega^1(I)\rightarrow P(I)\rightarrow I\rightarrow 0
\end{align}
together with lemma \ref{every second syzy of injective} we get $\Omega^1(I_{\Lambda})\cong\begin{vmatrix}
X_j\\\vdots\\ X_t
\end{vmatrix}$ for some $j$, $\topp\Delta(S_i),\topp\Delta(S_{i+1})$ satisfying \ref{exc seqq}.  $P(X_j)$ and $P(\Delta(S_{i+1}))$ are projective covers of $\Omega^1(I_{\Lambda})$ and $\tau X_t\cong\topp\Delta(S_{i+1})$ respectively, hence we get the first terms of the projective resolution of $I$.
\begin{align}
\cdots\rightarrow P(\Delta(S_{i+1}))\rightarrow P(X_j)\rightarrow P(I)\rightarrow I\rightarrow 0.
\end{align}
$\Omega^2(I)$ is proper quotient of $P(\Delta(S_{i+1}))$, because $P(\Delta(S_{i+1}))$ is not a submodule of $P(X_j)$ by \ref{sth ccc}. By Remark \ref{secondsyzygy}, $\Omega^2(I)$ has $\cB(\Lambda')$-filtration and $P(\Delta(S_{i+1}))$ is inherited projective module, therefore $\Hom_{\Lambda'}(\cP,\Omega^2(I))$ is injective $\Lambda$-module which is the quotient of $P(S_{i+1})$. This shows that for any injective non-projective $\Lambda'$-module, there exist injective non-projective $\Lambda$-module via the the functor $\Hom_{\Lambda'}(\cP,\Omega^2(-))$.\\

 We need to show that each injective quotient of $P(S_{i+1})$ corresponds to unique injective $\Lambda'$-module via the functor $\Hom_{\Lambda'}(\cP,\Omega^2(-))$.\\

Let $\defect(P(S_{i+1}))=t$ and $I_1,\ldots,I_t$ be injective quotients of $P(S_{i+1})$ where the indices show just enumeration not  the injective envelope of simple module having that index. Let $\Delta I_1,\ldots,\Delta I_t$ be the corresponding $\Lambda'$-modules i.e. $\Hom_{\Lambda'}(\cP,\Delta I_i)\cong I_i$. By the construction we used in the proof of the Proposition \ref{Structure of cyclic ordering of simple modules} and \ref{exc seqq} each $\Delta I_j$ is submodule of not inherited projective module $P(X_j)$ which is projective-injective and unique in its class \ref{projectiveclasses}. Let $M_j$ denote the quotient $\faktor{P(X_j)}{\Delta I_j}$. Notice that $M_j\ncong M_{j'}$ when $j\neq j'$. Each $M_j$ is submodule of $\Delta(S_i)$, and by the lemma \ref{injectivity}, submodule of the unique projective-injective module $P$, so
\begin{align*}
M_j\varlonghookrightarrow \Delta(S_i)\varlonghookrightarrow P 
\end{align*}
Each of the quotient $\faktor{P}{M_j}$ is injective $\Lambda'$-module, moreover $\faktor{P}{M_j}\ncong\faktor{P}{M_{j'}}$ because socles are different i.e. $\soc\faktor{P}{M_j}\cong \tau^{-1}X_j\ncong \tau^{-1}X_{j'}\cong\soc\faktor{P}{M_{j'}}$. therefore $\Hom_{\Lambda'}(\cP,\faktor{P}{M_j})\cong I_j$.
This shows for each injective $\Lambda$-module which is the quotient of $P(S_{i+1})$, there exists one injective $\Lambda'$-module. When we vary $i$, for any injective $\Lambda$-module, there exists at least one injective $\Lambda'$-module.

By Proposition \ref{different injectives}, if there are more than one injective $\Lambda'$-module which maps to the same injective $\Lambda$-module via $\Hom_{\Lambda'}(\cP,\Omega^2(-))$, then the socle of the projective covers of the first syzygies are same, however the tops are not isomorphic. Moreover the tops  are not inherited simple modules by the structure of non-simple elements of $\cB(\Lambda')$ (see \ref{exc seqq sec}). By using the Proposition \ref{newmodules are injective}, any projective cover of the first syzygy is unique not inherited projective module in its class, therefore this is not possible.

There is another case that we need to analyze: what happens if the first syzygy of an injective $\Lambda'$-module is projective. If $\Omega^1(I)$ is projective, by lemma \ref{every second syzy of injective}, it is a proper submodule of some $\Delta(S_i)$ which makes $\Delta(S_i)$ a projective module. Its $\cB(\Lambda')$-length and the length of the corresponding simple module $\Hom_{\Lambda'}(\cP,\Delta(S_i))$ are one. Since $\Lambda$ is cyclic, all indecomposable projective $\Lambda$-modules have the length at least two. Hence $\Omega^1(I)$ cannot be projective for any injective $\Lambda'$-module $I$.

Now we can adjust Proposition \ref{GorensteinReduction} accordingly. For any injective non-projective $\Lambda$-module $I$, there is a unique injective non-projective $\Lambda$-module $I_{\Lambda'}$ such that
\begin{align*}
\Hom_{\Lambda'}(\cP,\Omega^2(I_{\Lambda'}))\cong I.
\end{align*}
 Therefore the defects of $\Lambda$ and $\Lambda'$ are equal. Consider the projective resolution 
\begin{align}\label{bir isms}
\cdots\rightarrow P(\Omega^1(I_{\Lambda'}))\rightarrow P(I_{\Lambda'})\rightarrow I_{\Lambda'}\rightarrow 0.
\end{align}
$P(I_{\Lambda'})$ has to be projective-injective module since it is a projective cover of an injective. By combining the lemma \ref{every second syzy of injective} with \ref{exc seqq sec} the projective module $P(\Omega^1(I))$ is not inherited and by Proposition \ref{newmodules are injective} any non-inherited projective $\Lambda'$-module is projective-injective. Therefore we get
\begin{align}
\domdim I_{\Lambda'}=2+\domdim\Omega^2(I_{\Lambda'}).
\end{align}
By Remark \ref{secondsyzygy}, any syzygy $\Omega^i(I)$ and projective module $P(\Omega^i(I))$ have $\cB(\Lambda')$-filtration for $i\geq 2$. We can apply the functor $\Hom_{\Lambda'}(\cP,-)$ to \ref{bir isms} in order  to get 
\begin{align*}
\domdim_{\Lambda'}\Omega^2(I_{\Lambda'})=\domdim_{\Lambda}\Hom_{\Lambda'}(\cP,\Omega^2(I_{\Lambda'})).
\end{align*}
$\Hom_{\Lambda'}(\cP,\Omega^2(I_{\Lambda'}))$ is injective $\Lambda$-module which we denote it by $I$. In summary, we get
\begin{align*}
\domdim I_{\Lambda'}&=2+\domdim\Omega^2(I_{\Lambda'})\\
&=2+\domdim_{\Lambda}\Hom_{\Lambda'}(\cP,\Omega^2(I_{\Lambda'}))\\
&=2+\domdim_{\Lambda} I
\end{align*}
for any injective $\Lambda'$-module. By the characterization \ref{gercek tanim} of the dominant dimension  we conclude that 
\begin{align}
\begin{split}
\domdim\Lambda'&=\sup\left\{\domdim I_{\Lambda'}\,\vert\, I_{\Lambda'} \text{ is injective non-projective } \Lambda'\, \text{module}\right\}\\
&=\sup\left\{2+\domdim I\,\vert\, I \text{ is injective non-projective } \Lambda\, \text{module}\right\}\\
&=2+\sup\left\{\domdim I\,\vert\, I \text{ is injective non-projective } \Lambda\, \text{module}\right\}\\
&=2+\domdim\Lambda.
\end{split}
\end{align}
Because $\gldim\Lambda$ is finite and $\Lambda\cong\bm\varepsilon(\Lambda')$, we obtain that 
$\gldim\Lambda'=2+\gldim\Lambda$. $\Lambda$ is a higher Auslander algebra so $\gldim\Lambda=\domdim\Lambda$. Therefore $\Lambda'$ is also a higher Auslander algebra because 
\begin{align}
\domdim\Lambda'=2+\domdim\Lambda=2+\gldim\Lambda=\gldim\Lambda'.
\end{align}
\end{proof}

\begin{proposition}\label{prop uniqu} $\Lambda'$ is unique.
\end{proposition}
\begin{proof} Let $\Lambda'$ and $\Lambda''$ be two cyclic Nakayama algebras satisfying $\bm\varepsilon(\Lambda')\cong\bm\varepsilon(\Lambda'')\cong \Lambda$. We have the following key observations:
\begin{itemize}
\item By \ref{categoricalequivalnce} the category of $\cB(\Lambda')$-filtered $\Lambda'$-modules is equivalent to the category of $\cB(\Lambda'')$-filtered $\Lambda''$-modules.
\item If $\Lambda'$ and $\Lambda''$ are higher Auslander algebras then their rank is same which is $\rank\Lambda'=\rank\Lambda''=\rank\Lambda+\defect\Lambda$. Because $\Lambda$ is fixed, upto permutation induced by $\tau$ on simple modules the base sets $\cB(\Lambda')$ and $\cB(\Lambda'')$ are same because of the Proposition \ref{Structure of cyclic ordering of simple modules} and corollary \ref{length of deltas}. This is much stronger than the categorical equivalence.
\item More importantly, the number of projective $\Lambda'$ and $\Lambda''$-modules sharing the same socle is same. If $P_i\supset P_{i+1}\supset\cdots\supset P_j$ have the socle $T_i$ in mod-$\Lambda$, then the inherited projectives have the same socle $\soc\Delta'(T_i)$ in mod-$\Lambda'$ and $\soc\Delta''(T_i)$ in mod-$\Lambda''$. And by the Proposition \ref{newmodules are injective}, the remaining projective modules (non-inherited) are projective-injective i.e. unique projective modules having submodules $\Delta'(T_j)$'s in $\Lambda'$ and $\Delta''(T_j)$'s in $\Lambda''$ which correspond to socles of injective but non-projective $\Lambda$-modules $T_j$'s. Therefore we know the distribution of socles to projective modules in both $\Lambda'$ and $\Lambda''$. In other words, number of classes of projective modules are equal (see \ref{projectiveclasses}) and the number of projective modules in a class is same for both $\Lambda'$ and $\Lambda''$.
\end{itemize}

There are infinitely many algebras satisfying the items in the above list and they are parametrized by the lengths of projective modules. In details, If $(c_1,\ldots,c_n)$ is Kupisch series of an algebra $\Lambda'$ then each of the $(c_1+j,\ldots,c_n+j)$ for $j\neq 0$ can be $\Lambda''$. The increment or decrement $j$ has to be same for all $i$, otherwise it would change either the structure of $\cB(\Lambda')$ or the distribution of projective modules into projective classes and the number of projectives in the projective class. However, even for the same increment or decrement $j\neq 0$, $\bm\varepsilon(\Lambda'')\ncong\bm\varepsilon(\Lambda')$. Since $\Lambda$ is fixed, the length of inherited projective modules can take only one value which comes from the unique structure of the base set $\cB(\Lambda')$ via the corollary \ref{lengths of delta proj}. This finishes the proof.
\end{proof}

\begin{remark} If we relax the conditions in the proof above, we loose the uniqueness.
\begin{itemize}
\item Two algebras can produce the same syzygy filtered algebra, but the lengths of modules in their base sets cannot be same. For instance the base set of the algebras $(n,n-1,\ldots,4,3,3,3)$ is
\begin{align*}
\cB(\Lambda)=\left\{\Delta_1\cong S_1,\Delta_2\cong S_2,\Delta_3\cong S_3, \Delta_4\cong\begin{vmatrix}
S_{4}\\\vdots\\S_n
\end{vmatrix}\right\}
\end{align*}
where $\Delta_4$ produces different partition while $n$ varies. 
\item Projective covers of not inherited simple modules has to be projective-injective. If we drop the second and third conditions  we can get $\bm\varepsilon(\Lambda')\cong\Lambda$ but $\Lambda'$ is not higher Auslander algebra anymore. Here is an example.
Let $(3,2,2)$ be Kupisch series of cyclic Nakayama algebra $\Lambda$. $(4,3,3,3)$ is higher Auslander algebra, $(5,4,4,4,4)$ and $(5,4,3,3,3)$ are not, even though the number of projective classes are equal.
 This happens because we changed the number of projective modules belonging to the same projective class.
\item Third item is not enough by itself. For example if $\Lambda$ is given by $(3,2,2)$, then simple modules of $\Lambda'$ are $1>2>3>x>1$. So $(4,3,3,3)$ $(4,3,3,4)$ are possible. But the latter affects the defect of $P_1\in$mod-$\Lambda$, and therefore the filtration.
\item If $n\geq 3$ all algebras $(n+1,n,n,n)$ have the same base set upto permutation on simple modules induced by $\tau$. However only one of them is syzygy equivalent to $(3,2,2)$. The length of $\Lambda$-modules is decisive also.
\end{itemize}
\end{remark}

\section{Proof of Theorem \ref{thmreverseepsilon2}}\label{section linear}
The construction of cyclic Nakayama algebra from a linear one is similar to what we did in previous section. First we assume the existence of cyclic $\Lambda'$ satisfying $\bm\varepsilon(\Lambda')\cong\Lambda$. Then we will show $\Lambda'$ is unique. However we need some slight modifications. First we start with a definition.
\begin{definition}\label{nakayamacycle}
A Nakayama cycle $(\Lambda,\tau)$ is given by an algebra $\Lambda =\Lambda_1\times ... \times \Lambda_t$ with connected linear Nakayama algebras $\Lambda_1,\ldots,\Lambda_t$ and the cyclic permutation $\tau$ of the simple $\Lambda$-modules, such that the restriction to the simple $\Lambda_i$-modules is the Auslander-Reiten
translation for the simple $\Lambda_i$-modules, and $\tau$ maps the simple projective  $\Lambda_i$-module to the simple injective $\Lambda_{i-1}$-module (with $\Lambda_0 =\Lambda_t$).
\end{definition}
In the above definition, all components share the same global and dominant dimensions. If there exists a higher Auslander algebra $\Lambda'$ satisfying $\bm\varepsilon(\Lambda')\cong (\Lambda,\tau)$, then either all components are linear algebras or $(\Lambda,\tau)\cong\oplus_m\Aa_1$. Therefore, semisimple and linear components cannot appear together (see Remark \ref{remarklis1}, Example \ref{example dominant dimension}) for Nakayama cycles which are higher Auslander.
\begingroup
\allowdisplaybreaks
\begin{remark}\label{def relations liner} Let $\Lambda$ be a linear Nakayama algebra of rank $n$. The irredundant system of relations is
\begingroup
\allowdisplaybreaks
\begin{gather}\label{relrel3}
\begin{gathered}
\boldsymbol\alpha_{k_2}\ldots\boldsymbol\alpha_{k_1+1}\boldsymbol\alpha_{k_1}\ \ =0 \\
\boldsymbol\alpha_{k_4}\ldots\boldsymbol\alpha_{k_3+1}\boldsymbol\alpha_{k_3}\ \ =0 \\
\vdots\\
\boldsymbol\alpha_{k_{2r-2}}\ldots\boldsymbol\alpha_{k_{2r-3}+1}\boldsymbol\alpha_{k_{2r-3}}=0\\
\boldsymbol\alpha_{k_{2r}}\ldots\boldsymbol\alpha_{k_{2r-1}+1}\boldsymbol\alpha_{k_{2r-1}}=0\\
\boldsymbol\alpha_n=0
\end{gathered}
\end{gather}
\endgroup
where $1\leq k_1< k_3< \cdots<k_{2r-1}$ and $k_2<k_4<\ldots<k_{2r}<n$ for the linear quiver. To make the system irredundant $n$ has to be greater than $k_{2r}$. Notice that the relation for simple projective module $P_n$ is just $\alpha_n=0$. In the case of Nakayama cycles we have a similar description of the irredundant system of relations. Namely, If $L_1,\ldots,L_t$ are linear Nakayama algebras with irredundant system of relations $\rel_1,\ldots,\rel_k$, by shifting the indices of projective modules with respect to the corresponding ranks will be the irredundant system. This means, if $\bm\alpha_{k_{2j}}\cdots\bm\alpha_{k_{2j-1}}=0$ appears in the relations of $L_d$, then it becomes $\bm\alpha_{k_2j+R}\cdots\bm\alpha_{k_{2j-1}+R}=0$ in the relations defining $L_1\oplus\cdots \oplus L_{d}\oplus\ldots \oplus L_t$ where $R=\sum\limits^{d-1}_{i=1}\rank L_i$ for any $2\leq d\leq t$. By this, it is clear that whenever we add another linear algebra to $L_1\oplus\ldots\oplus L_t$, we just need to shift the indices by the total rank of the previous ones.

 To avoid any notational mess, we only give complete  description up to $t=3$. Let $L_1$, $L_2$ and $L_3$ be given by the relations \ref{relrel3}, \ref{relrel1} and \ref{relrel2} respectively.

\begin{minipage}{0.5\textwidth}
\begin{gather}\label{relrel1}
\begin{gathered}
\boldsymbol\alpha_{k'_2}\ldots\boldsymbol\alpha_{k'_1+1}\boldsymbol\alpha_{k'_1}\ \ =0 \\
\boldsymbol\alpha_{k'_4}\ldots\boldsymbol\alpha_{k'_3+1}\boldsymbol\alpha_{k'_3}\ \ =0 \\
\vdots\\
\boldsymbol\alpha_{k'_{2r}}\ldots\boldsymbol\alpha_{k'_{2r-1}+1}\boldsymbol\alpha_{k'_{2r-1}}=0\\
\boldsymbol\alpha_{n'}=0
\end{gathered}
\end{gather}
\end{minipage}
\bigskip
\begin{minipage}{0.5\textwidth}
\begin{gather}\label{relrel2}
\begin{gathered}
\boldsymbol\alpha_{k''_2}\ldots\boldsymbol\alpha_{k''_1+1}\boldsymbol\alpha_{k''_1}\ \ =0 \\
\boldsymbol\alpha_{k''_4}\ldots\boldsymbol\alpha_{k''_3+1}\boldsymbol\alpha_{k''_3}\ \ =0 \\
\vdots\\
\boldsymbol\alpha_{k''_{2r}}\ldots\boldsymbol\alpha_{k''_{2r-1}+1}\boldsymbol\alpha_{k''_{2r-1}}=0\\
\boldsymbol\alpha_{n''}=0
\end{gathered}
\end{gather}
\end{minipage}
\bigskip
For the algebra $L_1\oplus L_2$, the relations are
\begingroup
\allowdisplaybreaks
\begin{gather*}
\boldsymbol\alpha_{k_2}\ldots\boldsymbol\alpha_{k_1+1}\boldsymbol\alpha_{k_1}\ \ =0 \\
\boldsymbol\alpha_{k_4}\ldots\boldsymbol\alpha_{k_3+1}\boldsymbol\alpha_{k_3}\ \ =0 \\
\vdots\\
\boldsymbol\alpha_{k_{2r}}\ldots\boldsymbol\alpha_{k_{2r-1}+1}\boldsymbol\alpha_{k_{2r-1}}=0\\
\boldsymbol\alpha_n=0\\
\boldsymbol\alpha_{k'_2+\rank(L_1)}\ldots\boldsymbol\alpha_{k'_1+1+\rank(L_1)}\boldsymbol\alpha_{k'_1+\rank(L_1)}\ \ =0 \\
\boldsymbol\alpha_{k'_4+\rank(L_1)}\ldots\boldsymbol\alpha_{k'_3+1+\rank(L_1)}\boldsymbol\alpha_{k'_3+\rank(L_1)}\ \ =0 \\
\vdots \\
\boldsymbol\alpha_{k'_{2r}+\rank(L_1)}\ldots\boldsymbol\alpha_{k'_{2r-1}+1+\rank(L_1)}\boldsymbol\alpha_{k'_{2r-1}+\rank(L_1)}=0\\
\boldsymbol\alpha_{n'+\rank(L_1)}=0
\end{gather*}
\endgroup

For the algebra $L_1\oplus L_2\oplus L_3$, the relations are
\begingroup
\allowdisplaybreaks
\begin{gather*}
\boldsymbol\alpha_{k_2}\ldots\boldsymbol\alpha_{k_1+1}\boldsymbol\alpha_{k_1}\ \ =0 \\
\boldsymbol\alpha_{k_4}\ldots\boldsymbol\alpha_{k_3+1}\boldsymbol\alpha_{k_3}\ \ =0 \\
\vdots\\
\boldsymbol\alpha_{k_{2r}}\ldots\boldsymbol\alpha_{k_{2r-1}+1}\boldsymbol\alpha_{k_{2r-1}}=0\\
\boldsymbol\alpha_n=0\\
\boldsymbol\alpha_{k'_2+\rank(L_1)}\ldots\boldsymbol\alpha_{k'_1+1+\rank(L_1)}\boldsymbol\alpha_{k'_1+\rank(L_1)}\ \ =0 \\
\boldsymbol\alpha_{k'_4+\rank(L_1)}\ldots\boldsymbol\alpha_{k'_3+1+\rank(L_1)}\boldsymbol\alpha_{k'_3+\rank(L_1)}\ \ =0 \\
\vdots \\
\boldsymbol\alpha_{k'_{2r}+\rank(L_1)}\ldots\boldsymbol\alpha_{k'_{2r-1}+1+\rank(L_1)}\boldsymbol\alpha_{k'_{2r-1}+\rank(L_1)}=0\\
\boldsymbol\alpha_{n'+\rank(L_1)}=0\\
\boldsymbol\alpha_{k''_2+\rank(L_1)+\rank(L_2)}\ldots\boldsymbol\alpha_{k''_1+1+\rank(L_1)+\rank(L_2)}\boldsymbol\alpha_{k''_1+\rank(L_1)+\rank(L_2)}\ \ =0 \\
\boldsymbol\alpha_{k''_4+\rank(L_1)+\rank(L_2)}\ldots\boldsymbol\alpha_{k''_3+1+\rank(L_1)+\rank(L_2)}\boldsymbol\alpha_{k''_3+\rank(L_1)+\rank(L_2)}\ \ =0 \\
\vdots\\
\boldsymbol\alpha_{k''_{2r}+\rank(L_1)+\rank(L_2)}\ldots\boldsymbol\alpha_{k''_{2r-1}+1+\rank(L_1)+\rank(L_2)}\boldsymbol\alpha_{k''_{2r-1}+\rank(L_1)+\rank(L_2)}=0\\
\boldsymbol\alpha_{n''+\rank(L_1)+\rank(L_2)}=0
\end{gather*}
\endgroup

The algebras $\Aa_n$ are given by one relation $\bm\alpha_n=0$. Therefore the relations of Nakayama cycle whose components are $\Aa$-type algebras are given by
\begin{align}
\bm\alpha_{R_1}=0,\bm\alpha_{R_2}=0,\ldots, \bm\alpha_{R_t}=0
\end{align}
where $R_d=\sum\limits^{d}_{i=1}\rank L_i$ for $1\leq d\leq t$ and each $L_i$ is $\Aa$-type. 
\par Because the syzygy filtered algebra construction in \ref{defep} only works for cyclic Nakayama algebras we cannot define the base set of $\Lambda$ as given in \ref{defbase}. Nevertheless we have
\begingroup
\allowdisplaybreaks
\begin{align}\label{frfr1}
\begin{split}
\#\rel\Lambda&=\text{number of projective classes}\\
&=\text{number of minimal projectives}\\
&=\text{number of projective-injectives}\\
&=\text{number of injective classes}
\end{split}
\end{align}
\endgroup
where projective and injective classes mean projective modules sharing the same socle and injective modules sharing the same top. So we can define $\cS(\Lambda)$ as the set of representatives of isomorphism classes of socles of indecomposable projective $\Lambda$-modules, i.e. $\cS(\Lambda)=\left\{S_{k_2},\ldots,S_{k_{2r}},S_n\right\}$. An injective $\Lambda$-module $I$ is not projective if and only if socle $\soc I$ is not an element of $\cS(\Lambda)$. Therefore the defect of any linear Nakayama algebra $\Lambda$ is subject to equality
\begin{align}
\rank\Lambda=\defect\Lambda+\#\rel\Lambda
\end{align}
since $\#\cS(\Lambda)=\#\rel\Lambda$. In the case of Nakayama cycles, each connected component satisfies 
\begin{align}
\rank \Lambda_i=\#\rel \Lambda_i+\defect\Lambda_i
\end{align}
for all $1\leq i\leq t$.  Because the rank, the number of relations and as a result the defect are additive we conclude 
\begin{gather}\label{defect of linear}
\begin{gathered}
\rank \Lambda=\sum_i \rank \Lambda_i\\
\defect\Lambda=\sum_i\defect\Lambda_i
\end{gathered}
\end{gather}
For semisimple algebra $\Aa^n_1$, we set its defect to $n$. Therefore the defect of algebra $\Lambda$ where $\Lambda$ has linear and semisimple components is the sum of the defects of linear parts and the number of $\Aa_1$ appearing in it.
\end{remark}
\endgroup

Let $\Lambda$ be a linear Nakayama algebra or a Nakayama cycle and $\Lambda'$ be cyclic Nakayama algebra which are higher Auslander algebra satisfing $\bm\varepsilon(\Lambda')\cong \Lambda$. Here we assume the existence of $\Lambda'$.  Keeping this information in mind, we can deduce the following properties of $\Lambda$ and $\Lambda'$.
\begin{enumerate}[label=\arabic*)]
\item By \ref{rankequaation}, cyclic Nakayama algebra $\Lambda'$ satisfies
\begin{align*}
\rank\Lambda'=\#\rel\Lambda'+\defect\Lambda'
\end{align*}
\item By \ref{defect of linear}, linear Nakayama algebra or Nakayama cycle $\Lambda$ satisfies
\begin{align*}
\rank\Lambda=\#\rel\Lambda+\defect\Lambda
\end{align*}
\item Because $\Lambda'$ is higher Auslander algebra, by Theorem \ref{thmLamdaauslanderimplies}:
\begin{align*}
\defect\Lambda'=\defect\Lambda
\end{align*}
\item Because of the $\bm\varepsilon$-construction i.e. $\bm\varepsilon(\Lambda')\cong\Lambda$:
\begin{align*}
\#\rel\Lambda'=\#\cB(\Lambda')=\#\cS'(\Lambda')=\#\cS(\Lambda')=\rank\bm\varepsilon(\Lambda')=\rank\Lambda.
\end{align*}
\item Hence we obtain the rank of $\Lambda'$:
\begin{align*}
\rank\Lambda' & =\#\rel\Lambda'+\defect\Lambda'\\
&=\rank\Lambda+\defect\Lambda.
\end{align*}
\end{enumerate}

Definitions \ref{definheritance}, \ref{def defect of projective modules}, \ref{defconsecutive} are still valid for cyclic Nakayama algebra $\Lambda'$. We use the same notation \ref{defcyclic oredring} for simple modules. Statements \ref{lemma 3.3}, \ref{prop 3.3}, \ref{prop 3.5}, \ref{cor 3.3}, \ref{prop defect},\ref{newmodules are injective}, \ref{delta is simple iff},\ref{claim 1},\ref{claim 2},\ref{Structure of cyclic ordering of simple modules} are still valid for the $\Lambda$-modules $S_i$, $S_{i+1}$ when $\begin{vmatrix}
S_i\\S_{i+1}
\end{vmatrix}$ is indecomposable. We give statements and proofs for the cases not covered since there is no simple and projective $\Lambda$-module in the previous section \ref{section cyclic}.

\begin{lemma}\label{lemma simple proj linear } Let $\Lambda$ be a connected linear Nakayama algebra where $S_n$ is the simple projective $\Lambda$-module which is inherited to $\Lambda'$ as $\topp\Delta(S_n)$. $\Lambda'$-module $\Delta(S_n)$ is not simple. Indeed, the length of $\Delta(S_n)$ is $\defect(P_1)+1$.
\end{lemma}
\begin{proof}
If $\topp\Delta(S_n)\cong\Delta(S_n)$ then $P(\topp\Delta(S_n))$ is simple projective module since $\Hom_{\Lambda'}(\cP,\Delta(S_1))$ cannot be a submodule of $S_n$ by the linearity of $\Lambda$. Claims follows since $\Lambda'$ is cyclic Nakayama algebra, there is no simple projective module.
\par We want to show that $\ell\Delta(S_n)=1+\defect(P_1)$. $\Lambda'$ is cyclic so we consider simple $\Lambda'$-modules
\begin{gather}
\topp P(\Delta(S_n))>X_1>X_2>\cdots>X_{t}>\topp P(\Delta(S_{1})).
\end{gather}

First we assume that $t>\defect(P(\Delta(S_1)))$.
If $c_{1}$ is the length of $P(S_{1})$ in $\Lambda$, then  $c_1=\defect(P(S_1))+1$ which follows from every quotient of $P(S_1)$ is injective by linearity of $\Lambda$. Notice that $\soc\Delta(S_n)\cong X_t$ and $P(\Delta(S_n))\cong \Delta(S_n)$ because $\cB(\Lambda')$-length of $P(\Delta(S_n))$ is one. By Proposition \ref{newmodules are injective} every $P(X_i)$ is projective-injective and they have nonisomorphic socles. All $\cB(\Lambda')$-length one submodules of $P(X_i)$'s appears as $\cB(\Lambda')$-composition factors of $P(\Delta(S_1))$. This implies $\cB(\Lambda')$-length of $P(\Delta(S_n))=t+1$. We point out that summand $1$ comes from the inherited socle of $P(\Delta(S_1))$ . We get
\begin{align}\label{equ 4}
\begin{gathered}
t>\defect(P(\Delta(S_1)))\\
t+1>\defect(P(S_1))+1=c_1
\end{gathered}
\end{align}
which is contradiction. Therefore $t\leq \defect(P(\Delta(S_1)))$.
\par We can use the same proof of Proposition \ref{Structure of cyclic ordering of simple modules} to show that $t=\defect(P(S_1))$. Otherwise the inequalities \ref{some eq in } would imply the defects of $\Lambda'$ and $\Lambda$ are different.
\end{proof}

\begin{lemma}\label{lemma delta length nakayama cycle} Let $\Lambda$ be a Nakayama cycle and $S_{n_1},\ldots,S_{n_t}$ be simple projective modules. If $\Delta(S_{n_1}),\ldots,\Delta(S_{n_t})$ are the corresponding $\Lambda'$-modules then we have
\begin{align}
\ell\Delta(S_{n_i})=\defect(P(S_{n_{i+1}}))+1
\end{align}
for all $1\leq i\leq n_t-1$ and
\begin{align}
\ell\Delta(S_{n_t})=\defect(P(S_1))+1.
\end{align}
\end{lemma}
\begin{proof}
We can use the proof of the Proposition \ref{Structure of cyclic ordering of simple modules} (see corollary \ref{length of deltas}). The change of the indices of $S_n,S_1$ to $S_{n_i},S_{n_{i+1}}$ in the proof still works. For the second part, change of $n$ to $n_t$ is still valid.
\end{proof}

\begin{proposition}\label{lengths of delta proj linears} The lengths of inherited projective $\Lambda'$-modules $P(\topp\Delta(S_i))$ are given by
\begin{align}
\ell P\topp(\Delta(S_i))=\sum\limits^{\ell P(S_i)+i}_{j=i} \ell\Delta(S_j)=\sum\limits^{\ell P(S_i)+i}_{j=i} \defect(P(S_{j+1}))+1
\end{align}
\end{proposition}
\begin{proof}
The length of $\Delta(S_{n_j})$ where $S_{n_j}$ is simple projective $\Lambda$-module is given by lemma \ref{lemma simple proj linear }. For the inherited projective modules we can use Proposition \ref{Structure of cyclic ordering of simple modules} and corollary \ref{length of deltas} as we did for corollary \ref{lengths of delta proj}.
\end{proof}

The ordering in the Nakayama cycle is important because it affects the length of $P(\Delta(S_{n_i}))$ in mod-$\Lambda'$ because the defects of $P(S_{n_i+1})$ might be different and change the Kupisch series of $\Lambda'$. We give an example for this fact at \ref{example nak cycle}.

\begin{proposition}\label{defect prop 5} $\Lambda'$ is well-defined, its syzygy filtered algebra is $\Lambda$ and it is a higher Auslander algebra.
\end{proposition}
\begin{proof}
It is enough to find the lengths of indecomposable projective $\Lambda'$-modules because $\Lambda'$ is Nakayama algebra. By Proposition \ref{lengths of delta proj linears} we get the lengths of $\cB(\Lambda')$-filtered projective modules which are inherited projectives. We need to detect the lengths of  projective modules which are not inherited.
Consider simple $\Lambda'$-modules
\begin{align}\label{linear order 5}
\topp\Delta(S_i)>X_1>\cdots>X_t>\topp\Delta(S_{i+1})
\end{align}
where $S_i$ is not simple projective $\Lambda$-module.
If we denote the simple $\Lambda'$-module $\topp\Delta(S_i)$ by $\tilde{S_{i}}$, then by proposition \ref{Structure of cyclic ordering of simple modules}, we obtain the structure of $\Delta(S_i)$
\begin{align}\label{friday 3}
 \Delta(S_i)\cong\begin{vmatrix}
\tilde{S_i}\\X_1\\\vdots\\X_t
\end{vmatrix}
\end{align}
 and $\defect(P(S_{i+1}))=t$.
All $P(X_j)$ contains the module $\begin{vmatrix}
\Delta(S_{i+1})\\\vdots\\\Delta(T_i)
\end{vmatrix}$ as a subquotient where $\Delta(T_i)$ is the submodule of $P(\Delta(S_i))$. Moreover each $P(X_j)$ has a submodule $\Delta(T_{i+j})$ which corresponds to socle of injective module $I(T_{i+j})$ in $\Lambda$. Therefore the lengths of not inherited projective modules $P(X_j)$ $1\leq j\leq t$ satisfying \ref{linear order 5} are given by
\begin{align}
t-j+1+\ell\left(\begin{vmatrix}
\Delta(S_{i+1})\\\vdots\\\Delta(T_i)
\end{vmatrix}\right)+\ell\left(\begin{vmatrix}
\Delta(T_{i+1})\\\vdots\\\Delta(T_{i+j})
\end{vmatrix}\right).
\end{align}
By Proposition \ref{lengths of delta proj linears}, it can be expresses in terms of the defects of $\Lambda$-modules.
Now we want to describe projective covers of $P(X_j)$ satisfying
\begin{align}\label{linear order 6}
\topp\Delta(S_{n_i})>X_1>\cdots>X_{t}>\topp\Delta(S_{n_i+1})
\end{align}
where $S_{n_i}$ is simple projective $\Lambda$-module. By lemma \ref{lemma delta length nakayama cycle}, $t=\defect(P(S_{n_i+1}))$. By Proposition \ref{newmodules are injective}, each $P(X_j)$ are projective-injective and have nonisomorphic socles. Moreover, socles correspond to socles of injective $\Lambda$-modules via \ref{categoricalequivalnce}. $\Delta(S_{n_i})$ and $P(X_1)$ have the structures
\begin{align}\label{friday 1}
\Delta(S_{n_i})\cong \begin{vmatrix}
\tilde{S_{n_i}}\\X_1\\\vdots\\X_t
\end{vmatrix},&\quad\quad
P(X_1)\cong\begin{vmatrix}
X_1\\\vdots\\X_t\\\Delta(S_{n_i+1})
\end{vmatrix}
\end{align} 
because $S_{n_i+1}$ is simple injective $\Lambda$-module. Therefore the length of $P(X_1)$ is $t+\ell\Delta(S_{n_i+1})$. For the rest, each $P(X_j)$ contains  $\Delta(S_{n_i+1})$ as a subquotient. Therefore the lengths are
\begin{align}\label{friday 2}
t-j+1+\ell\left(\Delta(S_{n_i+1})\right)+\ell\left(\begin{vmatrix}
\Delta(T_1)\\\vdots\\\Delta(T_{j-1})
\end{vmatrix}\right)
\end{align}
for $2\leq j\leq t$ where $P(\Delta(S_{n_i+1}))$ is $\begin{vmatrix}
\Delta(S_{n_i+1})\\\Delta(T_1)\\\vdots\\\Delta(T_t)
\end{vmatrix}$.
By lemma \ref{lemma delta length nakayama cycle}, we can express the lengths in terms of the defects of projective $\Lambda$-modules precisely. By similar arguments, we can compute the lengths of $P(X_j)$'s where $\topp\Delta(S_{\rank\Lambda})>X_1>\ldots>X_t>\topp\Delta(S_1)$. This describes the algebra $\Lambda'$, because we know the length of each projective $\Lambda'$-module. However we  do not want to explicitly write its Kupisch series to avoid notational complications. \\

The second part follows from the definition \ref{defep} as we did in the proof of Proposition \ref{def prop4}. By the statements \ref{Structure of cyclic ordering of simple modules}, \ref{length of deltas}, \ref{lemma simple proj linear } and \ref{lemma delta length nakayama cycle} we know the structure of the elements of the base set $\cB(\Lambda')$
\begin{gather*}
\Delta(S_i)\, \text{is simple} \iff \defect(P(S_{i+1}))=0\\ \ell(\Delta(S_i))=t+1 \iff \defect(P(S_{i+1}))=t\\
\ell(\Delta(S_{n_i}))=t'+1 \iff \defect(P(S_{n_i+1}))=t'.
\end{gather*}
If $\Delta(S_j)$ is simple $\Lambda'$-module, let's denote it by $\tilde{S_j}=\Delta(S_j)$. If $\Delta(S_i)$ is not simple, it has the structure given in either \ref{friday 3} or \ref{friday 1}, so in either case we use the notation $\tilde{S_i}=\topp\Delta(S_i)$.
Therefore the sets given in definition \ref{defbase} are of the form
\begin{align*}
\cB(\Lambda')=&\left\{\Delta(S_1),\Delta(S_2),\ldots,\Delta(S_n)\right\}\\
\cS'(\Lambda')=&\left\{ \topp\Delta(S_1),\topp\Delta(S_2),\ldots,\topp\Delta(S_n)\right\}=\left\{\tilde{S_1},\tilde{S_2}\ldots,\tilde{S_n}\right\}\\
\cS(\Lambda')=&\left\{\soc\Delta(S_1),\soc\Delta(S_2),\ldots,\soc\Delta(S_n)\right\}=\left\{\tau^{-1}\tilde{S_1},\tau^{-1}\tilde{S_2}\ldots,\tau^{-1}\tilde{S_n}\right\}.
\end{align*}

$\cB(\Lambda')$-filtered projective modules are exactly the inherited projective modules, therefore by the definition of syzygy filtered algebra \ref{defep} we get
\begin{align*}
\bm{\varepsilon}(\Lambda'):=\End_{\Lambda'}\left(\bigoplus\limits_{S\in \cS'(\Lambda')}P(S)\right)=\End_{\Lambda'}\left(P(\tilde{S_1})\oplus\ldots\oplus P(\tilde{S_n})\right)
\cong \Lambda.
\end{align*}

Simple $\Lambda$-modules are $S_1,S_2,\ldots,S_n$ which follows from 
\begin{align*}
\Hom_{\Lambda'}\left(\bigoplus\limits^n_{i=1} P(\tilde{S_i}),\tilde{S_i}\right)=\Hom_{\Lambda'}\left(\bigoplus\limits^n_{i=1} P(\Delta(S_i)),\Delta(S_i)\right)=S_i.
\end{align*} Similarly, projective $\Lambda$-modules are 
\begin{align*}
\Hom_{\Lambda'}\left(\bigoplus\limits^n_{i=1} P(\tilde{S_i}),P(\tilde{S_i})\right)=\Hom_{\Lambda'}\left(\bigoplus\limits^n_{i=1} P(\Delta(S_i)),P(\Delta(S_i))\right)=P(S_i).
\end{align*}\\

The arguments we use for the last part are very similar to the proof of Proposition \ref{def prop4}. The existence of simple projective $\Lambda'$-modules requires attention. By lemma \ref{every second syzy of injective}, the first syzygy of any injective $\Lambda'$-module is a proper submodule of an element of the base set. By the Proposition \ref{Structure of cyclic ordering of simple modules} and corollary \ref{length of deltas} $\Delta(S_i)$ is not simple if and only if the defect of $P(S_{i+1})$ is nonzero provided that $\begin{vmatrix}
S_i\\S_{i+1}
\end{vmatrix}$ is an indecomposable $\Lambda$-module. If $S_{n_i}$ is simple projective $\Lambda$-module then by lemmas \ref{lemma simple proj linear } and \ref{lemma delta length nakayama cycle} $\Delta(S_{n_i})$ is not simple. Moreover, the projective covers of proper submodules of $\Delta(S_i)$ and $\Delta(S_{n_i})$ are projective-injective modules which are unique in their class by Propositions \ref{newmodules are injective} and the construction \ref{friday 1}, \ref{friday 2} respectively.  On the other hand, there exist projective $\Lambda'$-modules having $\Delta(S_i)$ and $\Delta(S_{n_i})$ as a submodule by Remark \ref{submodule of projectives}. Now we can construct the projective resolution of any injective $\Lambda'$-module. By using the exact sequence
\begin{align}
0\rightarrow \Omega^1(I)\rightarrow P(I)\rightarrow I\rightarrow 0
\end{align}
together with lemma \ref{every second syzy of injective} either we get $\Omega^1(I_{\Lambda})\cong\begin{vmatrix}
X_j\\\vdots\\ X_t
\end{vmatrix}$ for some $j$, $\topp\Delta(S_i),\topp\Delta(S_{i+1})$ satisfying \ref{exc seqq} or $\Omega^1(I_{\Lambda})\cong\begin{vmatrix}
X_{j'}\\\vdots\\ X_{t'}
\end{vmatrix}$ for some $j'$, $\topp\Delta(S_{n_i}),\topp\Delta(S_{n_i+1})$ satisfying \ref{linear order 6}. For the former,   $P(X_j)$ and $P(\Delta(S_{i+1}))$ are projective covers of $\Omega^1(I_{\Lambda})$ and $\tau X_t\cong\topp\Delta(S_{i+1})$ respectively, hence we get the first terms of the projective resolution of $I$.
\begin{align}
\cdots\rightarrow P(\Delta(S_{i+1}))\rightarrow P(X_j)\rightarrow P(I)\rightarrow I\rightarrow 0.
\end{align}
$\Omega^2(I)$ is proper quotient of $P(\Delta(S_{i+1}))$, because $P(\Delta(S_{i+1}))$ is not a submodule of $P(X_j)$ by \ref{sth ccc}. By Remark \ref{secondsyzygy}, $\Omega^2(I)$ has $\cB(\Lambda')$-filtration and $P(\Delta(S_{i+1}))$ is inherited projective module, therefore $\Hom_{\Lambda'}(\cP,\Omega^2(I))$ is injective $\Lambda$-module which is the quotient of $P(S_{i+1})$. 
\par For the latter, $P(X_{j'})$ and $P(\Delta(S_{n_i+1}))$ are projective covers of $\Omega^1(I_{\Lambda})$ and $\tau X_{t'}\cong\topp\Delta(S_{n_i+1})$ respectively, hence we get the first terms of the projective resolution of $I$.
\begin{align}
\cdots\rightarrow P(\Delta(S_{n_i+1}))\rightarrow P(X_{j'})\rightarrow P(I)\rightarrow I\rightarrow 0.
\end{align}
$\Omega^2(I)$ is proper quotient of $P(\Delta(S_{n_i+1}))$, because $P(\Delta(S_{n_i+1}))$ is not a submodule of $P(X_{j'})$ by \ref{friday 1},\ref{friday 2}. By Remark \ref{secondsyzygy}, $\Omega^2(I)$ has $\cB(\Lambda')$-filtration and $P(\Delta(S_{n_i+1}))$ is inherited projective module, therefore $\Hom_{\Lambda'}(\cP,\Omega^2(I))$ is injective $\Lambda$-module which is the quotient of $P(S_{n_i+1})$. This shows that for any injective non-projective $\Lambda'$-module, there exist injective non-projective $\Lambda$-module via the the functor $\Hom_{\Lambda'}(\cP,\Omega^2(-))$.\\

 We need to show that each injective quotient of $P(S_{i+1})$ and $P(S_{n_i+1})$ corresponds to unique injective $\Lambda'$-module via the functor $\Hom_{\Lambda'}(\cP,\Omega^2(-))$. Since we discussed the case for $P(S_{i+1})$ in the proof of Proposition  \ref{def prop4}, we do not want to repeat it here. We focus on $P(S_{n_i+1})$.\\

Let $\defect(P(S_{n_i+1}))=t$ and $I_1,\ldots,I_t$ be injective quotients of $P(S_{n_i+1})$ where the indices show just enumeration not  the injective envelope of simple module having that index. Let $\Delta I_1,\ldots,\Delta I_t$ be the corresponding $\Lambda'$-modules i.e. $\Hom_{\Lambda'}(\cP,\Delta I_i)\cong I_i$. By the construction we used in the proof of the Proposition \ref{Structure of cyclic ordering of simple modules} and \ref{friday 1},\ref{friday 2} each $\Delta I_j$ is submodule of not inherited projective module $P(X_j)$ which is projective-injective and unique in its class \ref{projectiveclasses}. Let $M_j$ denote the quotient $\faktor{P(X_j)}{\Delta I_j}$. Notice that $M_j\ncong M_{j'}$ when $j\neq j'$. Each $M_j$ is submodule of $\Delta(S_{n_i})$, and by the lemma \ref{injectivity}, submodule of the unique projective-injective module $P$, so
\begin{align*}
M_j\varlonghookrightarrow \Delta(S_{n_i})\varlonghookrightarrow P 
\end{align*}
Each of the quotient $\faktor{P}{M_j}$ is injective $\Lambda'$-module, moreover $\faktor{P}{M_j}\ncong\faktor{P}{M_{j'}}$ because socles are different i.e. $\soc\faktor{P}{M_j}\cong \tau^{-1}X_j\ncong \tau^{-1}X_{j'}\cong\soc\faktor{P}{M_{j'}}$. therefore $\Hom_{\Lambda'}(\cP,\faktor{P}{M_j})\cong I_j$. This shows for each injective $\Lambda$-module which is the quotient of $P(S_{n_i+1})$, there exists one injective $\Lambda'$-module. When we vary $n_i$ over the simple projective modules together with the proof of Proposition \ref{def prop4} for any injective $\Lambda$-module, there exists one injective $\Lambda'$-module.\\
\par By Proposition \ref{different injectives}, if there are more than one injective $\Lambda'$-module which maps to the same injective $\Lambda$-module via $\Hom_{\Lambda'}(\cP,\Omega^2(-))$, then the socle of the projective covers of the first syzygies are same, however the tops are not isomorphic. Moreover the tops  are non-inherited simple modules by the structure of non-simple elements of $\cB(\Lambda')$ (see \ref{friday 3},\ref{friday 1}). By using the Proposition \ref{newmodules are injective}, any projective cover of the first syzygy is unique not inherited projective module in its class, therefore this is not possible.

There is another case that we need to analyze, the first syzygy of an injective $\Lambda'$-module cannot be projective. If $\Omega^1(I)$ is projective, by lemma \ref{every second syzy of injective}, it is a proper submodule of either $\Delta(S_i)$ or $\Delta(S_{n_i})$ which makes $\Delta(S_i)$ a projective module. This cannot happen If $\begin{vmatrix}
S_i\\S_{i+1}
\end{vmatrix} $ is indecomposable. Because $\Omega^1(I)$ would be the subquotient of the module  $\begin{vmatrix}
\Delta(S_i)\\\Delta(S_{i+1})
\end{vmatrix}$  where  $\Hom_{\Lambda'}(\cP,\begin{vmatrix}
\Delta(S_i)\\\Delta(S_{i+1})
\end{vmatrix})\cong\begin{vmatrix}
S_i\\S_{i+1}
\end{vmatrix}$ which is not possible by the projectivity of $\Omega^1(I)$. If $\Omega^1(I)\subset \Delta(S_{n_i})$ and projective, then $\Omega^1(I)$ would be subquotient of the projective module of the form $P(X_j)$ where $X_j$ satisfies either \ref{friday 1} or \ref{friday 2}. Clearly, this cannot happen. Now we can adjust Proposition \ref{GorensteinReduction} accordingly. For any injective non-projective $\Lambda$-module $I$, there is a unique injective non-projective $\Lambda$-module $I_{\Lambda'}$ such that $\Hom_{\Lambda'}(\cP,\Omega^2(I_{\Lambda'}))\cong I$. Therefore the defects of $\Lambda$ and $\Lambda'$ are equal. Consider the projective resolution 
\begin{align}\label{friday 10}
\cdots\rightarrow P(\Omega^1(I_{\Lambda'}))\rightarrow P(I_{\Lambda'})\rightarrow I_{\Lambda'}\rightarrow 0.
\end{align}
$P(I_{\Lambda'})$ has to be projective-injective module since it is a projective cover of an injective. By combining the lemma \ref{every second syzy of injective} with \ref{friday 3}, \ref{friday 1} the projective module $P(\Omega^1(I))$ is not inherited and by Proposition \ref{newmodules are injective} any non-inherited projective $\Lambda'$-module is projective-injective. Therefore we get
\begin{align*}
\domdim I_{\Lambda'}=2+\domdim\Omega^2(I_{\Lambda'}).
\end{align*}
By Remark \ref{secondsyzygy}, any syzygy $\Omega^i(I)$ and projective module $P(\Omega^i(I))$ have $\cB(\Lambda')$-filtration for $i\geq 2$. We can apply the functor $\Hom_{\Lambda'}(\cP,-)$ to \ref{friday 10} in order  to get 
\begin{align*}
\domdim_{\Lambda'}\Omega^2(I_{\Lambda'})=\domdim_{\Lambda}\Hom_{\Lambda'}(\cP,\Omega^2(I_{\Lambda'})).
\end{align*}
$\Hom_{\Lambda'}(\cP,\Omega^2(I_{\Lambda'}))$ is injective $\Lambda$-module which we denote it by $I$. In summary, we get
\begingroup
\allowdisplaybreaks
\begin{align*}
\domdim I_{\Lambda'}&=2+\domdim\Omega^2(I_{\Lambda'})\\
&=2+\domdim_{\Lambda}\Hom_{\Lambda'}(\cP,\Omega^2(I_{\Lambda'}))\\
&=2+\domdim_{\Lambda} I
\end{align*}
\endgroup
for any injective $\Lambda'$-module. By the characterization \ref{gercek tanim} of the dominant dimension  we conclude that 
\begingroup
\allowdisplaybreaks
\begin{align}
\begin{split}
\domdim\Lambda'&=\sup\left\{\domdim I_{\Lambda'}\,\vert\, I_{\Lambda'} \text{ is injective non-projective } \Lambda'\, \text{module}\right\}\\
&=\sup\left\{2+\domdim I\,\vert\, I \text{ is injective non-projective } \Lambda\, \text{module}\right\}\\
&=2+\sup\left\{\domdim I\,\vert\, I \text{ is injective non-projective } \Lambda\, \text{module}\right\}\\
&=2+\domdim\Lambda.
\end{split}
\end{align}
\endgroup
Because $\gldim\Lambda$ is finite and $\Lambda\cong\bm\varepsilon(\Lambda')$, we obtain that 
$\gldim\Lambda'=2+\gldim\Lambda$. $\Lambda$ is a higher Auslander algebra so $\gldim\Lambda=\domdim\Lambda$. Therefore $\Lambda'$ is also a higher Auslander algebra because 
\begingroup
\allowdisplaybreaks
\begin{align}
\domdim\Lambda'=2+\domdim\Lambda=2+\gldim\Lambda=\gldim\Lambda'.
\end{align}
\endgroup \end{proof}

\begin{proposition}\label{defect prop 6} $\Lambda'$ is unique.
\end{proposition}
\begin{proof} 
We can use the proof of Proposition \ref{prop uniqu} with some modifications. Let $\Lambda'$ and $\Lambda''$ be two Nakayama algebras satisfying $\bm\varepsilon(\Lambda')\cong\bm\varepsilon(\Lambda'')\cong \Lambda$ where $\Lambda$ is either linear or Nakayama cycle. We have the following key observations:
\begin{itemize}
\item By \ref{categoricalequivalnce} the category of $\cB(\Lambda')$-filtered $\Lambda'$-modules is equivalent to the category of $\cB(\Lambda'')$-filtered $\Lambda''$-modules.
\item If $\Lambda'$ and $\Lambda''$ are higher Auslander algebras then their rank is same which is $\rank\Lambda'=\rank\Lambda''=\rank\Lambda+\defect\Lambda$ (see \ref{defect of linear}). Because $\Lambda$ is fixed either as a linear algebra or Nakayama cycle with fixed $\tau$, upto permutation the base sets $\cB(\Lambda')$ and $\cB(\Lambda)$ are same because of the statements \ref{Structure of cyclic ordering of simple modules}, \ref{lemma delta length nakayama cycle}, \ref{lemma simple proj linear }. This is much stronger than the categorical equivalence.
\item More importantly, the number of projective $\Lambda'$ and $\Lambda''$-modules sharing the same socle is same. This is true for simple projective modules also. If $P_i\supset P_{i+1}\supset\cdots\supset P_j$ have the socle $T_i$ in $\Lambda$, then the inherited projectives have the same socle $\soc\Delta'(T_i)$ in $\Lambda'$ and $\soc\Delta''(T_i)$ in $\Lambda''$. And by the Proposition \ref{newmodules are injective}, the remaining projective modules (non-inherited) are projective-injective i.e. unique projective modules having submodules $\Delta'(T_j)$'s in $\Lambda'$ and $\Delta''(T_j)$'s in $\Lambda''$ which correspond to socles of injective but non-projective $\Lambda$-modules $T_j$'s. Therefore we know distribution of socles to projective modules in both $\Lambda'$ and $\Lambda''$. In other words, number of classes of projective modules are equal (see \ref{frfr1}) and the number of projective modules in a class is same for both $\Lambda'$ and $\Lambda''$.
\end{itemize}

There are infinitely many algebras satisfying the items in the above list and they are parametrized by the lengths of projective modules. In details, If $(c_1,\ldots,c_n)$ is Kupisch series of an algebra $\Lambda'$ then each of the $(c_1+j,\ldots,c_n+j)$ for $j\neq 0$ can be $\Lambda''$. The increment or decrement $j$ has to be same for all $i$, otherwise it would change either the structure of $\cB(\Lambda')$ or the distribution of projective modules into projective classes and the number of projectives in the projective class. However, even for the same increment or decrement $j\neq 0$, $\bm\varepsilon(\Lambda'')\ncong\bm\varepsilon(\Lambda')$. Since $\Lambda$ is fixed, the length of inherited projective modules can take only one value which comes from the unique structure of the base set $\cB(\Lambda')$ via the Proposition \ref{lengths of delta proj linears}. This finishes the proof.
\end{proof}

\begin{remark} If we relax the conditions in the proof above, we loose the uniqueness.
\begin{itemize}
\item Two algebras can produce the same syzygy filtered algebra, but the lengths of modules in their base sets cannot be same. For instance the base set of the algebras $(n,n-1,\ldots,3,2,2)$ is
\begin{align*}
\cB(\Lambda)=\left\{\Delta_1\cong S_1,\Delta_2\cong\begin{vmatrix}
S_2\\\vdots\\S_n
\end{vmatrix}\right\}
\end{align*}
where each of them corresponds to different partition when $n$ varies. 
\item Projective covers of not inherited simple modules has to be projective-injective. If we drop the second and third conditions  we can get $\bm\varepsilon(\Lambda')\cong\Lambda$ but $\Lambda'$ is not higher Auslander algebra anymore. Here is an example. Let $(2,2,2,1)$ be Kupisch series of linear Nakayama algebra $\Lambda$. $(2,2,3,2,2)$ is higher Auslander algebra, $(3,2,2,3,2,3)$ and $(2,3,4,3,2,2)$ are not. This happens because there is a change in the numbers of projective modules in projective classes.
\item Third item is not enough by itself. For example if $\Lambda\cong\Aa_2$, simple modules of $\Lambda'$ are $1>2>x>1$. So $(3,2,2)$, $(3,2,3)$ are possible. But the latter changes the defect of $P_1$, and therefore the filtration. Moreover it does not finite global dimension
\item If $n\geq 2$ all algebras $(n+1,n,n)$ have the same base set. However only one of them is syzygy equivalent to $\Aa_2$. The length of $\Lambda$-modules is decisive also.
\end{itemize}
\end{remark}

\begin{corollary}\label{necklace} Let $\Lambda_1,\ldots,\Lambda_t$ be linear Nakayama algebras which are higher Auslander algebras satisying: $\gldim\Lambda_i=\domdim\Lambda_i=d$. For each Nakayama cycle $(\Lambda,\tau)$ where $\Lambda=\Lambda_1\times\cdots\times\Lambda_t$, by the Theorem \ref{thmreverseepsilon2}, there exists  unique cyclic Nakayama algebra $\Lambda'$.
The number of possible permutations $\tau$ is given by necklace numbers 
\begin{center}
$N_t(n)=\cfrac{1}{n}\sum_{p\vert n}\phi(p)t^{\frac{n}{p}}$
\end{center}
where $n=\sum^t_{i=1} im_i$ ,$m_i$ is the multiplicity of $\Lambda_i$ in $\Lambda$ and $\phi$ is Euler's totient function.  The number of cyclic Nakayama algebras of global dimension $d+2$ which are $\bm\varepsilon$-equivalent to $\Lambda$ is the number of possible permutations $\tau$ if we ignore the ordering in the Nakayama cycle. 
\end{corollary}

\subsection{Examples for Theorems \ref{thmreverseepsilon},\ref{thmreverseepsilon2} }
By the Theorems \ref{thmreverseepsilon},\ref{thmreverseepsilon2} we can generate all the chains:
\begin{align}
\left(\rank\Lambda,k\right)\longrightarrow \left(\rank\Lambda+\defect\Lambda,k+2\right)
\end{align}
where $k=\gldim\Lambda=\domdim\Lambda$. By the Theorem \ref{thmdoubling}, we obtain the other direction
\begin{align}
\left(\rank\Lambda,k\right)\longrightarrow \left(2\rank\Lambda,k\right)\longrightarrow\left(3\rank\Lambda,k\right).
\end{align}

We give some examples to show how we apply main theorems.

\begin{example}\label{example1}
Consider linear Nakayama algebra with the Kupisch series $(2,3,2,2,1)$. We apply the Theorem \ref{thmreverseepsilon2} to
\begin{align}
\left\{\begin{vmatrix}
1\\2
\end{vmatrix},\begin{vmatrix}
2\\3\\4
\end{vmatrix},\begin{vmatrix}
3\\4
\end{vmatrix},\begin{vmatrix}
4\\5
\end{vmatrix},\begin{vmatrix}
5
\end{vmatrix}\right\}.
\end{align}
Injective non-projective modules are $I_1$, $I_3$, their tops are $1,2$. Hence the ordering of simple $\Lambda'$-modules is 
\begin{gather*}
1>x_1>2>3>4>5>x_2>1
\end{gather*}
Since the defect of $P_2$ and $P_1$ are one, the base set is
\begin{gather*}
\cB(\Lambda')=\left\{\begin{vmatrix}
1\\
x_1
\end{vmatrix},\vert 2\vert,\vert 3\vert,\vert 4\vert,\begin{vmatrix}
5\\
x_2
\end{vmatrix}\right\}
\end{gather*}
This forces that projective module with top $x_1$ ends at $3$, and projective module with top $x_2$ ends at $x_1$. Projectives of $\Lambda'$ are:
\begin{align}
\left\{\begin{vmatrix}
1\\x_1\\2
\end{vmatrix},\begin{vmatrix}
x_1\\2\\3
\end{vmatrix},
\begin{vmatrix}
2\\3\\4
\end{vmatrix},\begin{vmatrix}
3\\4
\end{vmatrix},\begin{vmatrix}
4\\5\\x_2
\end{vmatrix},\begin{vmatrix}
5\\x_2
\end{vmatrix},\begin{vmatrix}
x_2\\1\\x_1
\end{vmatrix}\right\}
\end{align}
One can verify that ${\bm\varepsilon}(\Lambda')\cong L$.\\

 We apply construction in the proof of \ref{thmreverseepsilon} to the previous result. Let $\Lambda$ be given by $(3,3,3,3,2,3,2)$. Injective but non-projective modules are $I_2$ and $I_7$. New simple modules satisfy $5>x_1>6$ and $7>x_2>1$. We get:

\begin{align}\label{exampleextension}
\left\{\begin{vmatrix}
1\\2\\3
\end{vmatrix},\begin{vmatrix}
2\\3\\4
\end{vmatrix},
\begin{vmatrix}
3\\4\\5\\x_1
\end{vmatrix},\begin{vmatrix}
4\\5\\x_1\\6
\end{vmatrix},\begin{vmatrix}
5\\x_1\\6
\end{vmatrix},\begin{vmatrix}
x_1\\6\\7\\x_2
\end{vmatrix},\begin{vmatrix}
6\\7\\x_2\\1
\end{vmatrix},\begin{vmatrix}
7\\x_2\\1
\end{vmatrix},\begin{vmatrix}
x_2\\1\\2
\end{vmatrix}
\right\}
\end{align}
If we continue the process, algebras we get are:
\begin{align*}
(2,3,2,2,1)\mapsto (3,3,3,3,2,3,2)\mapsto (3,3,4,4,3,4,4,3,3)\mapsto (4,4,4,5,5,5,4,4,4,4,3)\mapsto\cdots
\end{align*}

\end{example}

\begin{example} We construct cyclic Nakayama algebra which is syzygy equivalent to $(221)\oplus (23221)$. Notice that there is only one Nakayama cycle in this case.
Projective modules are :
\begin{gather}
\left\{\begin{vmatrix}
1\\2
\end{vmatrix},\begin{vmatrix}
2\\3
\end{vmatrix},\begin{vmatrix}
3
\end{vmatrix},\begin{vmatrix}
4\\5
\end{vmatrix},\begin{vmatrix}
5\\6\\7
\end{vmatrix},\begin{vmatrix}
6\\7
\end{vmatrix},\begin{vmatrix}
7\\8
\end{vmatrix},\begin{vmatrix}
8
\end{vmatrix}\right\}
\end{gather}
Injective but non-projective modules are $I_1,I_4,I_6$. Their tops are $1,4,5$ therefore $\Delta(8),\Delta(3), \Delta(4)$ are not simple modules.
\begin{gather}
\left\{\begin{vmatrix}
1\\2
\end{vmatrix},\begin{vmatrix}
2\\3\\x
\end{vmatrix},\begin{vmatrix}
3\\x
\end{vmatrix},
\begin{vmatrix}
P_x
\end{vmatrix},
\begin{vmatrix}
4\\y\\5
\end{vmatrix},
\begin{vmatrix}P_y\end{vmatrix},\begin{vmatrix}
5\\6\\7
\end{vmatrix},\begin{vmatrix}
6\\7
\end{vmatrix},\begin{vmatrix}
7\\8\\z
\end{vmatrix},\begin{vmatrix}
8\\z
\end{vmatrix},\begin{vmatrix}P_z\end{vmatrix}\right\}
\end{gather}
We detect the new projective modules. Because $\Delta(4)$ is not simple, we get $P_x=\begin{vmatrix}
x\\4\\y
\end{vmatrix}$. $\Delta(6)$ and $\Delta(1)$ are simple, so $P_y=\begin{vmatrix}
y\\5\\6
\end{vmatrix}$ and $P_z=\begin{vmatrix}
z\\1
\end{vmatrix} $.  Therefore the algebra is given by the Kupisch series $(3,2,3,3,3,3,2,3,2,2,2)$. Injective but non-projective $\Lambda'$-modules are $\begin{vmatrix}
2\\3
\end{vmatrix},\begin{vmatrix}
x\\4
\end{vmatrix},\begin{vmatrix}
7\\8
\end{vmatrix}$. As we did in example \ref{example1}, here is the chain upto the $\bm\varepsilon$-equivalences:
\begingroup
\allowdisplaybreaks
 \begin{center}
 $\xymatrixcolsep{5pt}\xymatrixrowsep{4pt}
\xymatrix{ (2,2,1)\oplus (2,3,2,2,1)\ar[d]
\\
(3,2,3,3,3,3,2,3,2,2,2)\ar[d]\\
(4,4,3,3,3,3,4,4,3,3,3,2,3,3)\ar[d]\\
(4,4,3,4,4,4,4,4,4,3,4,4,3,4,4,4,4)\ar[d]\\
(5,5,4,4,4,4,4,5,5,5,4,5,5,5,4,4,4,4,4,5)\ar[d]\\
\cdots}$
 \end{center}
\endgroup
Now, we give a chain of algebras which are ${\bm\varepsilon}$-equivalent to direct sum of $\Aa$ type quivers, therefore global dimensions are odd.
 \begin{align}\label{exampe a2+a3}
 \begin{gathered}
 \Aa_2\oplus\Aa_3\cong (2,1)\oplus (3,2,1)\\\downarrow\\
(4,3,3,3,4,3,2,2)\\\downarrow\\
(5,4,4,4,4,6,5,4,4,4,4)\\\downarrow\\
(7,6,6,6,6,6,6,7,6,5,5,5,5,5)\\\downarrow\\
\cdots
 \end{gathered}
\end{align}  
 
We show how to construct a higher Auslander algebra which is syzygy equivalent to semisimple algebra $\oplus_m\Aa_1$. Its simple modules are indexed by
$\left\{1,2,\ldots,m\right\}$. The defect is $m$ by the set up \ref{defect of linear}, therefore there are $m$ new simple modules such that
$1>x_1>2>x_2>\ldots>m-1>x_{m-1}> m> x_m>1$.
By \ref{lemma delta length nakayama cycle}, the base set is 
\begin{align}\label{friday example}
\left\{\begin{vmatrix}
1\\x_1
\end{vmatrix}, \begin{vmatrix}
2\\x_2
\end{vmatrix},\ldots,\begin{vmatrix}
m-1\\x_{m-1}
\end{vmatrix},\begin{vmatrix}
m\\x_m
\end{vmatrix}\right\}
\end{align}
By the constructions \ref{friday 1}, \ref{friday 2} we get
\begin{align*}
\left\{\begin{vmatrix}
1\\x_1
\end{vmatrix},\begin{vmatrix}
x_1\\2\\x_2
\end{vmatrix}, \begin{vmatrix}
2\\x_2
\end{vmatrix},\begin{vmatrix}
x_2\\3\\x_3
\end{vmatrix}\ldots,\begin{vmatrix}
x_{m-1}\\m\\x_{m}
\end{vmatrix},\begin{vmatrix}
m\\x_m
\end{vmatrix},\begin{vmatrix}
x_{m}\\1\\x_{1}
\end{vmatrix}\right\}
\end{align*}
This is just m-fold covering of cyclic Nakayama algebra given by Kupisch series $(3,2)$. It is the only cyclic Nakayama algebra of global and dominant dimension two. See Proposition \ref{sectionkis2} and \ref{gustafson} for details.
 \end{example}

\begin{example}\label{example nak cycle}
We conclude this part by an example which shows why Nakayama cycle is required for the uniqueness part of the Theorem \ref{thmreverseepsilon}. As algebras $\Aa_2\oplus\Aa_2\oplus\Aa_3\oplus\Aa_3$ and $\Aa_2\oplus\Aa_3\oplus\Aa_2\oplus\Aa_3$ are isomorphic, however they lead different cyclic Nakayama algebras:
\begin{center}
 $\bm\varepsilon (3,2,2,4,3,3,3,5,4,3,3,3,4,3,2,2)\cong \Aa_2\oplus\Aa_2\oplus\Aa_3\oplus\Aa_3$\\$\bm\varepsilon (4,3,3,3,4,3,2,2,4,3,3,3,4,3,2,2)\cong \Aa_2\oplus\Aa_3\oplus\Aa_2\oplus\Aa_3$.
 \end{center}
If we express the projective $\Aa_2\oplus\Aa_2\oplus\Aa_3\oplus\Aa_3$-modules and $\Aa_2\oplus\Aa_3\oplus\Aa_2\oplus\Aa_3$-modules according to the remark \ref{def relations liner}, we get
\begingroup
\allowdisplaybreaks
\begin{gather}
\begin{vmatrix}
1\\2
\end{vmatrix},\begin{vmatrix}
2
\end{vmatrix},\begin{vmatrix}
3\\4
\end{vmatrix},\begin{vmatrix}
4
\end{vmatrix},\begin{vmatrix}
5\\6\\7
\end{vmatrix},\begin{vmatrix}
6\\7
\end{vmatrix},\begin{vmatrix}
7
\end{vmatrix},\begin{vmatrix}
8\\9\\10
\end{vmatrix},\begin{vmatrix}
9\\10
\end{vmatrix},\begin{vmatrix}
10
\end{vmatrix}\\
\begin{vmatrix}
1\\2
\end{vmatrix},\begin{vmatrix}
2
\end{vmatrix},\begin{vmatrix}
3\\4\\5
\end{vmatrix},\begin{vmatrix}
4\\5
\end{vmatrix},\begin{vmatrix}
5
\end{vmatrix},\begin{vmatrix}
6\\7
\end{vmatrix},\begin{vmatrix}
7
\end{vmatrix},\begin{vmatrix}
8\\9\\10
\end{vmatrix},\begin{vmatrix}
9\\10
\end{vmatrix},\begin{vmatrix}
10
\end{vmatrix}
\end{gather}
\endgroup
By using \ref{lemma delta length nakayama cycle} we can detect the lengths of elements of the base sets. We give the lengths of non-simple elements of $\cB(\Lambda')$ and $\cB(\Lambda'')$:
\begin{align*}
\ell\Delta(2)=2,\ell\Delta(4)=3,\ell\Delta(7)=3,\ell\Delta(10)=2\\
\ell\Delta(2)=3,\ell\Delta(5)=2,\ell\Delta(7)=3,\ell\Delta(10)=2
\end{align*}
Therefore $\Lambda'$ and $\Lambda''$ cannot have the same Kupisch series by propositions \ref{defect prop 5} and \ref{defect prop 6} where $\bm\varepsilon(\Lambda')\cong\Aa_2\oplus\Aa_2\oplus\Aa_3\oplus\Aa_3$ and $\bm\varepsilon(\Lambda'')\cong\Aa_2\oplus\Aa_3\oplus\Aa_2\oplus\Aa_3$.
\end{example}

\begin{example}\label{example dominant dimension}
In the remark \ref{remarklis1}, to make the reduction $\domdim\Lambda=\bm\varepsilon(\Lambda)+2$ we assume that $\domdim\Lambda\geq 3$. We want to explain the reason behind it. Dominant dimension of semisimple algebra is infinity.  If there is an injective $\Lambda$-module $I$ satisfying $\pdim I=\domdim I=2$, then we cannot conclude $\bm\varepsilon(\Lambda)$ has dominant dimension zero. Surprisingly, it can take any number. If we relax the equality of dominant dimensions of linear components in the definition of Nakayama cycle (\ref{nakayamacycle}), the below claims of Proposition \ref{defect prop 5} are still true
\begin{itemize}
\item $\Lambda'$ is well-defined, $\bm\varepsilon(\Lambda')\cong \Lambda$.
\item $\defect\Lambda'=\defect\Lambda$
\end{itemize}
Moreover, if we relax the condition on being higher Auslander in Proposition \ref{def prop4} we get $\bm\varepsilon$-equivalent cyclic Nakayama algebras sharing the same defect. In other words, we can construct defect invariant cyclic Nakayama algebras which are $\bm\varepsilon$-equivalent.
We present two examples. Let $\Lambda$ be given by Kupisch series $(4,3,3,4,3,3,4)$. Then $\bm\varepsilon(\Lambda)$ is given by $(3,2,3,2,2)$ and $\bm\varepsilon^2(\Lambda)\cong \Aa_2\oplus\Aa_1$. We have
\begin{align}
\begin{gathered}
\gldim\Lambda=5, \gldim\bm\varepsilon(\Lambda)=3, \gldim\bm\varepsilon^2(\Lambda)=1\\
\domdim\Lambda=4, \domdim\bm\varepsilon(\Lambda)=2, \domdim\bm\varepsilon^2(\Lambda)=1\\
\defect\Lambda=\defect\bm\varepsilon(\Lambda)=\defect\bm\varepsilon^2(\Lambda)=2
\end{gathered}
\end{align}
$\domdim \Aa_2\oplus\Aa_1=1$ because $1<\infty$.
If we apply $\bm\varepsilon$-construction to algebra $\Lambda$ with Kupisch series $(3,3,3,4,3,2,4,3,3,3)$, then $\bm\varepsilon(\Lambda)$ splits into direct sum of $\Aa_1$ and linear algebra with Kupisch series $(3,3,3,3,2,1)$. 
\begin{align}
\begin{gathered}
\gldim\Lambda=5, \domdim\Lambda=2\\
\gldim\bm\varepsilon(\Lambda)=3, \domdim\bm\varepsilon(\Lambda)=3
\end{gathered}
\end{align}
If we apply the construction in the section \ref{section cyclic} to $\Lambda$, we get the sequence of algebras

 \begin{center}
 $\xymatrixcolsep{5pt}\xymatrixrowsep{4pt}
\xymatrix{ (3,3,3,3,2,1)\oplus (1)\ar[d]
\\
(3,3,3,4,3,2,4,3,3,3)\ar[d]\\
(4,4,4,6,6,5,4,4,4,4,3,3,3)\ar[d]\\
(6,6,6,6,6,6,6,5,4,4,4,5,4,4,4,6)\ar[d]\\
\cdots}$
 \end{center}
 sharing the same defect.
\end{example}

\section{Global Dimension}\label{sectionCalculation}
First we will construct some higher Auslander algebras which are linear. Then, we can apply Theorem \ref{thmreverseepsilon2} to generate families of algebras which are $\bm\varepsilon$-equivalent to each other.

\begin{notation} \begin{itemize}
\item $a^b$ denotes the sequence $a,a,\ldots,a$ of length $b$.
\item Finite integer sequences in parenthesis $()$ denotes Kupisch series.
\item $\Aa_n$ is oriented $\Aa$-type Dynkin quiver of rank $n$.
\item To simplify the exposition, by abuse of notation we may use Kupisch series instead of algebra. For instance $\bm\varepsilon^2(4,3,3,3)\cong\bm\varepsilon(3,2,2)\cong (2,1)=\Aa_2$.
\end{itemize}
\end{notation}

\begin{lemma}\label{lemma1} Algebras $\Lambda$ of format $(2^d,1)$ where $d\geq 1$ satisfies:
\begin{itemize}
\item $\rank\Lambda=d+1$
\item $\gldim\Lambda=\domdim\Lambda=d$
\item $\defect\Lambda=1$
\end{itemize}
Moreover, it is the only defect one linear Nakayama algebra.
\end{lemma}
\begin{proof}
It is clear that the number of simple modules is $d+1$, and there is only one injective but non-projective module, simple $S_1$. Its projective dimension is $d+1$. This completes the proof of the first part.

Let $L$ be a linear Nakayama algebra of defect one and rank $n$. Therefore the number of relations and in particular number of classes of projective modules are $n-1$ (see \ref{frfr1}). This means two projective modules have isomorphic socles. Since  $S_n$ is simple projective and $L$ is connected, $P_{n-1}$ is $\begin{vmatrix}
S_{n-1}\\S_n
\end{vmatrix}$. There are $n-2$ socles and $n-2$ projective modules left, so every projective module $P_i$ is projective-injective. By the properties of Kupisch series \ref{kupisch}, $c_i\leq c_{i+1}$ for all $1\leq i\leq n-2 $. Because $c_{n-1}=2$, it follows either each $c_i$ is one or two. If there is $c_i=1$, then $S_{i+1}$ becomes simple injective module, together with $S_1$, it makes the defect at least two which contradicts to assumption on the defect. Therefore each $c_i=2$ for $1\leq i\leq n-1$ and $c_{n}=1$.
\end{proof}

\begin{lemma}[Gustafson's Example]\label{gustafson} Let $\Lambda$ be algebra of format $((n+1)^{n-1},n)$. $\Lambda$ satisfies:
\begin{itemize}
\item  ${\bm\varepsilon}^{n-2}(\Lambda)\cong (3,2)$
\item  ${\bm\varepsilon}^{n-1}(\Lambda)\cong \Aa_1$
\item $\gldim\Lambda=\domdim\Lambda=2n-2$
\end{itemize}
\end{lemma}
\begin{proof} The third item is very well known result \cite{gus}. It can be shown that $\cB(\Lambda)=\left\{ S_2,S_3,\ldots,S_{n-1},M\right\}$ where $M=\begin{vmatrix}
S_n\\S_1
\end{vmatrix}$. Every projective module except $P_1$ has $\cB(\Lambda)$ filtration. Lengths of the remaining projective modules drop by one, because $S_1$ appears exactly once in the composition series of projective modules. By induction, result follows.
\end{proof}

This shows that algebra $(3,2)$ is a special one, because it is the only higher Auslander algebra whose syzygy filtered algebra is $\Aa_1$.

\begin{lemma}\label{lemma defect one} If $\Lambda$ is a cyclic Nakayama algebra which is a higher Auslander algebra of defect one, then either $\bm\varepsilon^j(\Lambda)\cong \Aa_1$ or  $\bm\varepsilon^j(\Lambda)\cong (2^d,1)$ for some $d\geq 1$.
\end{lemma}
\begin{proof}
The only linear Nakayama algebra of defect one has Kupisch series $(2^d,1)$ by lemma \ref{lemma1}. Algebras syzygy equivalent to $\Aa_1$ comes from the Gustafson's example which have Kupisch series $((n+1)^{n-1},n)$.
\end{proof}

\begin{lemma}\label{defect one}
If $k=2n-2$, then there is a unique cyclic Nakayama algebra $\Lambda$ of rank $n$ which is higher Auslander algebra.
If $k$ is odd, then for all $n\leq k\leq 2n-3$ there is a unique cyclic Nakayama algebra $\Lambda$ of rank $n$ which is higher Auslander algebra of defect one and global dimension $k$.
If $k$ is even, then for all $n\leq k\leq 2n-4$ there is a unique cyclic Nakayama algebra $\Lambda$ of rank $n$ which is higher Auslander algebra of defect one and global dimension $k$.
\end{lemma}
\begin{proof} This result appears in the work \cite{rene}.
In all cases, the uniqueness of such $\Lambda$ holds by the lemma \ref{lemma defect one} because those are the only defect one linear algebras. The first claim follows from the Gustafson's example, see lemma  \ref{gustafson}.
$\Lambda$ is cyclic and defect one, it has to satisfy $\bm\varepsilon^j(\Lambda)\cong (2^d,1)$ for some $d\geq 1$. Uniqueness of such $\Lambda$ follows from Theorem \ref{thmreverseepsilon2} part i. Notice that
\begin{align}\begin{gathered}
\rank\Lambda=d+1+j\\
\gldim\Lambda=\domdim\Lambda=d+2j
\end{gathered}
\end{align}
Solutions to $k=d+2j$, $j\geq 1$, $d\geq 1$ leads to
\begin{align}
j+d\in\left\{k-1,k-2,\ldots,\frac{k+1}{2}\right\}
\end{align}
which is an interval. Therefore 
\begin{gather}
\begin{gathered}
\frac{k+1}{2}\leq j+d\leq k-1 \implies\\
\frac{k+3}{2}\leq j+d+1\leq k \implies\\
\frac{k+3}{2}\leq n\leq k
\end{gathered}
\end{gather}
If we regroup inequalities, we get $n\leq k\leq 2n-3$. If $k$ is even, solutions to $k=d+2j$, $j\geq 1$, $d\geq 2$ forms the interval
\begin{align}
j+d\in\left\{k-1,k-2,\ldots,\frac{k+2}{2}\right\}
\end{align}
which is an interval. Therefore 
\begin{gather}\begin{gathered}
\frac{k+2}{2}\leq j+d\leq k-1 \implies\\
\frac{k+4}{2}\leq j+d+1\leq k \implies\\
\frac{k+4}{2}\leq n\leq k
\end{gathered}
\end{gather}
If we regroup inequalities, we get $n\leq k\leq 2n-4$.
\end{proof}

\begin{lemma}\label{lemma2} Algebras $\Lambda$ of format $((2^{d-1},3)^a,2^d,1)$ with $a\geq 1$, $d\geq 2$ satisfies:
\begin{itemize}
\item $\rank\Lambda=ad+d+1$
\item $\gldim\Lambda=\domdim\Lambda=d$
\item $\defect\Lambda=a+1$
\end{itemize}
\end{lemma}

\begin{proof} The number of elements in the Kupisch series is simply $ad+d+1$. Injective non-projective $\Lambda$-modules are quotients of projective modules indexed by $\{1,d,2d,\ldots,ad\}$. Therefore defect is $a+1$. The projective resolution of $S_1$ is
\begin{align*}
0\longrightarrow P_{d+1}\longrightarrow\ldots \longrightarrow P_1\longrightarrow S_1\longrightarrow 0
\end{align*}
where $\Omega^{d}(S_1)\cong P_{d+1}$. For the remaining injective non-projective modules we get 
\begin{align*}
0\longrightarrow P_{td+d+1}\longrightarrow\ldots\longrightarrow P_{td+2}\longrightarrow P_{td}\longrightarrow I_{td+1}\longrightarrow 0
\end{align*}
Therefore, global and dominant dimensions are $d$.
\end{proof}

\begin{proposition} Algebra $\Lambda$ of format $((2^{d-1},3)^{a},2^d,1)^{\delta}\oplus (2^d,1)^{b} $ satisfies:
\begin{itemize}
\item $\gldim\Lambda=\domdim\Lambda=d$
\item $\rank \Lambda=\delta(a d+d+1)+b(d+1)$
\item $\defect\Lambda=\delta(a+1)+b$
\end{itemize} 
\end{proposition}
\begin{proof}
By lemmas \ref{lemma1}, \ref{lemma2}, global and dominant dimensions are $d$. After counting modules upto multiples $\delta$ and $b$, claim follows.
\end{proof}

Actually, $k=2$ corresponds to the original Auslander algebra. We have the following classification.
\begin{proposition} \label{sectionkis2} Fix $k=2$. We have
\begin{itemize}
\item if $n$ is odd, there is no cyclic Nakayama algebra, but there are linear algebras: $n=3$, $(2,2,1)$, $n\geq 5$, $((2,3)^a,2,2,1)$ which are higher Auslander algebras.
\item if $n$ is even then the covering algebra $(3,2)^{\frac{n}{2}}$ is a higher Auslander algebra.
\end{itemize}
\end{proposition}
\begin{proof} The global and dominant dimensions of the given algebras for $n$ is odd calculated in lemma \ref{lemma2}. If there were a cyclic $\Lambda$ of odd rank and global dimension $2$, then $\bm\varepsilon(\Lambda)$ would have to be $\bigoplus_m \Aa_1 $. Its rank and defect are $m$, so $n=2m$ which is not possible. If $n$ is even, the same algebra $\bigoplus_m \Aa_1 $ solves the problem, i.e. $n=2m$ or $m=\frac{n}{2}$.
\end{proof}

\begin{corollary} If a Nakayama algebra is cyclic and its global dimension is $2$,then it has to be covering algebra of $(3,2)$. In particular, rank has to be even.
\end{corollary}
\begin{proof}
Let $\Lambda$ be a cyclic Nakayama algebra with $\gldim\Lambda=2$. It has to be syzygy equivalent to semisimple algebra $\oplus_m\Aa_1$, because $\gldim{\bm\varepsilon}(\Lambda)=0$. Now, if we apply Theorem \ref{thmreverseepsilon2} to $\oplus_m\Aa_1$ (see Example \ref{friday example}) result follows.
\end{proof}

\begin{definition} If $E$ is a set whose elements are positive integers, we define the set $mE+t$ as
\begin{align}
mE+t:=\left\{me+t\,\,\vert\,\, e\in E\right\}
\end{align}
for given integers $m\geq 1$, $t\geq 0$. It is called \emph{translated set} of $E$.
\end{definition}

\begin{theorem}\label{thmnumericalreduction}
There exists at least one \textbf{cyclic, connected} Nakayama algebra of rank $n$ with global and dominant dimensions $k$ such that $k+4\leq n$ if $k$ is odd and $k+2\leq n$ if $k$ is even.
\end{theorem}

\begin{proof} 
By the Theorem \ref{thmreverseepsilon2}, we can assume that there is cyclic Nakayama algebra $\Lambda'$ with global and dominant dimensions $k$ such that ${\bm\varepsilon}^{j}(\Lambda')\cong\cL$. In other words, $j$th level syzygy filtered algebra is linear. First we analyze the case where $\cL$ has Kupisch series
\begin{align*}
((2^{d-1},3)^{a},2^d,1)^{\delta}\oplus (2^d,1)^{b}.
\end{align*}
Therefore:
\begin{gather*} 
\rank\Lambda'= n=\delta(a d+d+1)+b(d+1)+j(\delta(a+1)+b)\\
\gldim\Lambda'=k=d+2j
\end{gather*}
We set $\delta=1$. $n$ can be rewritten as:
\begin{align}\label{equationN}
n=\left(d+j\right)\left(a+1\right)+b\left(d+j+1\right)+1
\end{align}
Since $\Lambda'$ is cyclic, we have $j\geq 1$. By the construction of $\cL$, if $a$ is nonzero, then $d-1\geq 1$, so $d\geq 2$. First we analyze the case $k$ is odd. The possible values of $d+j$ forms the set (indeed an interval)
\begin{align}
X=\left\{k-1,k-2\ldots,\frac{k+5}{2},\frac{k+3}{2}\right\}
\end{align}
which follows from the integer solutions of $1\leq j$, $3\leq d$, $k=d+2j$. 
By \ref{equationN}, we get 
\begin{align}
n=x(a+1)+(x+1)b+1
\end{align}
where $a\geq 1$, $b\geq 0$, $x\in X$. Therefore the union of certain translated sets will provide such $n$. In details, we get the following translated sets
 \begin{align}\label{equ 1}
 \begin{split}
b=0,&\hspace{2cm} 2X+1, 3X+1, 4X+1, 5X+1,\ldots\\
b=1,&\hspace{3.5cm} 3X+2,4X+2,5X+2,\ldots\\
b=2,&\hspace{5cm} 4X+3,5X+3,\ldots\\
\vdots&\hspace{6cm}\vdots
\end{split}
\end{align}
If we set $\delta=0$ and $a=0$, then $n$ can take values $(x+1)(b+1)$ where $x\in X$. In details
\begin{align}\label{equ 55}
2X+2,3X+3,4X+4,5X+5,\ldots
\end{align}
are possible translated sets. Therefore any $mX+t$, $m\geq 1$, $m\geq t\geq 1$ can be obtained if we combine \ref{equ 1} and \ref{equ 55}. For a fixed $m$, the union of all translated sets $mX+t$, $1\leq t\leq m$ is an interval. Namely:
\begin{align} \label{equ 2}
\begin{gathered}
2X+1 \cup 2X+2=\left\{k+4,k+5,\ldots,2k\right\}\\
3X+1\cup 3X+2 \cup 3X+3=\left\{3\frac{k+3}{2}+1,\ldots,3k\right\}\\
4X+1\cup 4X+2 \cup 4X+3\cup 4X+4=\left\{2k+7,\ldots 4k\right\}\\
\vdots
\end{gathered}
\end{align}
In general, we get
\begin{align}
\bigcup\limits^{m}_{t=1} mX+t=\left\{m\frac{k+3}{2}+1,\ldots,mk\right\}.
\end{align}
Union of these sets covers every integer $n\geq k+4 $ if the infimum of $(m+1)X+1$ is smaller than the supremum of $mX$  where $m\geq 2$. We get:
\begin{align}\label{equ3}
\begin{gathered}
(m+1)\frac{k+3}{2}+1\leq mk\implies\\
3m+5\leq k(m-1)\implies\\
\frac{3m+5}{m-1}\leq k
\end{gathered}
\end{align}
Since the left hand side is decreasing function when $m\geq 2$, its maximum occurs at $m=2$, therefore for all $k\geq 11$, the union covers all integers $n\geq k+4$. As a result we get
\begin{align}
n\in\bigcup\limits_{m\geq 2}\left(\bigcup\limits^{m}_{t=1} mX+t\right)=\Zz_{\geq k+4}\quad\text{if}\,\,k\geq 11.
\end{align}

 The remaining possible values of $k$ are $3,5,7,9$.
 If $k=9$, $X=\left\{6,7,8\right\}$. By \ref{equ3}, if $m\geq 3$, the union of all translated sets $mX+t$ is $\Zz_{\geq 19}$. Therefore
 \begin{align}
 \begin{gathered}
 2X+1\cup 2X+2\cup \bigcup\limits_{m\geq 3}\left(\bigcup\limits^{m}_{t=1} mX+t\right)= \\
 \left\{13,14,15,16,17,18\right\}\cup \Zz_{\geq 19}=\Zz_{\geq 13}
 \end{gathered}
 \end{align}
 Hence for $k=9$, all numbers starting at $13$ is in the union of translated sets.\\
 
If $k=7$ then $X=\left\{5,6\right\}$ and the union 
$\bigcup\limits_{m\geq 4}\left(\bigcup\limits^{m}_{t=1} mX+t\right)$ gives $\Zz_{\geq 22}$ when $m\geq 4$. For $m=2$ and $m=3$ we get the following translated sets $\left\{11,12,13,14\right\}$ and $\left\{16,\ldots,21\right\}$. Therefore the union of translated sets covers all $\Zz_{\geq 11}$ but $15$. The following $3$-fold covering of $(5,4,4,4,4)$ provides $n=15$, $k=7$ algebra, which is $\bm\varepsilon$-equivalent to $3$-fold covering of $\Aa_2$, i.e. $\delta=0$ and $d=1$ in \ref{equationN}.\\

If $k=5$ then $X=\left\{4\right\}$. The union can take values
\begin{align}
\left\{9,10\right\}\cup\left\{13,14,15\right\}\cup \left\{17,18,19,20\right\}\cup\Zz_{\geq 21}.
\end{align}
Therefore the missing ranks are $11,12,16$.
For $n=11$, $(5,4,4,4,4,6,5,4,4,4,4)$ is the Kupisch series, which is equivalent to $\Aa_3\oplus\Aa_2$ (see example \ref{exampe a2+a3}).
$n=12$, the algebra $(4,3,3,3)$ is of global and dominant dimension $5$. Its threefold covering $(4,3,3,3,4,3,3,3,4,3,3,3)$ is what we want. For
$n=16$, $4$-fold covering of $(4,3,3,3)$ provides what we want. Notice that $(4,3,3,3)$ is $\bm\varepsilon$-equivalent to $\Aa_2$.\\

If $k=3$, then $j=1$ and $d=1$. Therefore $\bm\varepsilon(\Lambda')\cong \Aa_n$ where $n$ can take odd numbers $3,5,7,9,11,\ldots$. For $n=8$, we take $(4,3,3,3,4,3,2,2)$ which is $\bm\varepsilon$-equivalent to $\Aa_3\oplus\Aa_2$. Any even number greater than $8$, algebras $\bm\varepsilon$-equivalent to $\Aa_n\oplus\Aa_2$ solve the problem.\\

Now we analyze the case $k$ is even. The possible values of $d+j$ forms the set (indeed an interval)
\begin{align}
X=\left\{k-1,k-2\ldots,\frac{k+4}{2},\frac{k+2}{2}\right\}
\end{align}
which follows from the integer solutions of $1\leq j$, $2\leq d$, $k=d+2j$. 
By \ref{equationN}, we get 
\begin{align}
n=x(a+1)+(x+1)b+1
\end{align}
where $a\geq 1$, $b\geq 0$, $x\in X$. Therefore the union of certain translated sets will provide such $n$. In details,
 \begin{align}\label{equ 52}
 \begin{split}
b=0,&\hspace{2cm} 2X+1, 3X+1, 4X+1, 5X+1,\ldots\\
b=1,&\hspace{3.5cm} 3X+2,4X+2,5X+2,\ldots\\
b=2,&\hspace{5cm} 4X+3,5X+3,\ldots\\
\vdots&\hspace{6cm}\vdots
\end{split}
\end{align}
If we set $\delta=0$ and $a=0$, then $n$ can take values $(x+1)(b+1)$ where $x\in X$. In details
\begin{align}\label{equ 51}
2X+2,3X+3,4X+4,5X+5,\ldots
\end{align}
are possible values. Therefore any $mX+t$, $m\geq 1$, $m>t\geq 1$ can be obtained if we combine \ref{equ 51} and \ref{equ 52}. For a fixed $m$, the union of all translated sets $mX+t$, $1\leq t\leq m$ is an interval. Namely:
\begingroup
\allowdisplaybreaks
\begin{align}
\begin{gathered}
2X+1 \cup 2X+2=\left\{k+2,k+3,\ldots,2k\right\}\\
3X+1\cup 3X+2 \cup 3X+3=\left\{3\frac{k+2}{2}+1,\ldots,3k\right\}\\
4X+1\cup 4X+2 \cup 4X+3\cup 4X+4=\left\{2k+5,\ldots 4k\right\}\\
\vdots
\end{gathered}
\end{align}
\endgroup
In general, we get
\begingroup
\allowdisplaybreaks
\begin{align}
\bigcup\limits^{m}_{t=1} mX+t=\left\{m\frac{k+2}{2}+1,\ldots,mk\right\}
\end{align}
\endgroup
Union of these sets covers every integer $n\geq k+2 $ if the infimum of $(m+1)X+1$ is smaller than the supremum of $mX$  where $m\geq 2$. This implies
\begin{align}\label{ccc ddd}\begin{gathered}
(m+1)\frac{k+2}{2}+1\leq mk\implies\\
2m+4\leq k(m-1)\implies\\
\frac{2m+4}{m-1}\leq k.
\end{gathered}
\end{align}
Since the left hand side is decreasing function when $m\geq 2$, its maximum occurs at $m=2$, therefore for all $k\geq 8$, the union covers all integers $n\geq k+2$. As a result we get
\begin{align}
n\in\bigcup\limits_{m\geq 2}\left(\bigcup\limits^{m}_{t=1} mX+t\right)=\Zz_{\geq k+2}\quad\text{if}\,\, k\geq 8.
\end{align}

 The remaining possible values of $k$ are $2,4,6$.
 If $k=6$, $X=\left\{4,5\right\}$. By \ref{ccc ddd}, if $m\geq 3$, the union of all translated sets $mX+t$ is $\Zz_{\geq 13}$. Therefore
 \begin{gather}
 2X+1\cup 2X+2\cup \bigcup\limits_{m\geq 3}\left(\bigcup\limits^{m}_{t=1} mX+t\right)= \\
 \left\{9,10,11,12\right\}\cup \Zz_{\geq 13}= \Zz_{\geq 9}
  \end{gather}
 Hence for $k=6$, all numbers starting at $9$ is in the union of translated sets. For $n=8$, double cover of $(5,5,5,4)$ solves it.\\
 If $k=4$ then $X=\left\{3\right\}$.  If $m\geq 4$, the union of all translated sets $mX+t$ is $\Zz_{\geq 13}$. We get
 \begin{align}\begin{gathered}
 2X+1\cup 2X+2\cup \bigcup\limits_{m=3} mX+t \bigcup\limits_{m\geq 4}\left(\bigcup\limits^{m}_{t=1} mX+t\right)= \\
 \left\{7,8\right\}\cup\left\{10,11,12\right\}\cup \Zz_{\geq 13}.\end{gathered}
 \end{align}
 Therefore the union covers all $\Zz_{\geq 7}$ but $9$.
 If $n=9$, then the threefold covering of $(4,4,3)$ is what we are looking for. For $n=6$, the double covering of $(4,4,3)$ handles it. The case $k=2$ is analyzed in \ref{sectionkis2}.
\end{proof}

\begin{proposition}\label{thm k odd}
If $k$ is odd, there are cyclic Nakayama algebras of rank $k+2$ and $k+3$ which are higher Auslander algebras.
\end{proposition}
\begin{proof}
For the algebra $\Aa_3$, we have
\begin{align}
\rank\Aa_3=3, \defect\Aa_3=2, \gldim\Aa_3=1.
\end{align}
Therefore if $\Lambda$ is a cyclic Nakayama algebra which is a higher Auslander algebra satisfying $\bm\varepsilon^j(\Lambda)\cong \Aa_3$, then we have
\begingroup
\allowdisplaybreaks
\begin{align}\begin{gathered}
\rank\Lambda=3+2j\\
\gldim\Lambda=\domdim\Lambda=1+2j.
\end{gathered}
\end{align}
\endgroup
Therefore if $k=1+2j$, then $n=k+2$. It is easy to describe their Kupish series which are $(n,n-1,(n-2)^{n-2})$.\\
For the algebra $\Aa_2$, we have
\begin{align}
\rank\Aa_2=2, \defect\Aa_2=1, \gldim\Aa_2=1.
\end{align}
Therefore if $\Lambda$ is a cyclic Nakayama algebra which is a higher Auslander algebra satisfying $\bm\varepsilon^j(\Lambda)\cong \Aa_2$, then we have
\begin{align}\begin{gathered}
\rank\Lambda=2+j\\
\gldim\Lambda=\domdim\Lambda=1+2j
\end{gathered}
\end{align}
Therefore the rank of the double covering of $\Lambda$ is $n=4+2j$, and its dominant dimension is $k=1+2j$. So for all $n=k+3$, there is a higher Auslander algebra. It is easy to describe their Kupish series which are $(n,(n-1)^{n-1},n,(n-1)^{n-1})$.
\end{proof}

\subsection{Algebras of rank ${\bm{k+1}}$ and ${\bm{\gldim k}}$ }

First we prove lemma:

\begin{lemma}\label{lemmalinear} Let $L$ be linear Nakayama algebra of rank $n$. Then:
\begin{enumerate}[label=\roman*)]
\item $\gldim L\leq n-1$
\item $\gldim L=n-1 \iff L=(2^{n-1},1)$
\end{enumerate}
\end{lemma}

\begin{proof}
$L$ is linear, there are $n-1$ projective non-simple modules. Therefore longest projective resolution can have at most $n-1$ projective modules, which gives the claimed upper bound. To show  the second part: we can assume that there is at least one projective module with the length $\geq 3$. Without loss of generality, let it be $P_1$. The projective resolution of injective module which is the subquotient of $P_1$, can have at most $n-2$ projective modules, because the projective covers of the socle of the injective module cannot appear in the resolution. We conclude that the global dimension is $n-1$ if and only if there is no projective module of length longer than $2$. 
\end{proof}
Now, we show why $n-1$ is not in the spectrum \ref{defspec}.
\begin{proposition}\label{thmlinearNak} Nakayama algebra $\Lambda$ of rank $n$ and global dimension $k$ which is higher Auslander algebra where $n=k+1$ has to be linear Nakayama algebra $(2^k,1)$.
\end{proposition}

\begin{proof}
Assume to the contrary that let $\Lambda$ be a cyclic Nakayama algebra of rank  $n$, global dimension $k$ such that $n-k=1$. By remark \ref{remarklis1}, there is $m$ such that $L={\bm\varepsilon}^{m+1}(\Lambda)$ is linear. By the results \ref{thmLamdaauslanderimplies}, \ref{thmreverseepsilon}, the defect is invariant and the reductions given in Remark \ref{remarklis2} are two, we get:
\begin{itemize}
\item Global dimension of $L$ is $k-2(m+1)$
\item Rank of $L$ is  $n-(m+1)\defect\Lambda$
\end{itemize}
Global dimension of linear Nakayama algebra $L$ is bounded by $\rank L-1$ by lemma \ref{lemmalinear}. Hence
\begin{gather*}
k-2(m+1)\leq n-(m+1)\defect\Lambda-1\implies\\
\defect\Lambda\leq 2.
\end{gather*}
The defect of $\Lambda$ cannot be one, because in \ref{defect one}, we showed that $\defect\Lambda=1$ implies $\rank\Lambda\leq k$. Therefore the defect of $\Lambda$ and in particular the defect of $L$ has to be two. However, defect two forces that $\rank L=n-2(m+1)$, which makes $\rank L=\gldim L+1$, and by lemma \ref{lemmalinear} $L$ has to be $(2^d,1)$. But this is not possible because the defect of that algebra is not two but one. As a result, there is no cyclic $\Lambda$ satisfying $\rank\Lambda=\gldim\Lambda+1=\domdim\Lambda+1$.
\end{proof}

\subsection{Proof of the Theorem \ref{thmsolutiontoproblem} and the Conjecture}

Here we give the solution to the problem \ref{problem}.
\begin{proof}
By proposition \ref{thmlinearNak}, there is no cyclic Nakayama algebra of rank $k+1$ and global dimension $k$ which is higher Auslander algebra. By the proposition \ref{sectionkis2}, if $k=2$, there is no odd rank cyclic Nakayama algebra. It can take all the remaining values by the results \ref{thmnumericalreduction}, \ref{gustafson}, \ref{thm k odd}. Therefore $\zeta(Q)=\left\{3,\ldots,2n-2\right\}\setminus \left\{n-1\right\}$ if $n$ is odd. If the rank is even, by the results \ref{thmnumericalreduction}, \ref{gustafson}, together with \ref{sectionkis2} we get $\zeta(Q)=\left\{2,\ldots,2n-2\right\}\setminus \left\{n-1\right\}$. To verify the conjecture, we can just add linear algebras discussed in \ref{lemmalinear}, and \ref{sectionkis2}. This solves the conjecture. 
\end{proof}

\section{Appendix}\label{sectionappendix}
By the Theorem \ref{thmreverseepsilon}, finding all cyclic Nakayama algebras which are higher Auslander is equivalent to determining all linear Nakayama algebras which are higher Auslander. By varying the underlying linear quiver, we get different classes of algebras. Now we apply the Theorem to linear Nakayama algebras of format $(n^{an},n,n-1,\ldots,1)$.

\begin{proposition}\label{prop 3.2}
Nakayama algebra $\Lambda$ of format $(n^{\alpha n},n,n-1,\ldots,1)$ satisfies $\gldim\Lambda=\domdim\Lambda=2\alpha+1$.
\end{proposition}
\begin{proof}
The second syzygy of any injective non-projective module $I_a$ can be viewed injective non-projective module of algebra $(n^{(\alpha-1)n},n,n-1,\ldots,1)$, by Ringel's simplification process \cite{rin1}. Claim follows by induction, the last algebra is of type $\Aa$.
\end{proof}

\begin{definition}\label{defbra}
To present Kupisch series without cumbersome indices we define the bracket: $[X,Y]:=(X,X-1,\ldots,Y+1,Y^Y)$ where $X>Y$.
\end{definition}

\begin{proposition}\label{prop 4.1} Let $X=(j+1)n-j$, $Y=jn-(j-1)$. Nakayama algebra of format $[X,Y]$ (see definition \ref{defbra}) satisfies:
\begin{enumerate}[label=\roman*)]
\item $\bm{\varepsilon}(\Lambda)=[Y,Y-n+1]$
\item $\bm{\varepsilon}^j(\Lambda)\cong \Aa_n$
\item $\gldim\Lambda=2j+1$
\item $\domdim\Lambda=2j+1$
\end{enumerate}
\end{proposition}
\begin{proof}
Easiest way to prove the first item is based on relations and the computation of the base set \ref{defbase}. We also include description of projectives with respect to their socles \cite{sen18}:
Relations are: (see section \ref{sectionprelim})
\begin{gather*}
\alpha_{X}\cdots \alpha_{n}=0\\
\alpha_{1}\cdots \alpha_{n+1}=0\\
\vdots\\
\alpha_{Y-1}\cdots \alpha_{X}=0.
\end{gather*}

In particular, classes of projective modules are:
\begin{gather}\label{classesofprojectives}
P_{n}\hookrightarrow P_{n-1}\hookrightarrow\ldots\hookrightarrow  P_{2}\hookrightarrow  P_{1} \quad\text{ have simple } S_{X} \text{  as their socle}\\
 P_{n+1} \nonumber \quad\text{ has simple } S_{1} \text{  as its socle}\\\nonumber
\vdots\qquad\qquad\qquad\qquad\qquad\qquad\qquad\\\nonumber
P_{X}\quad  \text{ has simple } S_{Y-1} \text{  as its socle}\nonumber
\end{gather}

The base set is $\cB(\Lambda)=\{S_1,S_2,\ldots,S_{Y-1},M\}$ where $M=\begin{vmatrix}
S_Y\\
S_{Y+1}\\
\vdots\\
S_X
\end{vmatrix}$. All projective modules indexed by numbers $1,2,\ldots,Y-1,Y$ have $\bm{\varepsilon}$ filtration, their $\bm{\varepsilon}(\Lambda)$ length reduces by $X-Y$ which is $n-1$. Projective $\Lambda$-modules indexed by $Y+1,\ldots,X$ are eliminated, hence $\bm{\varepsilon}(\Lambda)$ algebra has $Y$ indecomposable projective modules. Lengths of those projectives forms Kupisch series $[Y,Y-n+1]$, the first item is done.

We can apply the proof above recursively to reach $\Aa_n$.
Completing one step reduces $j$ by $1$, recall that $Y=jn-(j-1)$. Because $\Aa_n$ is simply the case corresponds to $j=1$ in $X=(j+1)n-j$, we get $\bm{\varepsilon}^j(\Lambda)\cong\Aa_n$. Also, this shows that global dimension of $\Lambda$ is $2j+1$.

Now, we check dominant dimension of $\Lambda$ by using syzygy filtrations. $\Aa_n$ is a hereditary algebra hence its dominant dimension is one, which is the first step of induction. Assume that $\bm{\varepsilon}(\Lambda)$ is higher Auslander algebra. Injective modules of $\Lambda$ we need to check are $I_Y,\ldots, I_{X-1}$. We have the resolution:
\begin{align}
 P_{X\dotplus 1}\longrightarrow P_{z\dotplus 1}\longrightarrow P_1\longrightarrow I_z\longrightarrow 0
\end{align}
where $Y-1\leq z\leq X-1$. $z\dotplus 1$ is larger than $Y$ and in particular larger than $n$, therefore $P_{z\dotplus 1}$ is projective injective module. Projective cover of the second syzygy is $P_{X\dotplus 1}$ which is cyclically $P_1$. Therefore we get reduction:
\begin{align}
\domdim I_z=2+\domdim\Omega^2(I_z)
\end{align} 
By induction, $\Lambda$ is a higher Auslander algebra of global dimension $2j+1$.
\end{proof}

\begin{proposition}\label{prop 4.2}
Let $X=(j+1)n-j$, $Y=jn-j+1$.
Nakayama algebra $\Lambda$ of format $(X^{\alpha X},[X,Y],Y^{\beta Y})$ satisfies:
\begin{enumerate}[label=\roman*)]
\item If $\alpha\geq 1$ then  $\bm{\varepsilon}(\Lambda)\cong (X^{(\alpha-1) X},[X,Y],Y^{(\beta+1) Y})$
\item If $\alpha=0$, then  $\bm{\varepsilon}(\Lambda)\cong (\bm{\varepsilon}([X,Y]),Y^{\beta Y})$
\item If $\bm{\varepsilon}(\Lambda)$ is higher Auslander algebra, then $\Lambda$ is.
\end{enumerate}
\end{proposition}

\begin{proof}
First we need to find the base set of the given algebra. It includes all simple modules except $\{S_Y,\ldots S_X\}$. Indeed, as in the proof of previous proposition the only module of length greater than one in $\cB(\Lambda)$ is $M=\begin{vmatrix}
S_Y\\
S_{Y+1}\\
\vdots\\
S_X
\end{vmatrix}$.
Now, syzygy filtration reduces length of projective modules containing $M$ as subquotient by $n-1$, there are exactly $Y$ many of them. Furthermore, those modules are of length $X$, therefore we are shorting modules from the $(X^{\alpha X})$ part. Remaining modules have $\bm\varepsilon$-filtration by the simple modules of the base set $\cB(\Lambda)$. Therefore, they preserve their length. This finishes the proof of the first item.

If $\alpha=0$, base set has similar structure, all simple modules except simple composition factors of module $M$. By proposition \ref{prop 4.1}, this gives required result, since remaining modules cannot have $M$ as subquotient.

The last item follows easily, since injective modules indexed by $I_{Y},\ldots,I_{X-1}$ satisfies:
\begin{align}
\domdim I_z=2+\domdim \Omega^2(I_z)
\end{align}
where $I_z$ runs through $I_Y$ to $I_{X-1}$. By induction, $\Lambda$ is higher Auslander algebra.
\end{proof}

\begin{example}
We describe some classes of higher Auslander algebras with odd global dimensions.
\begin{itemize}
\item $k=5$
\begin{enumerate}[label=\arabic*)]
\item $[3n-2,2n-1]=(3n-2,\ldots,2n-1,(2n-1)^{2n-1})$
\item  $([2n-1,n],n^n)=(2n-1,\ldots,n^{2n})$
\item $(n^{2n},n,\ldots,1)$.
\end{enumerate}

\item $k=7$
\begin{enumerate}[label=\arabic*)]
\item $[4n-3,3n-2]$
\item $([2n-1,n],n^{2n})=(2n-1,\ldots,n^{3n})$
\item $((2n-1)^{2n-1},[2n-1,n])$
\item $(n^{3n},n,\ldots,1)$
\end{enumerate}
\item $k=9$
\begin{enumerate}[label=\arabic*)]
\item  $[5n-4,4n-3]$
\item  $([3n-2,2n-1],(2n-1)^{2n-1})$
\item $([2n-1,n],n^{3n})$
\item  $((2n-1)^{2n-1},[2n-1,n],n^n)$
\item  $(n^{4n},n,\ldots,1)$
\end{enumerate}
\end{itemize}
\end{example}

\begin{example}
We give some higher Auslander algebras of small rank with even global dimensions.
\begin{itemize}
\item $k=4$, ranks are in the parentheses.
\begin{enumerate}[label=(\arabic*)]
  \setcounter{enumi}{2}
\item $(4,4,3)$
\item $(3,2,2,2)$
\item $(2,2,2,2,1)$
\item $(4,4,3,4,4,3)$
\item $(3,3,3,3,2,3,2)$
\item $(3,2,2,2,3,2,2,2)$
\item $(4,4,3,4,4,3,4,4,3)$
\item $(3,3,4,4,3,3,3,2,3,2)$
\item $(2,2,3,2,3,3,3,3,2,3,2)$
\item $(3,2,2,2,3,2,2,2,3,2,2,2)$
\item $(3,3,4,4,3,4,4,3,3,3,2,3,2)$
\item $(3,3,3,3,2,3,2,3,3,3,3,2,3,2)$
\item $(2,2,3,2,2,2,3,2,3,3,3,3,2,3,2)$
\end{enumerate}
\item $k=6$, ranks are in the parentheses.
\begin{enumerate}[label=(\arabic*)]
  \setcounter{enumi}{3}
  \item $(5,5,5,4)$
  \item $(3,3,3,3,2)$
  \item $(3,2,2,2,2,2)$
  \item $(2,2,2,2,2,2,1)$
  \item $(5,5,5,4,5,5,5,4)$
  \item $(4,4,3,3,3,3,4,4,3)$
  \item $(3,3,3,3,2,3,3,3,3,2)$
  \item $(3,3,3,3,2,2,2,3,2,2,2)$
  \item $(3,2,2,2,2,2,3,2,2,2,2,2)$
  \item $(4,4,4,4,4,4,3,4,4,3,4,4,3)$
  \item $(4,4,3,3,3,3,4,4,3,3,3,2,3,3)$
  \item $(3,3,3,3,2,3,3,3,3,2,3,3,3,3,2)$
  \end{enumerate}
\end{itemize}

\end{example}

\bibliographystyle{alpha}

\begin{thebibliography}{99}

\bibitem[ARS]{aus} 
M.Auslander, I. Reiten, S. Smalo.
\newblock Representation theory of Artin algebras
\newblock {\em , Cambridge Studies in Advanced Mathematics, 1995}.


\bibitem[Gus85]{gus}
W. Gustafson. Global dimension in serial rings. {\em Journal of Algebra 97.1, 14-16 1985}

\bibitem[IT05]{it}
K. Igusa, G. Todorov. On the finitistic global dimension conjecture for artin algebras. {\em Representations of algebras and related topics, 201–204. Fields Inst. Commun., 45, Amer. Math. Soc., Province, RI, 2005}


\bibitem[Iyama07]{iyama} 
O. Iyama.
\newblock Auslander correspondence.
\newblock {\em Advances in Mathematics, 2007}.

\bibitem[MMZ20]{rene} 
D. Madsen,R. Marczinzik, G. Zaimi.
\newblock On the classification of higher Auslander algebras for Nakayama algebras.
\newblock {\em Journal of Algebra, 2020}.

\bibitem[Ringel76]{rin1} 
C.M. Ringel.
\newblock Representations of {$K$}-species and bimodules.
\newblock {\em Journal of Algebra, 1976}.

\bibitem[Ringel20]{rin2} 
CM. Ringel.
\newblock The finitistic dimension of a Nakayama algebra
\newblock {\em Journal of Algebra, 2021
}.



\bibitem[Sen18]{sen18}
E. Sen.
\newblock The $\varphi$-dimension of cyclic nakayama algebras.
\newblock Communications in Algebra 49 (6), 2278-2299, 2021

\bibitem[Sen19]{sen19}
E. Sen.
\newblock Syzygy filtrations of cyclic nakayama algebras.
\newblock {\em arXiv preprint arXiv:1903.04645}, 2019.


\bibitem[Sen20]{sen20} 
E. Sen.
\newblock Delooping level of Nakayama algebras
\newblock {Archiv der Mathematik 10.1007/s00013-021-01622-z 2021}


\bibitem[STZ]{stz} 
E. Sen, G. Todorov, S. Zhu.
\newblock A note on Auslander-Gorenstein Algebras
\newblock {\emph{preprint}}

\end{thebibliography}

\end{document}